\documentclass[12pt]{amsart}
\usepackage{array}
\usepackage{amsfonts}
\usepackage{amscd}
\usepackage{amssymb}
\usepackage{amsthm}
\usepackage{amsmath}
\usepackage{stmaryrd}
\usepackage{graphicx}
\usepackage{color}
\usepackage{verbatim}
\usepackage{mathrsfs}
\usepackage{tikz}
\usepackage{multirow}
\usetikzlibrary{calc}
\usepackage{calrsfs}
\usepackage{caption}
\usepackage[neveradjust]{paralist}

\parskip\medskipamount

\colorlet{ggrey}{black!50}

\usepackage{hyperref}
\usepackage[capitalize]{cleveref}

 \theoremstyle{plain}
 \newtheorem{thm1}{Theorem}
 \newtheorem{cor1}[thm1]{Corollary}
\newtheorem{thm}{Theorem}[section]

\newtheorem{lemma}[thm]{Lemma}
\newtheorem{prop}[thm]{Proposition}
\newtheorem{cor}[thm]{Corollary}

\theoremstyle{definition}
\newtheorem{defn}[thm]{Definition}
\newtheorem{remark}[thm]{Remark}

\newtheorem{example}[thm]{Example}

\numberwithin{equation}{section}

\def\sA{\mathsf{A}}
\def\sB{\mathsf{B}}

\def\sD{\mathsf{D}}
\def\sE{\mathsf{E}}
\def\sF{\mathsf{F}}

\def\cA{\mathcal{A}}

\def\cC{\mathcal{C}}

\def\cL{\mathcal{L}}

\def \cP{\mathcal{P}}

\def\cS{\mathcal{S}}

\def\wE{\widetilde{E}}

\def\bC{\mathbf{C}}

\def\AA{\mathbb{A}}

\def\FF{\mathbb{F}}

\def\HH{\mathbb{H}}

\def\KK{\mathbb{K}}
\def\LL{\mathbb{L}}

\def\K{\mathbb{K}}

\def\GQ{\mathsf{GQ}}

\def\uniclass{uniclass}

\def\UNICLASS{UNICLASS}

\DeclareMathOperator\OppRk{\mathsf{rk_{o}}}
\DeclareMathOperator\FixRk{\mathsf{rk_f}}
\DeclareMathOperator\Orb{\mathsf{Orb}}

\DeclareMathOperator\Cl{\mathsf{Cl}}
\DeclareMathOperator\Ch{\mathsf{Ch}}

\DeclareMathOperator\Res{\mathsf{Res}}
\DeclareMathOperator\URes{\mathsf{^URes}}
\DeclareMathOperator\LRes{\mathsf{^LRes}}
\DeclareMathOperator\proj{\mathsf{proj}}

\DeclareMathOperator\Opp{\mathsf{Opp}}
\DeclareMathOperator\Fix{\mathsf{Fix}}

\DeclareMathOperator\Diag{\mathsf{Diag}}
\def\Aut{\mathsf{Aut}}
\DeclareMathOperator\disp{\mathsf{Disp}}

\def\id{\mathsf{id}}
\def\<{\langle}
\def\>{\rangle}

\makeatletter
\renewcommand{\@makefnmark}{\mbox{\textsuperscript{}}}
\makeatother

\title{\UNICLASS\ automorphisms of spherical buildings}
\author{Yannick Neyt \and James Parkinson 
\and
Hendrik Van Maldeghem}
\date{\today}

\begin{document}

\maketitle

\begin{abstract} An automorphism of a building is called \textit{uniclass} if the Weyl distance between any chamber and its image lies in a single (twisted) conjugacy class of the Coxeter group. In this paper we characterise uniclass automorphisms of spherical buildings in terms of their fixed structure. For this purpose we introduce the notion of a \textit{Weyl substructure} in a spherical building. We also link uniclass automorphisms to the Freudenthal--Tits magic square.
\end{abstract}


\section*{Introduction}

In this paper we study the \textit{displacement spectra}
$$
\disp(\theta)=\{\delta(C,C^{\theta})\mid C\in\Delta\}
$$
of an automorphism $\theta$ of a building $(\Delta,\delta)$ with Coxeter system~$(W,S)$. More specifically we study the situation where $\disp(\theta)$ is as small as possible in the sense that $\disp(\theta)$ is contained in a single $\sigma$-conjugacy class of~$W$ (where $\sigma$ is the automorphism of the Coxeter graph $\Pi$ of $(W,S)$ given by $\delta(C,D)=s$ if and only if $\delta(C^{\theta},D^{\theta})=s^{\sigma}$, for $s\in S$). Automorphisms with this property are called \textit{uniclass automorphisms}, and it turns out that if $\theta$ is uniclass then $\disp(\theta)$ is necessarily a full $\sigma$-conjugacy class. 

The main result of this paper is a classification of uniclass automorphisms for spherical buildings. This classification is via the geometry fixed by the automorphism, and shows that the uniclass property for spherical buildings is intimately connected to having a large and highly structured fixed geometry. Our main theorem is as follows (recall that an automorphism of a spherical building is \textit{anisotropic} if it maps all chambers to opposite chambers). 

\begin{thm1}\label{thm:main1}
Let $\theta$ be a nontrivial automorphism of a thick irreducible spherical building~$\Delta$ of rank at least~$2$. Then $\theta$ is uniclass if and only if $\theta$ is either anisotropic, or:
\begin{compactenum}[$(1)$]
\item $\Delta$ has type $\mathsf{I}_2(2m)$ $(m\geq 2)$ and in the associated generalised $2m$-gon
\begin{compactenum}[$(a)$]
\item $\theta$ is a collineation that elementwise fixes an ovoid or a spread, or
\item $\theta$ is a polarity (and then necessarily its fixed element structure is an ovoid-spread pairing).
\end{compactenum}
\item $\Delta$ has type $\sA_{2n+1}$ $(n\geq 1)$ and in the associated projective space
\begin{compactenum}[$(a)$]
\item $\theta$ is a fix point free collineation fixing a line spread elementwise, or
\item $\theta$ is a symplectic polarity (a polarity fixing a symplectic polar space of rank~$n$). 
\end{compactenum}
\item $\Delta$ has (Coxeter) type $\sB_n$ $(n\geq 3)$ or $\sD_n$ $(n\geq 4)$ and in the associated polar space
\begin{compactenum}[$(a)$]
\item $\theta$ is a collineation whose fixed points form an ideal subspace, or
\item $\theta$ is a fix point free collineation fixing a line spread elementwise.
\end{compactenum}
\item $\Delta=\sE_6(\K)$ with $\K$ a field and 
\begin{compactenum}[$(a)$]
\item $\theta$ is a symplectic polarity (a polarity fixing a standard split metasymplectic space), or
\item $\theta$ is a collineation fixing an ideal Veronesian pointwise in $\sE_{6,1}(\KK)$. 
\end{compactenum}
\item $\Delta=\sE_7(\K)$ and
\begin{compactenum}[$(a)$]
\item the fixed point structure of $\theta$ in $\sE_{7,1}(\K)$ is a fully embedded metasymplectic space $\sF_4(\K,\LL)$ with $\LL$ a quadratic extension of $\K$, isometrically embedded as a long root subgroup geometry, or
\item the fixed point structure of $\theta$ in $\sE_{7,7}(\K)$ is an ideal dual polar Veronesian.
\end{compactenum}
\item $\Delta=\sE_8(\K)$ and the fixed point structure of $\theta$ in $\sE_{8,8}(\K)$ is a fully (and automatically isometrically) embedded metasymplectic space $\sF_4(\K,\mathbb{H})$ with $\mathbb{H}$ either a quaternion algebra over $\K$ or an inseparable quadratic field extension of degree~$4$.
\item $\Delta$ has type $\sF_4$ and 
\begin{compactenum}[$(a)$]
\item $\theta$ is type preserving and the fixed element structure of $\theta$ is an ideal quadrangular Veronesian, or
\item $\theta$ is a polarity (and then necessarily its fixed element structure forms a Moufang octagon).
\end{compactenum}
\end{compactenum}
Moreover, for each uniclass automorphism the twisted conjugacy class $\disp(\theta)$ is explicitly determined (see Table~\ref{AllWS}). 
\end{thm1}

The definition of the various fixed structures in the theorem (ovoids, spreads, line spreads, ideal subspace, standard split metasymplectic space, ideal Veronesian, and so on) will be given in the relevant subsections of Section~\ref{sec:4}. We call these fixed structures \textit{Weyl substructures} in $\Delta$. These Weyl substructures are large and highly structured subsets of the (simplicial) building. Indeed, each Weyl substructure $\Delta'$ is itself a thick spherical building. 

We record the list of Weyl substructures, along with their Coxeter type, in Table~\ref{AllWS}. In the table, the \textit{absolute type} is the Coxeter type of the ambient building $\Delta$, and the \textit{relative type} is the Coxeter type of the Weyl substructure~$\Delta'$. Moreover, in each case we list the $\sigma$-conjugacy class $\disp(\theta)$ of displacements of the associated automorphisms of $\Delta$ fixing~$\Delta'$. The notation for these classes is in terms of the admissible diagrams introduced in \cite{PVMclass,PVMexc} (see Theorem~\ref{thm:bicapped} for details).

\begin{table}[h!]
\renewcommand{\arraystretch}{1.5}
\begin{tabular}{c|c||c||c}
\mbox{Abs. type} & \mbox{Rel. type} & \mbox{Description of Weyl substructure} & \mbox{$\sigma$-class}\\ \hline\hline
 \multirow{2}{*}{$\mathsf{A}_{2n-1}$} & $\mathsf{B}_{n}$ & symplectic polar space of rank~$n$ & ${^2}\sA_{2n-1;n}^1$ \\
 &   $\mathsf{A}_{n-1}$  & composition line spread  & ${^1}\sA_{2n-1;n-1}^2$  \\ \hline 
 \multirow{2}{*}{$\mathsf{B}_{n}$} & $\mathsf{B}_{i}$ & ideal subspace of rank $i$ & $\sB_{n;i}^1$ \\ 
 &   $\mathsf{B}_{n/2}$  & composition line spread & $\sB_{n;n/2}^2$  \\ \hline 
 \multirow{2}{*}{$\mathsf{D}_{n}$} & $\mathsf{B}_{i}$ & ideal subspace of rank $i$ & $\sD_{n;i}^1$   \\
 &   $\mathsf{B}_{n/2}$  & composition line spread & $\sD_{n;n/2}^2$ or $\sD_{n;n'/2}^2$  \\ \hline 
  \multirow{2}{*}{$\mathsf{E}_{6}$} & $\mathsf{F}_{4}$& standard split metasymplectic space & ${^2}\sE_{6;4}$   \\ 
    & $\mathsf{A}_{2}$  & ideal Veronesian  & $\sE_{6;2}$  \\ \hline
     \multirow{2}{*}{$\mathsf{E}_{7}$} & $\mathsf{F}_{4}$& partial composition spread  & $\sE_{7;4}$  \\ 
    & $\mathsf{B}_{3}$  & ideal dual polar quaternion Veronesian  & $\sE_{7;3}$  \\ \hline
         \multirow{1}{*}{$\mathsf{E}_{8}$} & $\mathsf{F}_{4}$& quaternion metasymplectic space& $\sE_{8;4}$   \\ \hline
           \multirow{2}{*}{$\mathsf{F}_{4}$} & $\mathsf{B}_{2}$& ideal quadrangular Veronesian  & $\sF_{4;2}$   \\ 
    & $\mathsf{I}_{2}(8)$  & Ree--Tits octagon & ${^2}\sF_{4;2}$    \\ \hline
           \multirow{2}{*}{$\mathsf{I}_2(2m)$} & $\mathsf{A}_{1}$& an ovoid or a spread & $\Cl(s_i)$, $i\in\{1,2\}$   \\ 
    & $\mathsf{A}_1$  & an ovoid-spread pair  &$\Cl^{\sigma}(1)$, $\sigma\neq 1$ \\ \hline 
\end{tabular}
\vspace{6pt}
\caption{Weyl substructures\label{AllWS}}
\end{table}

Before proving Theorem~\ref{thm:main1} we first develop the general theory of the displacement spectra $\disp(\theta)$ of an automorphism of a building, a subject that is interesting in its own right. This is undertaken in Sections~\ref{sec:1} and~\ref{sec:2}. We begin with an analysis of twisted conjugacy classes in (mainly spherical) Coxeter systems (if $\sigma\in\Aut(\Pi)$ then the $\sigma$-conjugacy class of $x\in W$ is 
$
\Cl^{\sigma}(x)=\{w^{-1}xw^{\sigma}\mid w\in W\}
$).
We introduce the notion of a \textit{$\sigma$-involution}, generalising the concept of an ordinary involution. The twisted conjugacy classes of $\sigma$-involutions in spherical Coxeter groups turn out to be fundamentally important in our study of uniclass automorphisms, and we classify these classes in Theorem~\ref{thm:downwardsclosure}. 

We define the notion of a \textit{bi-capped} class of $\sigma$-involutions. These are the classes of $\sigma$-involutions with a unique minimal length element, and a unique maximal length element. We classify these classes in Theorem~\ref{thm:bicapped}, and it turns out that these classes correspond to a subset of the admissible diagrams of \cite{PVMclass,PVMexc}. Moreover, by comparing with Theorem~\ref{thm:main1} it turns out that the bi-capped classes are precisely the twisted classes that occur as the displacement sets of uniclass automorphisms. Thus we have the following corollary.

\begin{cor1}\label{cor:bicapped}
The twisted conjugacy classes that occur as the displacement set of an automorphism of some thick irreducible building of spherical type are precisely the bi-capped classes of twisted involutions. 
\end{cor1}

There is a natural ``duality'' on the set of all twisted conjugacy classes of a spherical Coxeter group given by multiplication by the longest element. This duality is explicitly computed for the bi-capped classes in Theorem~\ref{thm:bicapped}. For example, in type $\sE_7$ the class $\sE_{7;4}$ (this is the conjugacy class $\Cl(s_2s_5s_7)$) is dual to the class $\sE_{7;3}$ (this is the conjugacy class $\Cl(w_{\sD_4})$ with $w_{\sD_4}$ the longest element of the $\sD_4$ parabolic subgroup). We say that two Weyl substructures $\Gamma$ and $\Gamma'$ in (possibly different) buildings of the same type are \textit{paired} if the displacement sets of automorphisms fixing the substructures are dual twisted conjugacy classes in the Coxeter group. We have the following corollary to Theorem~\ref{thm:main1}. 

\begin{cor1}\label{cor:pairing}
If $\Gamma$ is a Weyl substructure in a thick irreducible spherical building~$\Delta$, then there exists a (possibly different) thick irreducible spherical building~$\Delta'$ of the same type, and Weyl substructure $\Gamma'$ in $\Delta'$, such that $\Gamma$ and $\Gamma'$ are paired. Moreover, the ranks of $\Gamma$ and $\Gamma'$ sum to the rank of $\Delta$.
\end{cor1}

The need to change the building (but not the type) in Corollary~\ref{cor:pairing} is unavoidable. For example, in the building $\sE_7(\KK)$ Weyl substructures of relative type $\sF_4$ exist if and only if $\K$ admits a quadratic extension, while Weyl substructures of relative type $\sB_3$ exist if and only if there exists a quaternion division algebra $\mathbb{H}$ over $\K$. The fact that the ranks of $\Gamma$ and $\Gamma'$ necessarily add to the rank of $\Delta$ admits a theoretical explanation, see Proposition~\ref{prop:ranksum}.

We also classify the uniclass automorphisms and Weyl substructures that occur in finite thick irreducible spherical buildings (see Theorem~\ref{thm:finitecase}). Moreover, in the finite case we show that the cardinalities of the sets
$
\Delta_w(\theta)=\{C\in\Delta\mid\delta(C,C^{\theta})=w\}
$
can be computed for all uniclass automorphisms. Explicitly, we have (see Theorem~\ref{thm:counts2})
$$
|\Delta_w(\theta)|=|\Delta_J||\Delta'|q_w^{1/2}q_{w_J}^{-1/2},
$$
where the unique minimal length element of the associated (bi-capped) class is $w_J$ (with $J\subseteq S$), $\Delta_J$ is the any residue of $\Delta$ of type $J$, and $\Delta'$ is the (chamber set) of the associated fixed Weyl substructure. Here $q_s$, $s\in S$, are the usual thickness parameters of the finite building, and $q_w=q_{s_1}\cdots q_{s_k}$ whenever $w=s_1\cdots s_k$ is reduced. 

Recall that an automorphism of a spherical building is called \textit{domestic} if it maps no chamber to an opposite chamber. This concept plays a crucial role in the present work, for the following reason: If $\theta$ is uniclass and the companion automorphism of $\theta$ is the opposition relation, then either $\theta$ is domestic or $\theta$ is anisotropic. Domestic automorphisms are rather rare, and there is now an extensive literature on the topic, including~\cite{Lam-Mal:23,NPVV,PTV,PVM:19a,PVM:19b,PVMexc,PVMexc2,PVMclass,TTVM3,TTVM2,Mal:12,Mal:14}, and forthcoming paper~\cite{PVMexc4} dealing with the $\sE_8$ case, which cumulatively give an essentially complete classification of domestic automorphisms of spherical buildings. 

Section~\ref{sec:1} gives background and preliminary results on twisted conjugacy classes and automorphisms of buildings, including the classification of bi-capped classes. In section~\ref{sec:2} we develop the general theory of displacement spectra for automorphisms of buildings. While ultimately we are interested mainly in the spherical case in this paper, we set up some of the machinery in a more general setting. This section also discusses the \textit{fixed} and \textit{opposition diagrams} of uniclass automorphisms (see Proposition~\ref{prop:duality}), and develops methods to compute the cardinalities $|\Delta_w(\theta)|$ in the finite case (see Theorems~\ref{thm:counting} and~\ref{thm:counts}). Moreover, we show in Proposition~\ref{prop:reduciblecase} and Theorem~\ref{thm:reducetothick} that the classification of uniclass automorphisms of non-thick spherical buildings reduces to the thick irreducible case, thus justifying the thickness and irreducibility assumptions in Theorem~\ref{thm:main1}. 

Section~\ref{sec:3} begins the proof of Theorem~\ref{thm:main1}. In particular, in Theorem~\ref{thm:uncapped} we prove that all uniclass automorphisms of a thick spherical building are ``capped'' (meaning that if there exist simplices of types $J_1$ and $J_2$ mapped to opposite simplices by $\theta$, then there is a type $J_1\cup J_2$ simplex mapped to an opposite by~$\theta$). This important property allows us to more easily apply the theory of domesticity. 

Section~\ref{sec:4} contains the bulk of the proof of Theorem~\ref{thm:main1}. The analysis is case-by-case on the type of the building, making extensive use of the literature on domestic automorphisms of spherical buildings (the results for $\sE_8$ are conditional on some results that will appear in a forthcoming paper~\cite{PVMexc4}).

 In Section~\ref{sec:5} we classify the finite Weyl substructures, and compute the sets $|\Delta_w(\theta)|$ for uniclass automorphisms of finite spherical buildings. In Section~\ref{FTMS} we provide a connection between uniclass automorphisms and the Freudenthal--Tits Magic Square.

\textit{Acknowledgement:} We thank Arun Ram for helpful discussions, in particular regarding the counting arguments contained in Section~\ref{sec:counting}.

\section{Background and preliminary results}\label{sec:1}

This section contains background on Coxeter groups, twisted conjugacy classes, buildings, and automorphisms of spherical buildings. We introduce the notion of \textit{bicapped} twisted conjugacy classes in a spherical Coxeter group, and classify these classes. It will turn out that these classes are precisely the classes that occur as the displacement sets of uniclass automorphisms of spherical buildings. 

\subsection{Coxeter groups and twisted conjugacy classes}

Let $(W,S)$ be an arbitrary Coxeter system with Coxeter graph $\Pi=\Pi(W,S)$. For $J\subseteq S$ let $W_J=\langle J\rangle$ be subgroup of $W$ generated by $J$, and let
\begin{align*}
W^J&=\{x\in W \mid \ell(xs)=\ell(x)+1\text{ for all $s\in J$}\}
\end{align*}
be the set of minimal length coset representatives for cosets in $W/W_J$. 

Let $D_L(w)=\{s\in S\mid \ell(sw)<\ell(w)\}$ and $D_R(w)=\{s\in S\mid \ell(ws)<\ell(w)\}$ be the left and right descent sets of~$w$.

A subset $J\subseteq S$ is called \textit{spherical} if $W_J$ is a finite group. In this case we write $w_J$ for the longest element of $W_J$. If $S$ is spherical (that is, $|W|<\infty$) we denote the longest element of $W$ by $w_0=w_S$.

The following well known facts about Coxeter groups are used frequently.

\begin{lemma}\label{lem:foldingcondition}
Let $(W,S)$ be an arbitrary Coxeter system. Let $w\in W$ and $s,t\in S$.
\begin{compactenum}[$(1)$]
\item If $\ell(sw)=\ell(wt)=\ell(w)+1$ then either $\ell(swt)=\ell(w)+2$ or $swt=w$. 
\item If $\ell(swt)=\ell(w)$ and $\ell(sw)=\ell(wt)$ then $swt=w$.
\end{compactenum}
\end{lemma}

\begin{proof}
(1) is an easy application of the exchange condition (see \cite[p.79]{AB:09}). For (2), if $\ell(sw)=\ell(w)+1$ then the result follows from (1), and if $\ell(sw)=\ell(w)-1$ then let $v=sw$. Then $\ell(sv)=\ell(w)=\ell(sw)+1=\ell(v)+1$, and $\ell(vt)=\ell(swt)=\ell(w)=\ell(v)+1$, and $\ell(svt)=\ell(wt)=\ell(w)-1=\ell(v)$, and apply~(1).
\end{proof}

Let $(W,S)$ be an arbitrary Coxeter system. Let $\sigma\in\mathrm{Aut}(\Pi)$ (a \textit{diagram automorphism}) and let $w\in W$. The $\sigma$-conjugacy class of $W$ is
$$
\Cl^{\sigma}(w)=\{v^{-1}wv^{\sigma}\mid v\in W\}.
$$
In particular, if $\sigma=1$ then $\Cl^{\sigma}(w)=\Cl(w)$ (a usual conjugacy class). We refer the reader to \cite{GKP:00} for a general treatment of twisted conjugacy classes in finite Coxeter groups.

The nonempty subsets $K\subseteq S$ that are minimal subject to being preserved by $\sigma$ are called the \textit{distinguished $\sigma$-orbits}, and for subsets $J\subseteq S$ preserved by $\sigma$ we write
$$
\Orb^{\sigma}(J)=\{\text{distinguished $\sigma$-orbits $K$ with $K\subseteq J$}\}.
$$
For example, $W$ is of type $\sA_5$ and $\sigma$ has order $2$ then $\Orb^{\sigma}(S)=\{\{1,5\},\{2,4\},\{3\}\}$.

We record some basic observations.

\begin{prop}\label{prop:basic1}
Let $(W,S)$ be a Coxeter system. Let $\sigma$ be a diagram automorphism and let $\bC$ be a $\sigma$-conjugacy class.
\begin{compactenum}[$(1)$]
\item If $u,v\in\bC$ then $\ell(u)=\ell(v)\mod 2$.
\item The order of $ww^{\sigma}w^{\sigma^2}\cdots w^{\sigma^{n-1}}$ is constant for $w\in\bC$, where $n=\mathrm{ord}(\sigma)$.
\end{compactenum}
\end{prop}

\begin{proof}
(1) follows from the fact that $\ell(sws^{\sigma})\in\{\ell(w)-2,\ell(w),\ell(w)+2\}$ and induction. (2) follows from the fact that if $y=v^{-1}xv^{\sigma}$ then $yy^{\sigma}\cdots y^{\sigma^{n-1}}=v^{-1}xx^{\sigma}\cdots x^{\sigma^{n-1}}v$. 
\end{proof}

\begin{lemma}\label{lem:min1}
Let $\sigma\in\Aut(\Pi)$. Suppose that $J\subseteq S$ is a spherical subset with the property that $sw_J=w_Js^{\sigma}$ for all $s\in J$. Then $w_J$ has minimal length in the class $\Cl^{\sigma}(w_J)$.
\end{lemma}

\begin{proof}
Let $w\in W$ and write $w=xu$ with $x\in W^J$ and $u\in W_J$. The hypothesis gives $uw_Ju^{-\sigma}=w_J$, and hence $ww_Jw^{-\sigma}=xw_Jx^{-\sigma}$. Note that for any $v\in W_J$ we have 
$$
\ell(x)+\ell(v)=\ell(xv)=\ell(xvx^{-\sigma}x^{\sigma})\leq \ell(xvx^{-\sigma})+\ell(x^{\sigma}),
$$
and since $\ell(x)=\ell(x^{\sigma})$ this gives $\ell(xvx^{-\sigma})\geq \ell(v)$ for all $v\in W_J$. In particular we have
$$
\ell(ww_Jw^{-\sigma})=\ell(xw_Jx^{-\sigma})\geq \ell(w_J),
$$
and so $w_J$ is of minimal length in $\Cl^{\sigma}(w_J)$. 
\end{proof}

\subsection{An involution on the set of twisted classes}\label{sec:paired}

Let $(W,S)$ be a spherical Coxeter system and let $\sigma_0\in\Aut(\Pi)$ be the opposition relation given by $s^{\sigma_0}=w_0sw_0$ for $s\in S$. 

\begin{lemma}\label{lem:pairing}
Let $(W,S)$ be spherical and let $\sigma\in\Aut(\Pi)$. For $x\in W$ we have $\Cl^{\sigma}(x)w_0=\Cl^{\sigma\sigma_0}(xw_0)$ and $w_0\Cl^{\sigma}(x)=\Cl^{\sigma\sigma_0}(w_0x)$.
\end{lemma}

\begin{proof}
Let $\bC=\Cl^{\sigma}(x)$. Then $\bC w_0=\{w^{-1}xw^{\sigma}w_0\mid w\in W\}$. If $\sigma=1$ then $w^{-1}xw^{\sigma}w_0=w^{-1}xw_0w^{\sigma\sigma_0}$ and so $\bC w_0=\Cl^{\sigma\sigma_0}(xw_0)$. If $\sigma=\sigma_0$ then since $\sigma_0^2=1$ we have $w^{-1}xw^{\sigma}w_0=w^{-1}xw_0w$ and so $\bC w_0=\Cl(xw_0)=\Cl^{\sigma\sigma_0}(xw_0)$. If $\sigma\neq 1$ with $\sigma\neq \sigma_0$ then, from the classification of spherical Coxeter systems we have $\sigma_0=1$ and so $w^{-1}xw^{\sigma}w_0=w^{-1}xw_0w^{\sigma}$ and hence $\bC w_0=\Cl^{\sigma}(xw_0)=\Cl^{\sigma\sigma_0}(xw_0)$. The statement for $w_0\bC$ is similar.
\end{proof}

Let $\cC$ denote the set of all twisted conjugacy classes in $W$ (for all choices of $\sigma$). Lemma~\ref{lem:pairing} shows that there is an involutive bijection $\psi:\cC\to\cC$ given by 
$$
\psi(\bC)=\bC w_0.
$$
This bijection maps $\sigma$-classes to $\sigma\sigma_0$-classes. 

\subsection{Twisted involutions}

The following notion will play an important role in this work.

\begin{defn}
An element $w\in W$ is called a \textit{$\sigma$-involution} if $\sigma^2=1$ and $ww^{\sigma}=1$. 
\end{defn}

In particular note that if $\sigma=1$ then $\sigma$-involutions are precisely involutions, and if $\sigma^2\neq 1$ then there are no $\sigma$-involutions. We record some basic observations.

\begin{lemma}\label{lem:sigmainvolutions}
If $w\in W$ is a $\sigma$-involution then every element of the $\sigma$-conjugacy class $\Cl^{\sigma}(w)$ is a $\sigma$-involution.
\end{lemma}

\begin{proof}
Immediately follows from Proposition~\ref{prop:basic1}(2).
\end{proof}

\begin{lemma}\label{lem:sigmacondition}
If $w$ is a $\sigma$-involution and $s\in S$ with $\ell(sws^{\sigma})=\ell(w)$ then $sws^{\sigma}=w$. 
\end{lemma}

\begin{proof}
Since $w$ is a $\sigma$-involution we have $ws^{\sigma}=w^{-\sigma}s^{\sigma}=(sw)^{-\sigma}$, and hence $\ell(ws^{\sigma})=\ell(sw)$, and so by Lemma~\ref{lem:foldingcondition} we have $sws^{\sigma}=w$. 
\end{proof}

\begin{lemma}
Let $(W,S)$ be spherical. If $\bC$ is a class of $\sigma$-involutions, then $\psi(\bC)=\bC w_0$ is a class of $\sigma\sigma_0$-involutions. 
\end{lemma}

\begin{proof}
If $\bC$ is a $\sigma$-class then $\bC w_0$ is a $\sigma\sigma_0$-class, and since $\sigma^2=\sigma_0^2=1$ and $\sigma\sigma_0=\sigma_0\sigma$ we have $(\sigma\sigma_0)^2=1$. Let $w\in \bC$, and so $ww^{\sigma}=1$. Then $ww_0\in\bC w_0$ and since $w_0^{\rho}=w_0$ for all $\rho\in\Aut(\Pi)$ we have $(ww_0)(ww_0)^{\sigma\sigma_0}=ww_0w^{\sigma\sigma_0}w_0=ww^{\sigma}w_0^2=1$. So $\psi(\bC)$ is a $\sigma\sigma_0$-involution class. 
\end{proof}

The following theorem gives the classification of $\sigma$-classes of $\sigma$-involutions, showing that they are in bijection with the set of spherical subsets $J\subseteq S$ up to $\sigma$-conjugaction  for which $sw_J=w_Js^{\sigma}$ for all $s\in J$. The proof is adapted from \cite[Proposition~3.2.10]{GF:00} (where the case $\sigma=1$ is given). Geck and Pfeiffer attribute the original result to Richardson~\cite{Rich} and Howlett~\cite{Howlett}.

\begin{thm}\label{thm:downwardsclosure}
Let $(W,S)$ be an arbitrary Coxeter system and let $\bC$ be a $\sigma$-conjugacy class consisting of $\sigma$-involutions. 
\begin{compactenum}[$(1)$]
\item Every minimal length element of $\bC$ is equal to $w_J$ for some subset $J\subseteq S$ (possibly empty) with the property that $sw_J=w_Js^{\sigma}$. 
\item If $w_J$ and $w_{J'}$ are minimal length in $\bC$ then $J'=wJw^{-\sigma}$ for some $w\in W$. 
\item If $w\in \bC$ then there exists $v\in W$ with $\ell(v^{-1}wv^{\sigma})=\ell(w)-2\ell(v)$ and $v^{-1}wv^{\sigma}=w_J$ for some spherical subset $J\subseteq S$ with $w_J$ of minimal length in $\bC$.
\end{compactenum}
\end{thm}

\begin{proof}
Let $K=\{s\in D_L(w)\mid sw=ws^{\sigma}\}$ (a spherical subset by \cite[Proposition~2.17]{AB:09}). Since $w$ is a $\sigma$-involution we have $D_L(w)=D_R(w)^{\sigma}$, and hence $w=w_{K}x=x'w_{K}^{\sigma}$ for some $x,x'\in W$ with $\ell(w_Kx)=\ell(w_K)+\ell(x)$ and $\ell(x'w_K^{\sigma})=\ell(x')+\ell(w_K^{\sigma})$. We have $x=w_Kw$ and $x'=ww_K^{\sigma}$, and the definition of $K$ implies that $w_Kw=ww_K^{\sigma}$. Thus $x'=x$.

Suppose that $x\neq 1$. Since $w_Kw=ww_K^{\sigma}$ we have
$$
x^{-1}=(w_Kw)^{-1}=w^{-1}w_K=w^{\sigma}w_K=(ww_K^{\sigma})^{\sigma}=(w_Kw)^{\sigma}=x^{\sigma},
$$
and so $x$ is a $\sigma$-involution. Thus $D_L(x)=D_R(x)^{\sigma}$ and so $\ell(xt^{\sigma})=\ell(tx)$ for all $t\in S$. Let $t\in D_L(x)$. Since $w=xw_K^{\sigma}$ we have $t\in D_L(w)$ and so $\ell(twt^{\sigma})\in\{\ell(w)-2,\ell(w)\}$. Since $\ell(w_Kx)=\ell(w_K)+\ell(x)$ we have $t\notin K$, and hence $tw\neq wt^{\sigma}$. It follows from Lemma~\ref{lem:foldingcondition} that $\ell(twt^{\sigma})=\ell(w)-2$.

Note that $twt^{\sigma}\in\bC$ is a $\sigma$-involution. Repeating the above argument until $x=1$ we obtain a sequence $t=t_1,\ldots,t_n$ such that $w'=t_n\cdots t_1wt_1^{\sigma}\cdots t_n^{\sigma}$ has length $\ell(w')=\ell(w)-2n$, and such that $w'=w_{J}=w_{J}^{\sigma}$ where $J=\{s\in D_L(w')\mid sw'=w's^{\sigma}\}$, and so $sw_{J}=w_{J}s^{\sigma}$ for all $s\in J$. The element $w_J$ is of minimal length in $\bC$ by Lemma~\ref{lem:min1}, and the lemma easily follows.
\end{proof}

Let $\bC$ be a class of $\sigma$-involutions in a spherical Coxeter group. By the classification of $\sigma$-involution classes given in Theorem~\ref{thm:downwardsclosure}, and the fact that the involution $\psi$ produces a $\sigma\sigma_0$-involution class, one can define a ``fixed rank'' and a ``opposition rank'' for $\bC$ as follows. Let $w_J$ (respectively $w_{J'}$) be a minimal length element of $\bC$ (respectively $\psi(\bC)$), and let
\begin{align*}
\FixRk(\bC)&=|\Orb^{\sigma}(S\backslash J)|\\
\OppRk(\bC)&=|\Orb^{\sigma\sigma_0}(S\backslash J')|.
\end{align*}
This does not depend on the particular $J,J'$ chosen (by Theorem~\ref{thm:downwardsclosure}). See Theorem~\ref{thm:bicapped} for motivation regarding the terminology of fixed rank and opposite rank. 

It is clear from the definition that
$$
\OppRk(\bC)=\FixRk(\bC w_0).
$$

\begin{prop}\label{prop:ranksum}
Let $\bC$ be a class of twisted involutions in a spherical Coxeter system $(W,S)$. Then 
$$
\FixRk(\bC)+\OppRk(\bC)=|S|
$$
\end{prop}

\begin{proof} 
Let $\Phi$ be a (not necessarily crystallographic) root system associated to $(W,S)$ and let $V$ be the real vector space spanned by the simple roots $\{\alpha_s\mid s\in S\}$. Each diagram automorphism $\rho$ acts on $V$ by $\rho\cdot \alpha_s=\alpha_{s^{\rho}}$ for $s\in S$. 

Since $\sigma^2=1$ we have $V=V^{\sigma}\oplus \tilde{V}^{\sigma}$ with $\sigma\cdot v=v$ for all $v\in V^{\sigma}$, and $\sigma\cdot v=-v$ for all $v\in \tilde{V}^{\sigma}$. Explicitly, $V^{\sigma}$ has basis $\{v_K\mid K\in\Orb^{\sigma}(S)\}$ where $v_K=\sum_{s\in K}\alpha_s$, and $\tilde{V}^{\sigma}$ has basis $\{\alpha_s-\alpha_{s^{\sigma}}\mid s\in S,\, s^{\sigma}\neq s\}$. We claim that $w_J\cdot v=v$ for all $v\in \tilde{V}^{\sigma}$. To see this, let $\tilde{V}_J^{\sigma}$ be the subspace of $\tilde{V}^{\sigma}$ spanned by $\{\alpha_s-\alpha_{s^{\sigma}}\mid s\in J,\,s^{\sigma}\neq s\}$ and let $(\tilde{V}_J^{\sigma})^{\perp}$ be an orthogonal complement in $\tilde{V}^{\sigma}$. Since $w_J\cdot \alpha_s=-\alpha_{s^{\sigma}}$ for $s\in J$ we have that $w_J$ acts by $1$ on $\tilde{V}_J^{\sigma}$. If $v\in (\tilde{V}_J^{\sigma})^{\perp}$ then $(v,\alpha_s-\alpha_{s^{\perp}})=0$ for all $s\in J$, which gives $(v,\alpha_s)=(v,\sigma\cdot \alpha_s)=(\sigma\cdot v,\alpha_s)=-(v,\alpha_s)$, and so $(v,\alpha_s)=0$ for all $s\in J$. Thus $w_J$ also acts by $1$ on $(\tilde{V}_J^{\sigma})^{\perp}$. 

Now let $V^{\sigma}_J$ be the subspace of $V^{\sigma}$ spanned by $\{v_K\mid K\in\Orb^{\sigma}(J)\}$ and let $(V_J^{\sigma})^{\perp}$ be an orthogonal complement on $V^{\sigma}$. Then $w_J$ acts by $-1$ on $V_J^{\sigma}$, and we claim that $w_J$ acts by $1$ on $(V_J^{\sigma})^{\perp}$. For if $v\in (V_J^{\sigma})^{\perp}$ then $(v,v_K)=0$ for all $K\in\Orb^{\sigma}(\bC)$. If $K=\{s\}$ then $(v,\alpha_s)=0$, and if $K=\{s,s^{\sigma}\}$ with $s^{\sigma}\neq s$ then $(v,\alpha_s)=-(v,\alpha_{s^{\sigma}})=-(\sigma\cdot v,\alpha_s)=-(v,\alpha_s)$ as $v\in V^{\sigma}$, and so $(v,\alpha_s)=0$ for all $s\in J$. Thus $w_J$ acts by $1$ on $(V_J^{\sigma})^{\perp}$. 

In summary, the $1$-eigenspace for the action of $w_J\sigma$ on $V$ is $(V_J^{\sigma})^{\perp}$. Since 
$$\dim (V_J^{\sigma})^{\perp}=\dim V^{\sigma}-\dim V_J^{\sigma}=|\Orb^{\sigma}(S)|-|\Orb^{\sigma}(J)|=|\Orb^{\sigma}(S\backslash J)|
$$
it follows that $\FixRk(\bC)$ is the dimension of the $1$-eigenspace of $w_J\sigma$ acting on $V$ and hence $\FixRk(\bC)$ is the dimension of the $1$-eigenspace of $x\sigma$ acting on $V$ for all $x\in \bC$. 

Since $\psi(\bC)$ is a $\sigma\sigma_0$-involution class, and since $w_Jw_0\in\psi(\bC)$, the above analysis shows that $\OppRk(\bC)=\FixRk(\psi(\bC))$ is the dimension of the $1$-eigenspace of $w_Jw_0\sigma_0\sigma$ acting on $V$. But $w_0\sigma_0$ acts by $-1$ on $V$, and hence (from the previous paragraphs) this $1$-eigenspace is $V_J^{\sigma}\oplus \tilde{V}^{\sigma}$, and the result follows. 
\end{proof}

\subsection{Admissible diagrams}\label{sec:admissible}
In \cite{PVM:19a} the second and third authors introduced the notion of admissible diagrams in the study of domestic automorphisms. An admissible diagram is a triple $(\Pi,J,\sigma)$ where $\Pi$ is the Coxeter graph of an irreducible spherical Coxeter system~$S$, $J$ is a subset of $S$, and $\sigma\in\Aut(\Pi)$, and various axioms are satisfied. We shall not require the axiomatic definition in the present paper, and instead we refer the reader to \cite[Theorem~2.3]{PVM:19a} for the complete list of admissible diagrams of irreducible spherical Coxeter systems.

\begin{remark}\label{rem:change}
It is more convenient in the present paper to let $(\Pi,J,\sigma)$ denote the admissible diagram that was denoted $(\Pi,J,\sigma\sigma_0)$ in \cite{PVM:19a}. Apart from this notational change the definition of opposition diagrams from~\cite{PVM:19a} is unchanged. 
\end{remark}

In \cite[Tables 1 and 2]{PVMexc} and \cite[Tables 1 and 2]{PVMclass} we have given each admissible diagram a symbol, typically of the form ${^t}\mathsf{X}_{n;i}^j$, where $t$ is the ``twisting index'', and we will use these symbols in the present paper. With the notational change in Remark~\ref{rem:change}, the twisting index of $(\Pi,J,\sigma)$ is $\mathrm{ord}(\sigma)$ (rather than $\mathrm{ord}(\sigma\sigma_0)$). We often omit the twisting index in the case $t=1$, and moreover as noted in the introduction to \cite{PVMclass} the twisting index for diagrams of classical type can be omitted without ambiguity -- this is particularly useful for type $\sD$ diagrams. 

The admissible diagram $(\Pi,J,\sigma)$ is drawn by encircling the distinguished orbits $K\in\Orb^{\sigma}(J)$ on the Coxeter graph~$\Pi$. The diagram is drawn ``straight'' if $\sigma=1$ and ``bent'', in the usual way, otherwise. 

\begin{example}\label{ex:noncrystallographic}
The following are admissible diagrams:
$$
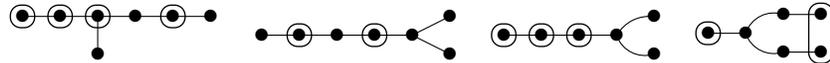
\begin{tikzpicture}[scale=0.5]//,baseline=-2ex]
\node [inner sep=0.8pt,outer sep=0.8pt] at (-2,0) (1) {$\bullet$};
\node [inner sep=0.8pt,outer sep=0.8pt] at (-1,0) (3) {$\bullet$};
\node [inner sep=0.8pt,outer sep=0.8pt] at (0,0) (4) {$\bullet$};
\node [inner sep=0.8pt,outer sep=0.8pt] at (1,0) (5) {$\bullet$};
\node [inner sep=0.8pt,outer sep=0.8pt] at (2,0) (6) {$\bullet$};
\node [inner sep=0.8pt,outer sep=0.8pt] at (3,0) (7) {$\bullet$};
\node [inner sep=0.8pt,outer sep=0.8pt] at (0,-1) (2) {$\bullet$};
\draw (-2,0)--(3,0);
\draw (0,0)--(0,-1);
\draw [line width=0.5pt,line cap=round,rounded corners] (1.north west)  rectangle (1.south east);
\draw [line width=0.5pt,line cap=round,rounded corners] (3.north west)  rectangle (3.south east);
\draw [line width=0.5pt,line cap=round,rounded corners] (4.north west)  rectangle (4.south east);
\draw [line width=0.5pt,line cap=round,rounded corners] (6.north west)  rectangle (6.south east);
\end{tikzpicture}
\quad 
\begin{tikzpicture}[scale=0.5]
\node [inner sep=0.8pt,outer sep=0.8pt] at (0,0) (0) {$\bullet$};
\node [inner sep=0.8pt,outer sep=0.8pt] at (1,0) (1) {$\bullet$};
\node [inner sep=0.8pt,outer sep=0.8pt] at (2,0) (2) {$\bullet$};
\node [inner sep=0.8pt,outer sep=0.8pt] at (3,0) (3) {$\bullet$};
\node [inner sep=0.8pt,outer sep=0.8pt] at (4,0) (4) {$\bullet$};
\node [inner sep=0.8pt,outer sep=0.8pt] at (5,0.5) (5a) {$\bullet$};
\node [inner sep=0.8pt,outer sep=0.8pt] at (5,-0.5) (5b) {$\bullet$};
\draw (0,0)--(4,0);
\draw (4,0) to   (5,0.5);
\draw (4,0) to   (5,-0.5);
\draw [line width=0.5pt,line cap=round,rounded corners] (1.north west)  rectangle (1.south east);
\draw [line width=0.5pt,line cap=round,rounded corners] (3.north west)  rectangle (3.south east);
\end{tikzpicture}\quad
 \begin{tikzpicture}[scale=0.5]
\node [inner sep=0.8pt,outer sep=0.8pt] at (1,0) (1) {$\bullet$};
\node [inner sep=0.8pt,outer sep=0.8pt] at (2,0) (2) {$\bullet$};
\node [inner sep=0.8pt,outer sep=0.8pt] at (3,0) (3) {$\bullet$};
\node [inner sep=0.8pt,outer sep=0.8pt] at (4,0) (4) {$\bullet$};
\node [inner sep=0.8pt,outer sep=0.8pt] at (5,0.5) (5a) {$\bullet$};
\node [inner sep=0.8pt,outer sep=0.8pt] at (5,-0.5) (5b) {$\bullet$};
\draw (1,0)--(4,0);
\draw (4,0) to [bend left] (5,0.5);
\draw (4,0) to [bend right=45] (5,-0.5);
\draw [line width=0.5pt,line cap=round,rounded corners] (1.north west)  rectangle (1.south east);
\draw [line width=0.5pt,line cap=round,rounded corners] (3.north west)  rectangle (3.south east);
\draw [line width=0.5pt,line cap=round,rounded corners] (2.north west)  rectangle (2.south east);
\end{tikzpicture}
\quad 
\begin{tikzpicture}[scale=0.5,baseline=-2.2ex]
\node at (0,0.3) {};
\node [inner sep=0.8pt,outer sep=0.8pt] at (-2,0) (2) {$\bullet$};
\node [inner sep=0.8pt,outer sep=0.8pt] at (-1,0) (4) {$\bullet$};
\node [inner sep=0.8pt,outer sep=0.8pt] at (0,-0.5) (5) {$\bullet$};
\node [inner sep=0.8pt,outer sep=0.8pt] at (0,0.5) (3) {$\bullet$};
\node [inner sep=0.8pt,outer sep=0.8pt] at (1,-0.5) (6) {$\bullet$};
\node [inner sep=0.8pt,outer sep=0.8pt] at (1,0.5) (1) {$\bullet$};
\draw (-2,0)--(-1,0);
\draw (-1,0) to [bend left=45] (0,0.5);
\draw (-1,0) to [bend right=45] (0,-0.5);
\draw (0,0.5)--(1,0.5);
\draw (0,-0.5)--(1,-0.5);
\draw [line width=0.5pt,line cap=round,rounded corners] (2.north west)  rectangle (2.south east);
\draw [line width=0.5pt,line cap=round,rounded corners] (1.north west)  rectangle (6.south east);
\end{tikzpicture}
$$
denoted $\sE_{7;4}$, ${^1}\sD_{7;2}^2$, ${^2}\sD_{6;3}^1$, and ${^2}\sE_{6;2}$. In the first case, $J=\{1,3,4,6\}$ (in Bourbaki labelling) and $\sigma=1$. In the third case $J=\{1,2,3\}$ and $\sigma$ has order~$2$. We note that the non-crystallographic diagrams are not explicitly given symbols in \cite{PVMexc,PVMclass}, and thus we extend the notation as follows (we shall not need symbols for the $\mathsf{H}_3$ and $\mathsf{H}_4$ cases):
$$\mathsf{I}_{2;0}(m)\,=\,\begin{tikzpicture}[scale=0.5,baseline=-0.5ex]
\node at (1,0.8) {};
\node at (1,-0.8) {};
\node at (1,0) (0) {};
\node [inner sep=0.8pt,outer sep=0.8pt] at (1,0) (1) {$\bullet$};
\node [inner sep=0.8pt,outer sep=0.8pt] at (2,0) (2) {$\bullet$};
\draw (1,0)--(2,0);
\node at (1.5,0.1) [above] {\small{$m$}};
\end{tikzpicture}\quad \mathsf{I}_{2;1}^1(m)\,=\,\begin{tikzpicture}[scale=0.5,baseline=-0.5ex]
\node at (1,0.8) {};
\node at (1,-0.8) {};
\node at (1,0) (0) {};
\node [inner sep=0.8pt,outer sep=0.8pt] at (1,0) (1) {$\bullet$};
\node [inner sep=0.8pt,outer sep=0.8pt] at (2,0) (2) {$\bullet$};
\draw (1,0)--(2,0);
\node at (1.5,0.1) [above] {\small{$m$}};
\draw [line width=0.5pt,line cap=round,rounded corners] (1.north west)  rectangle (1.south east);
\end{tikzpicture}\quad \mathsf{I}_{2;1}^2(m)\,=\,\begin{tikzpicture}[scale=0.5,baseline=-0.5ex]
\node at (1,0.8) {};
\node at (1,-0.8) {};
\node at (1,0) (0) {};
\node [inner sep=0.8pt,outer sep=0.8pt] at (1,0) (1) {$\bullet$};
\node [inner sep=0.8pt,outer sep=0.8pt] at (2,0) (2) {$\bullet$};
\draw (1,0)--(2,0);
\node at (1.5,0.1) [above] {\small{$m$}};
\draw [line width=0.5pt,line cap=round,rounded corners] (2.north west)  rectangle (2.south east);
\end{tikzpicture}\quad \mathsf{I}_{2;2}(m)\,=\,\begin{tikzpicture}[scale=0.5,baseline=-0.5ex]
\node at (1,0.8) {};
\node at (1,-0.8) {};
\node at (1,0) (0) {};
\node [inner sep=0.8pt,outer sep=0.8pt] at (1,0) (1) {$\bullet$};
\node [inner sep=0.8pt,outer sep=0.8pt] at (2,0) (2) {$\bullet$};
\draw (1,0)--(2,0);
\node at (1.5,0.1) [above] {\small{$m$}};
\draw [line width=0.5pt,line cap=round,rounded corners] (1.north west)  rectangle (1.south east);
\draw [line width=0.5pt,line cap=round,rounded corners] (2.north west)  rectangle (2.south east);
\end{tikzpicture}\quad {^2}\mathsf{I}_{2;1}(m)\,=\,\begin{tikzpicture}[scale=0.5,baseline=-0.5ex]
\node at (1,0.8) {};
\node at (1,-0.8) {};
\node at (1,0) (0) {};
\node [inner sep=0.8pt,outer sep=0.8pt] at (1,0.5) (1) {$\bullet$};
\node [inner sep=0.8pt,outer sep=0.8pt] at (1,-0.5) (2) {$\bullet$};
\node at (2.5,0) {\small{$m$}};
\draw [domain=-90:90] plot ({1+(0.7)*cos(\x)}, {(0.5)*sin(\x)});
\draw [line width=0.5pt,line cap=round,rounded corners] (1.north west)  rectangle (2.south east);
\end{tikzpicture}
$$
\end{example}

We note that this notion of admissible diagrams is distinct from the ``admissible diagrams'' introduced by Carter~\cite{Car:72} to classify conjugacy classes in Weyl groups.

\begin{prop}
Let $D=(\Pi,J,\sigma)$ be an admissible diagram. For each distinguished $\sigma$-orbit $K\in\Orb^{\sigma}(J)$ let $s_K=w_{S\backslash J}w_K$, and let $W_{D}$ be the subgroup of $W$ generated by $S_D=\{s_K\mid K\in\Orb^{\sigma}(J)\}$. Then $(W_{D},S_{D})$ is a Coxeter system.
\end{prop}

\begin{proof}
Each admissible diagram is a \textit{Tits index} in the sense of \cite[Definition~20.1]{MPW:15}, and the result follows from \cite[Theorem~20.32]{MPW:15} (see also \cite[\S2.5]{Tits:66}).
\end{proof}

\begin{defn}\label{defn:relativetype}
The \textit{relative type} of an admissible diagram $D$ is the type of the associated Coxeter system $(W_D,S_D)$. 
\end{defn}

For example, the relative type of the admissible diagram $\sE_{7;4}$ is $\sF_4$. 

\subsection{Bi-capped classes}\label{sec:bicapped} Let $\bC$ be a class of $\sigma$-involutions in a spherical Coxeter group. We call $\bC$ \textit{lower capped} (respectively \textit{upper capped}) if $\bC$ has a unique minimal (respectively maximal) length element. We call $\bC$ \textit{bi-capped} if it is both lower capped and upper capped. 

A subset $J\subseteq S$ is called \textit{$\sigma$-rigid} if $w_Jsw_J=s^{\sigma}$ for all $s\in J$, and if $wJw^{-\sigma}\subseteq S$ for some $w\in W$ then $wJw^{-\sigma}=J$. Theorem~\ref{thm:downwardsclosure} implies that the set of lower capped $\sigma$-involution classes corresponds bijectively to the set of $\sigma$-rigid subsets of $S$ (this bijection is given by $\bC\leftrightarrow J$ if and only if $w_J\in \bC$). Moreover, by Lemma~\ref{lem:pairing} it follows that the set of upper capped $\sigma$-involution classes corresponds bijectively to the set of $\sigma\sigma_0$-rigid subsets of $S$ (with the bijection given by $\bC\leftrightarrow J'$ if and only if $w_{J'}w_0\in \bC$). 

If $\bC$ is bi-capped then there exists a unique $\sigma$-rigid set $J\subseteq S$ and a unique $\sigma\sigma_0$-rigid set $J'\subseteq S$ such that $w_J,w_{J'}w_0\in \bC$.

\begin{thm}\label{thm:bicapped} Let $(W,S)$ be an irreducible spherical Coxeter system with Coxeter graph~$\Pi$, and let $\bC$ be a bi-capped class of $\sigma$-involutions. 
\begin{compactenum}[$(1)$]
\item If $J$ is the unique $\sigma$-rigid set with $w_J\in \bC$ and $J'$ is the unique $\sigma\sigma_0$-rigid set  with $w_{J'}w_0\in \bC$ then the triples $(\Pi,S\backslash J,\sigma)$ and $(\Pi,S\backslash J',\sigma\sigma_0)$ are admissible diagrams. 
\item The relative type of $(\Pi,S\backslash J,\sigma)$ has rank $\FixRk(\bC)$, and the relative type of $(\Pi,S\backslash J',\sigma\sigma_0)$ has rank~$\OppRk(\bC)=\FixRk(\psi(\bC))$. 
\item The complete list of bi-capped twisted involution classes is given in \emph{\cref{fig:bicapped}} up to the duality $\psi$ (see \emph{\cref{rem:conv}}). In the table we list the classes via their associated admissible diagram, with the class $\bC=\Cl^{\sigma}(w_J)$ identified with the admissible diagram $(\Pi,S\backslash J,\sigma)$. Moreover, we list the dual class $\psi(\bC)$, the relative type of $\bC$ and the relative type of $\psi(\bC)$ (c.f.\ \emph{\cref{defn:relativetype}}).
\end{compactenum}
\begin{table}[h!]
\renewcommand{\arraystretch}{1.5}
\begin{tabular}{c||c|c||c|c||c}
\mbox{\emph{Type}} & \mbox{\emph{Class}}&\mbox{\emph{Dual class}} & \mbox{\emph{Rel. type}}& \mbox{\emph{Dual rel. type}} &\mbox{\emph{Remarks}}\\ 
\hline\hline
\multirow{2}{*}{$\mathsf{X}_n$} &  $\Cl(1)$ & $\Cl^{\sigma_0}(w_0)$ & $\mathsf{X}_n$ & $\mathsf{X}_0$ &\multirow{2}{*}{\emph{all types $\mathsf{X}_n$}} \\
 &  $\Cl^{\sigma_0}(w_0)$ & $\Cl(1)$ & $\mathsf{X}_0$ & $\mathsf{X}_n$ & \\
\hline
\multirow{1}{*}{$\sA_n$} &  ${^1}\sA^2_{n;(n-1)/2}$ & ${^2}\sA_{n;(n+1)/2}^1$ & $\sA_{(n-1)/2}$ & $\sB_{(n+1)/2}$ & \emph{$n$ odd}\\
\hline
\multirow{2}{*}{$\sB_n$} & ${^1}\sB_{n;j}^1$ & ${^1}\sB_{n;n-j}^1$ & $\sB_j$ & $\sB_{n-j}$ & $0\leq j\leq n$\\
& ${^1}\sB_{n;n/2}^2$ & ${^1}\sB_{n;n/2}^2$ & $\sB_{n/2}$ & $\sB_{n/2}$ & \emph{$n$ even}\\
\hline
\multirow{2}{*}{$\sD_n$} & $\sD_{n;j}^1$ & $\sD_{n;n-j}^1$ & $\sB_j$ & $\sB_{n-j}$ & $0\leq j\leq n$\\
& $\sD_{n;n/2}^2$ & $\sD_{n,n'/2}^2$ & $\sB_{n/2}$ & $\sB_{n/2}$ & \emph{$n$ even}\\
\hline
\multirow{1}{*}{$\sE_6$} & $\sE_{6;2}$ & ${^2}\sE_{6;4}$ & $\sA_2$ & $\sF_4$ &\\
\hline 
\multirow{1}{*}{$\sE_7$} & $\sE_{7;3}$ & $\sE_{7;4}$ & $\sB_3$ & $\sF_4$ &\\
\hline
\multirow{1}{*}{$\sE_8$} & $\sE_{8;4}$ & $\sE_{8;4}$ & $\sF_4$ & $\sF_4$ & \\
\hline
\multirow{2}{*}{$\sF_4$} & $\sF_{4;2}$ & $\sF_{4;2}$ & $\sB_2$ & $\sB_2$ &\\
& ${^2}\sF_{4;2}$ & ${^2}\sF_{4;2}$ & $\mathsf{I}_2(8)$ & $\mathsf{I}_2(8)$ &\\
\hline
\multirow{2}{*}{$\mathsf{I}_2(2m)$} & $\mathsf{I}_{2;1}^j(2m)$ & $\mathsf{I}_{2;1}^j(2m)$ & $\sA_1$ & $\sA_1$ & \emph{$j\in\{1,2\}$}\\
& ${^2}\mathsf{I}_{2;1}(2m)$ & ${^2}\mathsf{I}_{2;1}(2m)$ & $\sA_1$ & $\sA_1$ &\\
\hline
\end{tabular}
\vspace{6pt}
\caption{Bi-capped $\sigma$-involution classes of irreducible spherical $(W,S)$}
\label{fig:bicapped}
\end{table}
\end{thm}

\begin{proof}
We first sketch the proof of (3). The strategy is as follows. First determine all $\sigma$-balanced subsets $J$ (these are easily determined from the Coxeter graph, and it turns out that for each type there are relatively few such sets). For each such subset, find a minimal length element $w_{J'}$ of the class $\psi(\bC)=\bC w_0$. The bi-capped classes are then the classes for which $J'$ is $\sigma\sigma_0$-rigid. The most complicated part of this procedure is determining a minimal length element $w_{J'}$ of $\bC w_0$. One approach is to use the tables in~\cite{GKP:00} (alternatively, in the classical types one can use concrete descriptions of the groups as permutations or signed permutations, and in the exceptional types some basic computation achieves the goal). We omit the explicit details. 

It is then a direct consequence of the classification that each triple $(\Pi,S\backslash J,\sigma)$ and $(\Pi,S\backslash J',\sigma\sigma_0)$ is admissible. We record the corresponding admissible diagrams in Table~\ref{fig:bicapped}. 
\end{proof}

\begin{remark}\label{rem:conv}
The following conventions are in force for Table~\ref{fig:bicapped}.
\begin{compactenum}[$(1)$]
\item $\mathsf{X}_0$, for $\mathsf{X}\in\{\sA,\sB,\ldots\}$, denotes the empty type.
\item We adopt the naming conventions for the diagrams of classical type from \cite[Tables 1 and 2]{PVMclass} and of exceptional type from \cite[Tables 1 and 2]{PVMexc}. The symbols for the non-crystallographic diagrams are given in Example~\ref{ex:noncrystallographic}.
\item In the type $\sD$ diagrams the twisting index is omitted, as described in the introduction to~\cite{PVMclass}. 
\item The diagram $\sD_{n;n/2}^2$ ($n$ even) has nodes $\{2,4,\ldots,n-2,n\}$ encircled, and the diagram $\sD_{n;n'/2}^2$ has nodes $\{2,4,\ldots,n-2,n-1\}$ encircled. 
\end{compactenum}
\end{remark}

\begin{example}
The $\sE_7$ row in \cref{fig:bicapped} reads $\sE_{7;3}$, $\sE_{7;4}$, $\sB_3$, $\sF_4$. This says that $\Cl(w_{\{2,3,4,5\}})$ and $\Cl(w_{\{2,5,7\}})$ are dual classes, and that the relative Coxeter groups associated to these classes have types $\sB_3$ and $\sF_4$ respectively. In diagrams, this row gives:
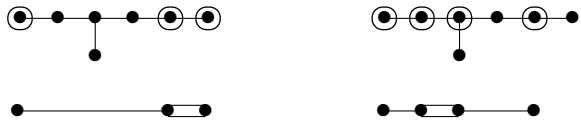
\begin{align*}
&\begin{tikzpicture}[scale=0.5]
\node at (0,0.3) {};
\node [inner sep=0.8pt,outer sep=0.8pt] at (-2,0) (1) {$\bullet$};
\node [inner sep=0.8pt,outer sep=0.8pt] at (-1,0) (3) {$\bullet$};
\node [inner sep=0.8pt,outer sep=0.8pt] at (0,0) (4) {$\bullet$};
\node [inner sep=0.8pt,outer sep=0.8pt] at (1,0) (5) {$\bullet$};
\node [inner sep=0.8pt,outer sep=0.8pt] at (2,0) (6) {$\bullet$};
\node [inner sep=0.8pt,outer sep=0.8pt] at (3,0) (7) {$\bullet$};
\node [inner sep=0.8pt,outer sep=0.8pt] at (0,-1) (2) {$\bullet$};
\draw (-2,0)--(3,0);
\draw (0,0)--(0,-1);
\draw [line width=0.5pt,line cap=round,rounded corners] (1.north west)  rectangle (1.south east);
\draw [line width=0.5pt,line cap=round,rounded corners] (6.north west)  rectangle (6.south east);
\draw [line width=0.5pt,line cap=round,rounded corners] (7.north west)  rectangle (7.south east);
\end{tikzpicture}\hspace{2cm}
\begin{tikzpicture}[scale=0.5]
\node [inner sep=0.8pt,outer sep=0.8pt] at (-2,0) (1) {$\bullet$};
\node [inner sep=0.8pt,outer sep=0.8pt] at (-1,0) (3) {$\bullet$};
\node [inner sep=0.8pt,outer sep=0.8pt] at (0,0) (4) {$\bullet$};
\node [inner sep=0.8pt,outer sep=0.8pt] at (1,0) (5) {$\bullet$};
\node [inner sep=0.8pt,outer sep=0.8pt] at (2,0) (6) {$\bullet$};
\node [inner sep=0.8pt,outer sep=0.8pt] at (3,0) (7) {$\bullet$};
\node [inner sep=0.8pt,outer sep=0.8pt] at (0,-1) (2) {$\bullet$};
\draw (-2,0)--(3,0);
\draw (0,0)--(0,-1);
\draw [line width=0.5pt,line cap=round,rounded corners] (1.north west)  rectangle (1.south east);
\draw [line width=0.5pt,line cap=round,rounded corners] (3.north west)  rectangle (3.south east);
\draw [line width=0.5pt,line cap=round,rounded corners] (4.north west)  rectangle (4.south east);
\draw [line width=0.5pt,line cap=round,rounded corners] (6.north west)  rectangle (6.south east);
\end{tikzpicture}\\
&\begin{tikzpicture}[scale=0.5]
\node at (0,0.3) {};
\node [inner sep=0.8pt,outer sep=0.8pt] at (-2,0) (1) {$\bullet$};
\node [inner sep=0.8pt,outer sep=0.8pt] at (2,0) (6) {$\bullet$};
\node [inner sep=0.8pt,outer sep=0.8pt] at (3,0) (7) {$\bullet$};
\draw (-2,0)--(2,0);
\draw (2,0.15)--(3,0.15);
\draw (2,-0.15)--(3,-0.15);
\end{tikzpicture}\hspace{2.1cm}
\begin{tikzpicture}[scale=0.5]
\node [inner sep=0.8pt,outer sep=0.8pt] at (-2,0) (1) {$\bullet$};
\node [inner sep=0.8pt,outer sep=0.8pt] at (-1,0) (3) {$\bullet$};
\node [inner sep=0.8pt,outer sep=0.8pt] at (0,0) (4) {$\bullet$};
\node [inner sep=0.8pt,outer sep=0.8pt] at (2,0) (6) {$\bullet$};
\draw (-2,0)--(-1,0);
\draw (0,0)--(2,0);
\draw (-1,-0.15)--(0,-0.15);
\draw (-1,0.15)--(0,0.15);
\end{tikzpicture}
\end{align*}
We have $\FixRk(\Cl(w_{\{2,3,4,5\}}))=3$ and $\OppRk(\Cl(w_{\{2,5,7\}}))=4$. The fact that these ranks add to the rank of $\sE_7$ is explained by Proposition~\ref{prop:ranksum}.
\end{example}

\begin{defn}\label{defn:psiondiagrams}
We shall call the admissible diagrams appearing in Figure~\ref{fig:bicapped} the \textit{bi-capped admissible diagrams}. That is, an admissible diagram is bi-capped if and only if it corresponds to a bi-capped class of $\sigma$-involutions, as described in Theorem~\ref{thm:bicapped}. Thus the duality $\psi$ restricted to the class of bi-capped classes gives rise to a duality $\psi$ on the set of all bi-capped admissible diagrams. For example, $\psi({^2}\sE_{6;4})=\sE_{6;2}$.  
\end{defn}

\subsection{Buildings} Our main references for the theory of buildings are \cite{AB:09,Tits:74}, and we assume that the reader is already acquainted with the theory. Let $\Delta$ be a building of type $(W,S)$, regarded as a simplicial complex, with chamber set $\Ch(\Delta)$ and $W$-distance function~$\delta:\Ch(\Delta)\times\Ch(\Delta)\to W$. 

A building is called \textit{thick} if $|\{D\in\Ch(\Delta)\mid \delta(C,D)=s\}|\geq 3$ for all $C\in\Ch(\Delta)$ and $s\in S$, and \textit{thin} if $|\{D\in\Ch(\Delta)\mid \delta(C,D)=s\}|=2$ for all $C\in\Ch(\Delta)$ and $s\in S$. Note that a non-thick building is \textit{not} the same as a thin building.

Let $\tau:\Delta\to 2^S$ be a fixed type map on the simplicial complex~$\Delta$.  The \textit{residue} of a simplex $\alpha$ of $\Delta$ is the set $\Res(\alpha)$ of all simplices of $\Delta$ which contain $\alpha$, together with the order relation induced by that on~$\Delta$. Then $\Res(\alpha)$ is a building whose Coxeter diagram is obtained from the Coxeter diagram $\Pi$ of $\Delta$ by removing all nodes which belong to $\tau(\alpha)$. 

The \textit{projection} onto a simplex $\alpha$ is the map $\proj_{\alpha}:\Delta\to\Res(\alpha)$ where $\proj_{\alpha}(\beta)$ is the unique simplex $\gamma$ of $\Res(\alpha)$ which is maximal subject to the property that every minimal length gallery from a chamber of $\Res(\beta)$ to $\Res(\alpha)$ ends in a chamber containing~$\gamma$.

Suppose that $\Delta$ is spherical (that is, $|W|<\infty$). Chambers $A,B\in\Ch(\Delta)$ are \textit{opposite} if they are at maximum distance in the chamber graph, or equivalently if $\delta(A,B)=w_0$. Simplices $\alpha,\beta$ of $\Delta$ are \textit{opposite} if $\tau(\beta)=w_0\tau(\alpha)w_0$ and there exist a chamber $A$ containing $\alpha$ and a chamber $B$ containing $\beta$ such that $A$ and $B$ are opposite.

We call the thick irreducible spherical buildings of rank at least $3$ with no Fano plane residues \textit{large buildings}, and those containing at least one Fano plane residue are called \textit{small buildings}. 

\emph{Generalised polygons} are the point-line geometries associated to spherical buildings of rank 2. More precisely, a \emph{generalised $d$-gon, $d\geq 2$}, is a point-line geometry for which the incidence graph has diameter $d$ and girth $2d$. The chamber set of the associated type $\mathsf{I}_2(d)$ building consists of the set of pairs $\{p,L\}$ with $p$ a point incident with a line~$L$.

\subsection{Automorphisms of spherical buildings and opposition diagrams}\label{sec:automorphisms}

Let $\Delta=(\Ch(\Delta),\delta)$ be a building of type $(W,S)$, where $\Ch(\Delta)$ is the set of chambers of the building, and $\delta:\Ch(\Delta)\times\Ch(\Delta)\to W$ is the Weyl distance function. If $\theta$ is an automorphism of $\Delta$ then there is an associated automorphism $\sigma\in\Aut(\Pi)$ of the Coxeter graph given by $\delta(C,D)=s$ if and only if $\delta(C^{\theta},D^{\theta})=s^{\sigma}$. We call $\sigma$ the \textit{companion diagram automorphism} of~$\theta$. Note that if $(W,S)$ is irreducible and spherical then $\sigma$ has order $1$, $2$ or $3$, with order $3$ only occurring for trialities of $\sD_4$. 


\begin{defn}
Let $\Delta$ be a spherical building of type $(W,S)$. An automorphism $\theta$ of $\Delta$ is called:
\begin{compactenum}[$(1)$]
\item \textit{domestic} if $\delta(C,C^{\theta})\neq w_0$ for all $C\in\Ch(\Delta)$;
\item \textit{anisotropic} if $\delta(C,C^{\theta})=w_0$ for all $C\in\Ch(\Delta)$;
\item \textit{lower capped} if whenever there exist type $J$ and type $J'$ simplices fixed by~$\theta$, there exists a type $J\cup J'$ simplex fixed by~$\theta$;
\item \textit{upper capped} (or simply \textit{capped}) if whenever there exist type $J$ and type $J'$ simplices mapped onto opposite simplices by $\theta$, there exists a type $J\cup J'$ simplex mapped onto an opposite simplex by~$\theta$;
\item \textit{bi-capped} if it is both upper and lower capped. 
\end{compactenum}
\end{defn}

In \cite{PVM:19a} we introduced the notion of cappedness (the notion of ``lower cappedness'' did not play a role in that paper, and so the property of being ``upper capped'' was simply referred to as being ``capped'', and we shall continue to use this terminology in the present paper). An automorphism that is not capped is called \textit{uncapped}. By the main result of \cite{PVM:19a} if $\Delta$ has rank at least $3$ and has no Fano plane residues then all automorphisms of $\Delta$ are capped. We note that the lower capped property does not satisfy such a statement.

If $\theta$ is anisotropic then necessarily the companion automorphism $\sigma$ is opposition, and $\disp(\theta)=\Cl^{\sigma}(w_0)=\{w_0\}$ (this is a very special case of \cite[Theorem~1.3]{DPV:13}). Thus anisotropic automorphisms are uniclass. 

Let $\theta$ be an automorphism of a spherical building $\Delta$ with companion automorphism $\sigma$. Let $\sigma_0\in\mathrm{Aut}(\Pi)$ be the automorphism of $\Pi$ induced by the longest element~$w_0$. The \textit{opposition diagram} $\Diag(\theta)$ of $\theta$ is the triple $(\Pi,J,\sigma\sigma_0)$ where $J$ is the union of all subsets $K\subseteq S$ for which there exists a type $K$ simplex of $\Delta$ mapped onto an opposite simplex by~$\theta$. 

In \cite{PVM:19a,PVM:19b} we showed that if $\theta$ is an automorphism of a thick irreducible spherical building then $\Diag(\theta)=(\Pi,J,\sigma\sigma_0)$ is an admissible diagram (c.f. Section~\ref{sec:admissible}). We emphasise that the diagram automorphism associated to this admissible diagram is $\sigma\sigma_0$ (c.f. Remark~\ref{rem:change}). For example, the admissible diagrams
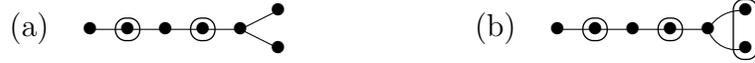
\begin{center}
(a)\quad \begin{tikzpicture}[scale=0.5,baseline=-0.5ex]
\node at (1,0.8) {};
\node [inner sep=0.8pt,outer sep=0.8pt] at (0,0) (0) {$\bullet$};
\node [inner sep=0.8pt,outer sep=0.8pt] at (1,0) (1) {$\bullet$};
\node [inner sep=0.8pt,outer sep=0.8pt] at (2,0) (2) {$\bullet$};
\node [inner sep=0.8pt,outer sep=0.8pt] at (3,0) (3) {$\bullet$};
\node [inner sep=0.8pt,outer sep=0.8pt] at (4,0) (4) {$\bullet$};
\node [inner sep=0.8pt,outer sep=0.8pt] at (5,0.5) (5a) {$\bullet$};
\node [inner sep=0.8pt,outer sep=0.8pt] at (5,-0.5) (5b) {$\bullet$};
\draw (0,0)--(4,0);
\draw (4,0) to   (5,0.5);
\draw (4,0) to   (5,-0.5);
\draw [line width=0.5pt,line cap=round,rounded corners] (1.north west)  rectangle (1.south east);
\draw [line width=0.5pt,line cap=round,rounded corners] (3.north west)  rectangle (3.south east);
\end{tikzpicture}\qquad\qquad\qquad (b)\quad \begin{tikzpicture}[scale=0.5,baseline=-0.5ex]
\node at (0,0.8) {};
\node at (0,-0.8) {};
\node [inner sep=0.8pt,outer sep=0.8pt] at (0,0) (0) {$\bullet$};
\node [inner sep=0.8pt,outer sep=0.8pt] at (1,0) (1) {$\bullet$};
\node [inner sep=0.8pt,outer sep=0.8pt] at (2,0) (2) {$\bullet$};
\node [inner sep=0.8pt,outer sep=0.8pt] at (3,0) (3) {$\bullet$};
\node [inner sep=0.8pt,outer sep=0.8pt] at (4,0) (4) {$\bullet$};
\node [inner sep=0.8pt,outer sep=0.8pt] at (5,0.5) (5a) {$\bullet$};
\node [inner sep=0.8pt,outer sep=0.8pt] at (5,-0.5) (5b) {$\bullet$};
\draw (0,0)--(4,0);
\draw (4,0) to [bend left] (5,0.5);
\draw (4,0) to [bend right=45] (5,-0.5);
\draw [line width=0.5pt,line cap=round,rounded corners] (1.north west)  rectangle (1.south east);
\draw [line width=0.5pt,line cap=round,rounded corners] (3.north west)  rectangle (3.south east);
\draw [line width=0.5pt,line cap=round,rounded corners] (5a.north west)  rectangle (5b.south east);
\end{tikzpicture}
\end{center}
are the opposition diagrams of: (a) a non-type preserving automorphism of a $\sD_7$ building mapping type $2$ and type $4$ vertices to opposite vertices, and (b) a type preserving automorphism of a $\sD_7$ building mapping type $2$ and $4$ vertices onto opposite, and type $\{6,7\}$ simplices onto opposite simplices (recall that the opposition relation $\sigma_0$ on $\sD_n$ is type preserving if $n$ is even, and interchanges types $n-1$ and $n$ if $n$ is odd). 

Note that if $\Delta$ is capped then the opposition diagram completely determines the partially ordered set of types of simplices that are mapped to opposite simplices by~$\theta$. For uncapped automorphisms some additional decorations are required on the opposition diagram to capture the structure of this poset. In \cite{PVM:19b} we introduced these decorated opposition diagrams for uncapped automorphisms (of  necessarily small buildings). This diagram is obtained from the opposition diagram of the automorphism as follows. If $J$ denotes the set of encircled nodes of the opposition diagram, we shade those distinguished orbits $J'\subseteq J$ with the property that there exists a type $J\backslash J'$ simplex mapped onto an opposite simplex by $\theta$. For example, the diagram \begin{tikzpicture}[scale=0.5,baseline=-0.5ex]
\node at (0,0.3) {};
\node [inner sep=0.8pt,outer sep=0.8pt] at (-1.5,0) (1) {$\bullet$};
\node [inner sep=0.8pt,outer sep=0.8pt] at (-0.5,0) (2) {$\bullet$};
\node [inner sep=0.8pt,outer sep=0.8pt] at (0.5,0) (3) {$\bullet$};
\draw [line width=0.5pt,line cap=round,rounded corners,fill=ggrey] (1.north west)  rectangle (1.south east);
\draw [line width=0.5pt,line cap=round,rounded corners,fill=ggrey] (2.north west)  rectangle (2.south east);
\draw [line width=0.5pt,line cap=round,rounded corners] (3.north west)  rectangle (3.south east);
\draw (-1.5,0)--(-0.5,0);
\draw (-0.5,0.07)--(0.5,0.07);
\draw (-0.5,-0.07)--(0.5,-0.07);
\node [inner sep=0.8pt,outer sep=0.8pt] at (-1.5,0) (1) {$\bullet$};
\node [inner sep=0.8pt,outer sep=0.8pt] at (-0.5,0) (2) {$\bullet$};
\node [inner sep=0.8pt,outer sep=0.8pt] at (0.5,0) (3) {$\bullet$};
\end{tikzpicture} is the decorated opposition diagram of an uncapped automorphism of a $\sB_3$ building mapping vertices of each type to opposite vertices (as each node is encircled), and also mapping simplices of types $\{2,3\}$ and $\{1,3\}$ onto opposite simplices (because the first and second nodes are shaded), yet no simplex of type $\{1,2\}$ is mapped onto an opposite (because the third node is not shaded). 

We refer the reader to \cite{PVM:19b} for the complete list of possible decorated opposition diagrams of uncapped automorphisms. This list is very restricted, for example the only possible uncapped automorphisms of the building $\sE_7(2)$ have diagrams
 $$
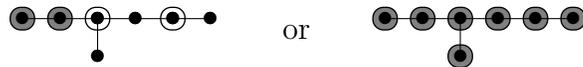
\begin{tikzpicture}[scale=0.5,baseline=-1.5ex]
\node [inner sep=0.8pt,outer sep=0.8pt] at (-2,0) (1) {$\bullet$};
\node [inner sep=0.8pt,outer sep=0.8pt] at (-1,0) (3) {$\bullet$};
\node [inner sep=0.8pt,outer sep=0.8pt] at (0,0) (4) {$\bullet$};
\node [inner sep=0.8pt,outer sep=0.8pt] at (1,0) (5) {$\bullet$};
\node [inner sep=0.8pt,outer sep=0.8pt] at (2,0) (6) {$\bullet$};
\node [inner sep=0.8pt,outer sep=0.8pt] at (3,0) (7) {$\bullet$};
\node [inner sep=0.8pt,outer sep=0.8pt] at (0,-1) (2) {$\bullet$};
\draw [line width=0.5pt,line cap=round,rounded corners,fill=ggrey] (1.north west)  rectangle (1.south east);
\draw [line width=0.5pt,line cap=round,rounded corners,fill=ggrey] (3.north west)  rectangle (3.south east);
\draw [line width=0.5pt,line cap=round,rounded corners] (4.north west)  rectangle (4.south east);
\draw [line width=0.5pt,line cap=round,rounded corners] (6.north west)  rectangle (6.south east);
\node at (0,0) {$\bullet$};
\node at (2,0) {$\bullet$};
\node at (0,-1.3) {};
\node at (0,0.3) {};
\node [inner sep=0.8pt,outer sep=0.8pt] at (-2,0) (1) {$\bullet$};
\node [inner sep=0.8pt,outer sep=0.8pt] at (-1,0) (3) {$\bullet$};
\node [inner sep=0.8pt,outer sep=0.8pt] at (0,0) (4) {$\bullet$};
\node [inner sep=0.8pt,outer sep=0.8pt] at (1,0) (5) {$\bullet$};
\node [inner sep=0.8pt,outer sep=0.8pt] at (2,0) (6) {$\bullet$};
\node [inner sep=0.8pt,outer sep=0.8pt] at (3,0) (7) {$\bullet$};
\node [inner sep=0.8pt,outer sep=0.8pt] at (0,-1) (2) {$\bullet$};
\draw (-2,0)--(3,0);
\draw (0,0)--(0,-1);
\end{tikzpicture}
\qquad \text{or}\qquad 
\begin{tikzpicture}[scale=0.5,baseline=-1.5ex]
\node [inner sep=0.8pt,outer sep=0.8pt] at (-2,0) (1) {$\bullet$};
\node [inner sep=0.8pt,outer sep=0.8pt] at (-1,0) (3) {$\bullet$};
\node [inner sep=0.8pt,outer sep=0.8pt] at (0,0) (4) {$\bullet$};
\node [inner sep=0.8pt,outer sep=0.8pt] at (1,0) (5) {$\bullet$};
\node [inner sep=0.8pt,outer sep=0.8pt] at (2,0) (6) {$\bullet$};
\node [inner sep=0.8pt,outer sep=0.8pt] at (3,0) (7) {$\bullet$};
\node [inner sep=0.8pt,outer sep=0.8pt] at (0,-1) (2) {$\bullet$};
\draw [line width=0.5pt,line cap=round,rounded corners,fill=ggrey] (1.north west)  rectangle (1.south east);
\draw [line width=0.5pt,line cap=round,rounded corners,fill=ggrey] (3.north west)  rectangle (3.south east);
\draw [line width=0.5pt,line cap=round,rounded corners,fill=ggrey] (4.north west)  rectangle (4.south east);
\draw [line width=0.5pt,line cap=round,rounded corners,fill=ggrey] (6.north west)  rectangle (6.south east);
\draw [line width=0.5pt,line cap=round,rounded corners,fill=ggrey] (2.north west)  rectangle (2.south east);
\draw [line width=0.5pt,line cap=round,rounded corners,fill=ggrey] (5.north west)  rectangle (5.south east);
\draw [line width=0.5pt,line cap=round,rounded corners,fill=ggrey] (7.north west)  rectangle (7.south east);
\node [inner sep=0.8pt,outer sep=0.8pt] at (-2,0) (1) {$\bullet$};
\node [inner sep=0.8pt,outer sep=0.8pt] at (-1,0) (3) {$\bullet$};
\node [inner sep=0.8pt,outer sep=0.8pt] at (0,0) (4) {$\bullet$};
\node [inner sep=0.8pt,outer sep=0.8pt] at (1,0) (5) {$\bullet$};
\node [inner sep=0.8pt,outer sep=0.8pt] at (2,0) (6) {$\bullet$};
\node [inner sep=0.8pt,outer sep=0.8pt] at (3,0) (7) {$\bullet$};
\node [inner sep=0.8pt,outer sep=0.8pt] at (0,-1) (2) {$\bullet$};
\draw (-2,0)--(3,0);
\draw (0,0)--(0,-1);
\end{tikzpicture} 
$$
Using Bourbaki labelling (so that the node on the short arm is labelled ``$2$''), the first diagram means that there are vertices of types $1,3,4$ and $6$ mapped onto opposite vertices, and simplices of types $\{3,4,6\}$ and $\{1,4,6\}$ mapped onto opposite simplices, but no simplices of types $\{1,3,6\}$ or $\{1,3,4\}$ are mapped onto opposite simplices. 

The following proposition is a refinement of \cite[Proposition~4.2]{AB:09}.

\begin{prop}\label{prop:attaindisplacement}
Let $\Delta$ be a thick building of spherical type $(W,S)$ and let $\theta$ be an automorphism of $\Delta$. If $J$ is a maximal element in the partially ordered set of all types of simplices that are mapped onto opposite simplices then there exists a chamber $C\in\Ch(\Delta)$ with $\delta(C,C^{\theta})=w_{S\backslash J}w_0$. 
\end{prop}

\begin{proof}
Let $C\in\Ch(\Delta)$ be such that $w=\delta(C,C^{\theta})$ is of maximal length subject to the condition that the type $J$-simplex of $C$ is mapped onto an opposite simplex. Thus $w=vw_0$ for some $v\in W_{S\backslash J}$ (see \cite[Lemma~2.5]{PVM:19a}). We claim that if $s\in S$ with $sws^{\sigma}\neq w$ then $\ell(sw)<\ell(w)$. For if $\ell(sw)>\ell(w)$ then:
\begin{compactenum}[$(a)$]
\item If $\ell(ws^{\sigma})>\ell(w)$ then since $\ell(ws^{\sigma})=\ell(sw)$ and $sw\neq ws^{\sigma}$ we have (by Lemma~\ref{lem:foldingcondition}) $\ell(sws^{\sigma})=\ell(w)+2$, and so if $D\sim_s C$ we have $\delta(D,D^{\theta})=sws^{\sigma}=svs^{\sigma w_0}w_0$. We have $J^{\sigma w_0}=J$ (by virtue of the fact that there is a simplex of this type mapped onto an opposite simplex). Thus if $s\in S\backslash J$ then $s^{\sigma w_0}\in S\backslash J$ and hence $svs^{\sigma w_0}\in W_{S\backslash J}$, and so the type $J$ simplex of $D$ is mapped onto an opposite simplex, contradicting maximality of $\ell(w)$. So $s\in J$, but then $\ell(sv)>\ell(v)$ and so $\ell(svs^{\sigma w_0})\geq \ell(v)$, contradicting the fact that $\ell(sws^{\sigma})=\ell(w)+2$. 
\item If $\ell(ws^{\sigma})<\ell(w)$ then there is a unique chamber $E\sim_{s^{\sigma}} C^{\theta}$ such that $\delta(C,E)=ws^{\sigma}$, and by thickness we can choose $D\sim_s C$ with $D\neq E^{\theta^{-1}}$. For any such $D$ we have $\delta(D,D^{\theta})=sw$. If $s\in S\backslash J$ then since $\delta(D,D^{\theta})=(sv)w_0$ we see that the type $J$ simplex of $D$ is mapped onto an opposite simplex, contradicting the maximality of $\ell(w)$. Thus $s\in J$, but then $\ell(sv)=\ell(v)+1$, contradicting the hypothesis $\ell(sw)>\ell(w)$. Hence the claim.
\end{compactenum} 
Let $K=\{s\in S\mid \ell(sw)>\ell(w)\}$. The above claim show that for all $s\in K$ we have $sw=ws^{\sigma}$. Then by \cite[Lemma~2.4]{AB:09} we have $w=w_Kw_0$, and so $v=w_K$. By the maximality of $J$ it follows that every reduced expression for $v$ contains all generators of $S\backslash J$ (for otherwise there is a simplex of larger type mapped onto an opposite simplex), and so $K=S\backslash J$, and the proof is complete. 
\end{proof}

\begin{remark}\label{rem:elementsofdisp}
Proposition~\ref{prop:attaindisplacement} allows one to immediately write down an element of $\disp(\theta)$ directly from the opposition diagram (in the capped case) or the decorated opposition diagram (in the uncapped case). If $\theta$ is capped then there is a unique maximal element $J$ in the partially ordered set of all types of simplices mapped onto opposite simplices, and $J$ is the set of all encircled nodes in the opposition diagram. For example, in the second $\sD_7$ diagram above (with $3$ orbits encircled) we have $s_1s_3s_5w_0\in\disp(\theta)$.

If $\theta$ is uncapped then the maximal elements $J$ are the sets $K\backslash\{j\}$ where $K$ is the set of all encircled nodes in the diagram, and $j$ is a shaded node. For example, in the first uncapped $\sE_7$ diagram listed above (with $4$ encircled nodes and $2$ shaded nodes) we have $s_1s_2s_5s_7w_0\in\disp(\theta)$ and $s_2s_3s_5s_7w_0\in \disp(\theta)$, while in the second uncapped $\sE_7$ diagram we have $s_iw_0\in\disp(\theta)$ for all $1\leq i\leq 7$. 
\end{remark}

\section{Displacement spectra}\label{sec:2}

This section develops fundamental properties of the set $\disp(\theta)$, and sets up the tools that will be applied in the classification results of the following sections. While the classification results in later sections are restricted to buildings of spherical type, we set up some of the machinery in this section for arbitrary Coxeter type. We note that $\disp(\theta)$ has also been studied in~\cite{Wer:15}, where the focus was mainly on buildings of infinite type.

\subsection{Fundamental properties of displacement spectra}

The following proposition shows that the displacement spectra $\disp(\theta)$ necessarily has an ``upwards $\sigma$-conjugacy closure'' property, and under additional assumptions $\disp(\theta)$ is closed under $\sigma$-conjugation. We note that part (2) of the proposition is also contained in \cite{AB:09}.

\begin{prop}\label{prop:basic}
Let $\Delta$ be a building (not necessarily thick) of arbitrary type, and let $\theta$ be an automorphism of~$\Delta$ with diagram automorphism~$\sigma$.
\begin{compactenum}[$(1)$]
\item If $\theta$ is an involution then each $w\in\disp(\theta)$ is a $\sigma$-involution.
\item If $w\in\disp(\theta)$ and $s\in S$ with $\ell(sws^{\sigma})=\ell(w)+2$ then $sws^{\sigma}\in\disp(\theta)$.
\item If $\disp(\theta)$ consists of $\sigma$-involutions in $W$ then $\disp(\theta)$ is a union of $\sigma$-conjugacy classes in~$W$. 
\end{compactenum}
\end{prop}

\begin{proof}
(1) If $\theta$ is an involution then necessarily $\sigma^2=1$. For $C\in\Delta$ we have
$$
\delta(C,C^{\theta})^{-1}=\delta(C^{\theta},C)=\delta(C^{\theta^2},C^{\theta})^{\sigma^{-1}}
$$
(with the first equality a building axiom, and the second coming from applying $\theta$), and the result follows.

(2) Let $w\in \disp(\theta)$ and $s\in S$ with $\ell(sws^{\sigma})=\ell(w)+2$. Let $C\in \Delta$ with $\delta(C,C^{\theta})=w$, and let $D$ be any chamber with $\delta(C,D)=s$. Then $\delta(C^{\theta},D^{\theta})=s^{\sigma}$, and since $\ell(sws^{\sigma})=\ell(w)+2$ we have $\delta(D,D^{\theta})=sws^{\sigma}$. Thus $sws^{\sigma}\in \disp(\theta)$.

(3) It suffices to show that $\disp(\theta)$ is closed under $\sigma$-conjugation, and for this it suffices to show that if $w\in\disp(\theta)$ and $s\in S$ then $sws^{\sigma}\in\disp(\theta)$. If $\ell(sws^{\sigma})=\ell(w)+2$ then $sws^{\sigma}\in\disp(\theta)$ by~(2). If $\ell(sws^{\sigma})=\ell(w)$ then since $w$ is a $\sigma$-involution we have $(sw)^{-1}=w^{-1}s=w^{\sigma}s=(ws^{\sigma})^{\sigma}$ (as $\sigma$ has order $1$ or $2$) and so $\ell(sw)=\ell(ws^{\sigma})$ (as both inversion and application of $\sigma$ preserve lengths). Hence $sws^{\sigma}=w\in\disp(\theta)$ by Lemma~\ref{lem:foldingcondition}.

Suppose that $\ell(sws^{\sigma})=\ell(w)-2$. Write $w=svs^{\sigma}$, so that $\ell(svs^{\sigma})=\ell(v)+2$. Let 
$$C=C_0\sim_s C_1\sim\cdots \sim C_{n-1}\sim_{s^{\sigma}} C_n=C^{\theta}$$ be a (necessarily minimal length) gallery from $C$ to $C^{\theta}$ of type $svs^{\sigma}$. Since $\delta(C,C_1)=s$ we have $\delta(C^{\theta},C_1^{\theta})=s^{\sigma}$, and so either $\delta(C_1,C_1^{\theta})=sws^{\sigma}$ (in the case that $C_1^{\theta}=C_{n-1}$) or $\delta(C_1,C_1^{\theta})=sw$ (in the case that $C_1^{\theta}\neq C_{n-1}$). However we claim that $sw$ is not a $\sigma$-involution (which will eliminate the second case). To see this, note that if $(sw)^{-1}=(sw)^{\sigma}$ then $w^{-1}=s^{\sigma}w^{\sigma}s$ and so $w^{\sigma}=s^{\sigma}w^{\sigma}s$, and applying $\sigma$ (and using $\sigma^2=1$) gives $w=sws^{\sigma}$, contradicting $\ell(sws^{\sigma})=\ell(w)-2$. 
\end{proof}

\begin{example}\label{ex:SL3}
We provide an example of an automorphism whose displacement is not a union of $\sigma$-conjugacy classes. Let $G=\mathsf{SL}_3(\FF)$, with $\FF$ any field, and let $B$ the the subgroup upper triangular matrices in~$G$. Let $\Delta=G/B$ be the associated building of type $\sA_2$ (a projective plane). Let $a\in\FF$, and suppose that the polynomial $p(X)=X^3+aX^2-1$ is irreducible over~$\FF$. Consider the type preserving automorphism of $\Delta$ given by the matrix
$$
\theta=\begin{bmatrix}-a&0&1\\
-1&0&0\\
0&-1&0\end{bmatrix}
$$
(in Chevalley generators this element is $\theta=x_{\alpha_1}(a)s_1s_2$, where $s_i=x_{\alpha_i}(1)x_{-\alpha_i}(-1)x_{\alpha_i}(1)$). Irreducibility of $p(X)$ implies that $\theta$ has no fixed points and no fixed lines in the projective plane. It follows that $e,s_1,s_2\notin\disp(\theta)$. On the other hand, $\delta(B,\theta B)=s_1s_2$, $\delta(s_1B,\theta s_1B)=w_0$, and $\delta(s_2B,\theta s_2 B)=s_2s_1$, showing that $\disp(\theta)=\{s_1s_2,s_2s_1,w_0\}$. 
\end{example}

The following proposition is of independent interest.

\begin{prop}\label{prop:containinvolution}
Let $\theta$ be an automorphism of a thick spherical building. Then $\disp(\theta)$ contains an involution and a $\sigma$-involution.
\end{prop}

\begin{proof}
If $\theta$ is capped, and if $J$ denotes the set of all nodes encircled in the opposition diagram, then $w_0w_{S\backslash J}\in\disp(\theta)$. Since $J$ is stable under both opposition and $\sigma$ (for general reasons, but also easily checked from the diagrams), the set $S\backslash J$ is also stable under $\sigma$ and opposition. It follows that $w_0w_{S\backslash J}\in\disp(\theta)$ is an involution that is also a $\sigma$-involution.

Suppose now that $\theta$ is uncapped. Let $J$ denote the set of encircled nodes in the (decorated) opposition diagram, and let $K$ denote the set of shaded nodes. For each $k\in K$ the element $w_0w_{S\backslash(J\backslash \{k\})}$ lies in $\disp(\theta)$ by Proposition~\ref{prop:attaindisplacement}. By inspection, with the exception of exceptional domestic dualities of $\sA_{2n}$ with $n\geq 1$, there exists $k\in K$ such that $J\backslash\{k\}$ is stable under both $\sigma$ and opposition, and hence $\disp(\theta)$ contains an involution that is also a $\sigma$-involution (for example, for domestic dualities of $\sA_{2n-1}$, take $k=n$). 

Consider the excluded case of exceptional (hence strongly exceptional) domestic dualities of $\sA_{2n}$. Let $s=s_n$ and $t=s_{n+1}$ (so that $s^{\sigma}=t$). The element $w=w_0s$ lies in $\disp(\theta)$. This element is a $\sigma$-involution (but not an involution), and so it remains to prove that $\disp(\theta)$ contains an involution. Since $sws^{\sigma}=stsw_0$ we have $\ell(sws^{\sigma})=\ell(w)-2$, and so an argument as in Proposition~\ref{prop:full} shows that either $stsw_0\in\disp(\theta)$, or $sw=sw_0s=stw_0\in\disp(\theta)$. Both of these elements are involutions, completing the proof. 
\end{proof}

\subsection{Uniclass automorphisms}

Let $(W,S)$ be an arbitrary Coxeter system.

\begin{defn}
An automorphism $\theta$ of of a building with companion diagram automorphism~$\sigma$ is called \textit{uniclass} if $\disp(\theta)$ is contained in a single $\sigma$-conjugacy class. 
\end{defn}

\begin{example}\label{ex:thincase}
If $\Delta$ is thin then every automorphism is uniclass. To see this, recall that a thin building of type $(W,S)$ is isomorphic to the Coxeter complex of $(W,S)$, and hence we may take $\Ch(\Delta)=W$ and $\delta(u,v)=u^{-1}v$. If $\theta$ is an automorphism of $\Delta$ then (by thinness) we have $\delta(u,u^{\theta})=\delta(u,1)\delta(1,1^{\theta})\delta(1^{\theta},1^{\theta})=u^{-1}wu^{\sigma}$ where $w=\delta(1,1^{\theta})$, and hence $\disp(\theta)=\Cl^{\sigma}(w)$. 
\end{example}

In contrast, for general buildings the property of being uniclass is very rare (this is quantified by Theorem~\ref{thm:main1}). We note the following basic facts. 

\begin{lemma} Let $\theta$ be a uniclass automorphism with companion automorphism~$\sigma$.
\begin{compactenum}[$(1)$]
\item If $\sigma$ is the identity then $\theta$ is either the identity, or $\theta$ fixes no chamber.
\item If $(W,S)$ is spherical and $\sigma$ is the opposition relation then either $\theta$ is anisotropic, or $\theta$ is domestic. 
\item If $(W,S)$ is spherical and both the opposition relation and $\sigma$ are trivial then $\theta$ is either the identity, anisotropic, or is domestic with no fixed chamber.
\end{compactenum}
\end{lemma}

\begin{proof}
(1) In this case $\{1\}$ is a $\sigma$-conjugacy class. (2) In this case $\{w_0\}$ is a $\sigma$-conjugacy class. (3) In this case both (1) and (2) apply.
\end{proof}

Example~\ref{ex:SL3} shows that for general $\theta$ the set $\disp(\theta)$ is not necessarily equal to a union of $\sigma$-conjugacy classes. The following proposition shows that the behaviour of uniclass automorphisms is more regular.

\begin{prop}\label{prop:full}
Let $(W,S)$ be an arbitrary Coxeter system. If $\theta$ is \uniclass, then $\disp(\theta)$ is a full $\sigma$-conjugacy class. Moreover, if $(W,S)$ is spherical and $\sigma$ has order $1$ or $2$ then $\disp(\theta)$ consists of $\sigma$-involutions.
\end{prop}

\begin{proof}
We show that $\disp(\theta)$ is closed under conjugation. It suffices to show that if $w=\delta(C,C^{\theta})\in\disp(\theta)$ and $s\in S$ then $sws^{\sigma}\in \disp(\theta)$. If $\ell(sws^{\sigma})=\ell(w)+2$ then apply Proposition~\ref{prop:basic}(2). Suppose that $\ell(sws^{\sigma})=\ell(w)-2$. Thus there is a reduced gallery from $C$ to $C^{\theta}$ with $C\sim_sC_1\sim\cdots\sim C_{n-1}\sim_{s^{\sigma}}C^{\theta}$. The \uniclass\ assumption forces $C_1^{\theta}=C_{n-1}$ (otherwise $\delta(C_1,C_1^{\theta})=sw$, which is not in the same $\sigma$-class by parity of length, see Proposition~\ref{prop:basic1}), and hence $sws^{\sigma}\in\disp(\theta)$. 

Now suppose that $\ell(sws^{\sigma})=\ell(w)$. There are two cases to consider. Suppose first that $\ell(sw)=\ell(w)-1$. Thus there is a reduced gallery from $C$ to $C^{\theta}$ starting with $C\sim_s C_1$. Since $C^{\theta}\sim_{s^{\sigma}}C_1^{\theta}$ and $\ell(sws^{\sigma})=\ell(sw)+1$ we have $\delta(C_1,C_1^{\theta})=sws^{\sigma}$ and so $sws^{\sigma}\in\disp(\theta)$. Suppose now that $\ell(sw)=\ell(w)+1$. In this case there is a reduced gallery from $C$ to $C^{\theta}$ ending with $C_{n-1}\sim_{s^{\sigma}}C^{\theta}$. Since $C\sim_s C_{n-1}^{\theta^{-1}}$ and $\ell(sws^{\sigma})=\ell(ws^{\sigma})+1$ we have $\delta(C_1^{\theta^{-1}},C_{n-1})=sws^{\sigma}$, and so again $sws^{\sigma}\in\disp(\theta)$. 

Thus we have shown that $\disp(\theta)$ is a full $\sigma$-conjugacy class. By Proposition~\ref{prop:containinvolution} there exists a $\sigma$-involution in this $\sigma$-conjugacy class, and if $\sigma$ has order $1$ or $2$ then Lemma~\ref{lem:sigmainvolutions} gives that every element of $\disp(\theta)$ is a $\sigma$-involution. 
\end{proof}

\begin{remark}
Following from Remark~\ref{rem:elementsofdisp}, note that if $\theta$ is uniclass then the $\sigma$-class $\disp(\theta)$ is completely determined from the opposition diagram (in the case of capped automorphisms) or the decorated opposition diagram (in the case of uncapped automorphisms).
\end{remark}

The following proposition allows us to restrict attention the case that $\Delta$ is irreducible.

\begin{prop}\label{prop:reduciblecase}
An automorphism of a (not necessarily irreducible) building is uniclass if and only if it preserves each component and is uniclass on each component. 
\end{prop}

\begin{proof}
For each $1\leq j\leq k$ let $\Delta_j=(\Ch(\Delta_j),\delta_j)$ be a building with Weyl group $W_j$ and let $\Delta=(\Ch(\Delta),\delta)$ be the building with $\Delta=\Delta_1\times\cdots\times\Delta_k$ with $\delta(C,D)=\delta_1(C_1,D_1)\cdots\delta_k(C_k,D_k)\in W_1\times\cdots W_k$, where $C=(C_1,\ldots,C_k)$ and $D=(D_1,\ldots,D_k)$.

It is clear that if $\theta$ preserves each component of $\Delta$, and is uniclass on each component, then $\theta$ is uniclass on $\Delta$. On the other hand, suppose that $\theta$ is uniclass on $\Delta$, and assume (for a contradiction) that $\theta$ does not preserve each component. If $\theta$ preserves a sub-product $\Delta_{i_1}\times\cdots\times\Delta_{i_r}$ then $\theta$ is uniclass on this sub-product, and so up to taking a sub-product of $\Delta$ and relabelling the components we may assume that $\theta(\Delta_j)=\Delta_{j+1}$ for $1\leq j<k$ and $\theta(\Delta_k)=\Delta_1$. 

Let $C_1\in\Ch(\Delta_1)$ be any chamber of $\Delta_1$, and consider the chamber $C=(C_1,C_1^{\theta},C^{\theta^2},\ldots,C^{\theta^{k-1}})$ of $\Delta$. We have
$
C^{\theta}=(D_1,C_1^{\theta},\ldots,C_1^{\theta^{k-1}})
$ where $D_1=C_1^{\theta^k}$ (a chamber of $\Delta_1$), and so $\delta(C,C^{\theta})=\delta_1(C_1,D_1)\in W_1$. Now consider a chamber $C'=(C_1,C_2,\ldots,C_k)$ where $C_j\in\Ch(\Delta_j)\backslash\{C_1^{\theta^{j-1}}\}$ for $2\leq j\leq k$. Then 
$
C'^{\theta}=(C_k^{\theta},C_1^{\theta},C_2^{\theta},\ldots,C_{k-1}^{\theta})
$
and so 
$$
\delta(C',C'^{\theta})=\delta_1(C_1,C_k^{\theta})\delta_2(C_2,C_1^{\theta})\cdots\delta(C_k,C_{k-1}^{\theta}). 
$$
Since $C_j\neq C_{j-1}^{\theta}=C_1^{\theta^{j-1}}$ for $2\leq j\leq k$ we have $\delta(C',C'^{\theta})\notin W_1$. Combined with the fact that $\delta(C,C^{\theta})\in W_1$ we obtain a contradiction with the assumption that $\theta$ is uniclass (because if $\sigma$ is the companion automorphism of $\theta$ then $\sigma$-conjugacy classes in $W$ are products of $\sigma$-conjugacy classes in the components). 
\end{proof}

The following proposition is particularly useful in the finite case, as it places severe restrictions on uniclass automorphisms of finite spherical buildings due to the rarity of anisotropic automorphisms of such buildings (see \cite[Theorem~5.1]{DPV:13}).

\begin{prop}\label{prop:anisotropic}
Let $\theta$ be a nontrivial uniclass automorphism of a thick spherical building with companion automorphism $\sigma$ with $\sigma^2=1$. Then either $\theta$ fixes a chamber, or is anisotropic, or there is a proper residue $R\subseteq \Delta$ stabilised by~$\theta$ and the automorphism $\theta|_R:R\to R$ is anisotropic. 
\end{prop}

\begin{proof}
By Proposition~\ref{prop:full} $\disp(\theta)$ is a $\sigma$-conjugacy class consisting of $\sigma$-involutions. Thus by Theorem~\ref{thm:downwardsclosure} there is a subset $J\subseteq S$ with $sw_J=w_Js^{\sigma}$ for all $s\in J$ such that $w_J$ is a minimal length element in $\disp(\theta)$. If $J=\emptyset$ then $1\in\disp(\theta)$ and so $\theta$ fixes a chamber. Suppose that $J\neq\emptyset$. Let $C_0$ be a chamber with $\delta(C_0,C_0^{\theta})=w_J$, and let $R=\{C\in\Ch(\Delta)\mid \delta(C_0,C)\in W_J\}$ be the $J$-residue of $C_0$. Since the type $S\backslash J$ simplex of $C_0$ is fixed pointwise by $\theta$ the residue $R$ is stabilised by~$\theta$, and
$$
\disp(\theta|_R)=\{\delta(C,C^{\theta})\mid C\in R\}\subseteq \disp(\theta)\cap W_J=\{w_J\},
$$
where the final equality is because $w_J$ has minimal length in $\disp(\theta)$. Since $C_0\in R$ it follows that $\disp(\theta|_R)=\{w_J\}$, and so $\theta|_R$ is anisotropic. If $J\neq S$ then $R$ is a proper residue, and if $J=S$ then $R=\Delta$ and $\theta|_R=\theta$ is isotropic.
\end{proof}

The following proposition shows that uniclass automorphisms enjoy a useful residual property.

\begin{prop}\label{prop:residual}
Let $\theta$ be a uniclass automorphism of a building with companion automorphism $\sigma$. Suppose there is a subset $J\subseteq S$ with $J^{\sigma}=J$ and a chamber $C\in\Ch(\Delta)$ with $\delta(C,C^{\theta})\in W_J$ (that is, the type $S\backslash J$ simplex of $C$ is fixed by $\theta$). Let $R=\Res_J(C)$. Then the restriction $\theta|_R:R\to R$ is uniclass. 
\end{prop}

\begin{proof}
Let $w=\delta(C,C^{\theta})$. It suffices to show $\delta(D,D^{\theta})$ is conjugate in $W_J$ to $w$ for all chambers $D\in R$. By induction is suffices to consider the case when $D$ is adjacent to $C$, and so $\delta(C,D)=s\in J$. Since $\theta$ is uniclass we have $\delta(D,D^{\theta})\in\{w,sws^{\sigma}\}$ (because $sw$ and $ws^{\sigma}$ are not in the same $\sigma$-class as $w$ by parity of lengths). Hence the result. 
\end{proof}

We conclude this subsection with two lemmas that will be useful later. Recall that the \textit{convex hull} of chambers $C$ and $D$ of a building is the set $\mathsf{conv}\{C,D\}$ consisting of all chambers that lie on a minimal length gallery from $C$ to $D$.

\begin{lemma}\label{lem:determineimages}
Let $\theta$ be a uniclass automorphism with companion automorphism~$\sigma$. Let $C$ be any chamber, and let $\delta(C,C^{\theta})=w$. Suppose that $u\in W$ with $\ell(u^{-1}wu^{\sigma})=\ell(w)-2\ell(u)$. Then there exist unique chambers $D,D'\in\mathsf{conv}\{C,C^{\theta}\}$ with $\delta(C,D)=u$ and $\delta(C^{\theta},D')=u^{\sigma}$, and $D^{\theta}=D'$.
\end{lemma}

\begin{proof}
Let $u=s_1\cdots s_m$ and $u^{-1}wu^{\sigma}=t_1\cdots t_n$ be reduced expressions. The condition $\ell(u^{-1}wu^{\sigma})=\ell(w)-2\ell(u)$ implies that 
$
w=s_1\cdots s_mt_1\cdots t_ns_m^{\sigma}\cdots s_1^{\sigma}
$
is a reduced expression, and hence there exists a gallery $C=C_0\sim C_1\sim\cdots \sim C_{2m+n}=C^{\theta}$ of this reduced type from $C$ to $C^{\theta}$. Then $D=C_m$ and $D'=C_{n+m}$ are the unique chambers of $\mathsf{conv}\{C,C^{\theta}\}$ with $\delta(C,D)=u$ and $\delta(C^{\theta},D')=u^{\sigma}$. We claim that $C_k^{\theta}=C_{2m+n-k}$ for $k=0,1,\ldots,m$. The case $k=0$ is true by assumption, and the claim follows by inductively applying the following observation: Suppose that $\theta$ is \uniclass, and that $\delta(E,E^{\theta})=v$ and $s\in S$ with $\ell(svs^{\sigma})=\ell(v)-2$. Let $E=E_0\sim_s E_1\sim\cdots \sim E_{\ell-1}\sim_{s^{\sigma}}E_{\ell}=E^{\theta}$ be a reduced gallery. Since $\delta(E,E_1)=s$ we have $\delta(E^{\theta},E_1^{\theta})=s^{\sigma}$. If $E_1^{\theta}\neq E_{\ell-1}$ then $\delta(E_1,E_1^{\theta})=sv$, which is not in the same $\sigma$-conjugacy class as $v$ (by parity), a contradiction. Thus $E_1^{\theta}=E_{\ell-1}$. 
\end{proof}

The following lemma will be used extensively in proving that certain automorphisms are uniclass. 

\begin{lemma}\label{ext} Let $\theta$ be an automorphism of an arbitrary building $\Delta$ with companion automorphism $\sigma$. Let $C,D$ be two chambers with the property that  $\{C,C^\theta\}\subseteq\mathsf{conv}\{D,D^\theta\}$. Then $\delta(D,D^\theta)\in\Cl^\sigma(\delta(C,C^\theta))$.\end{lemma}

\begin{proof}
Let $\gamma$ be a minimal length gallery from $C$ to $D$, of type $(s_1,\ldots,s_n)$. Then $\gamma^{\theta}$ is a minimal length gallery from $C^{\theta}$ to $D^{\theta}$ of type $(s_1^{\sigma},\ldots,s_n^{\sigma})$. Let $\gamma'$ be the reverse gallery of $\gamma$ (hence $\gamma'$ is a gallery of type $(s_n,\ldots,s_1)$ from $D$ to $C$). Let $\gamma''$ be a minimal length gallery from $C$ to $C^{\theta}$. Since $C,C^{\theta},D,D^{\theta}$ all lie in $\mathrm{conv}\{D,D^{\theta}\}$ all chambers of the galleries $\gamma'$, $\gamma''$, and $\gamma^{\theta}$ lie in $\mathrm{conv}\{D,D^{\theta}\}$. Thus every chamber of the gallery $\gamma'\gamma''\gamma^{\theta}$ from $D$ to $D^{\theta}$ lies in $\mathrm{conv}\{D,D^{\theta}\}$, and in particular every chamber of this gallery lies in a common apartment of $\Delta$. Thus $\delta(D,D^{\theta})=s_n\cdots s_2s_1\delta(C,C^{\theta})s_1^{\sigma}s_2^{\sigma}\cdots s_n^{\sigma}\in\Cl^{\sigma}(\delta(C,C^{\theta}))$.
\end{proof}

\subsection{Fixed and opposition diagrams of uniclass automorphisms}\label{Fixed-Opposite}

Let $\theta$ be an automorphism of a spherical building, with companion automorphism $\sigma$. The \textit{fixed diagram} of $\theta$ is the triple $\Fix(\theta)=(\Pi,J,\sigma)$, where $J$ is the union of all distinguished $\sigma$-orbits $K\in\Orb^{\sigma}(S)$ for which there is a type $K$ simplex stabilised by~$\theta$. 

A remarkable consequence of our classification of uniclass automorphisms of thick spherical buildings (Theorem~\ref{thm:main1}) is that the associated twisted conjugacy classes are precisely the bi-capped classes of $\sigma$-involutions in the Coxeter group. The following proposition shows that this property implies a beautiful connection between the fixed and opposition diagrams of a uniclass automorphism.

\begin{prop}\label{prop:duality}
Let $\theta$ be a uniclass automorphism of a thick spherical building of type $(W,S)$, with companion automorphism~$\sigma\in\Aut(\Pi)$. Assume that $\disp(\theta)$ is a bi-capped class of $\sigma$-involutions. Let $J$ (respectively $J'$) be the unique $\sigma$-rigid (respectively $\sigma\sigma_0$-rigid) subset of $S$ such that $w_J,w_{J'}w_0\in \disp(\theta)$. Then
$$
\Fix(\theta)=(\Pi,S\backslash J,\sigma)\quad\text{and}\quad \Opp(\theta)=(\Pi,S\backslash J',\sigma\sigma_0).
$$
In particular, $\Opp(\theta)=\psi(\Fix(\theta))$, where $\psi$ is the duality on the set of bi-capped admissible diagrams given in Definition~\ref{defn:psiondiagrams}. Moreover if $\Fix(\theta)$ has relative type of rank $r$ and $\Opp(\theta)$ has relative type of rank $r'$ (see Definition~\ref{defn:relativetype}) then $r+r'=|S|$. 
\end{prop}

\begin{proof}
Note that for $C\in\Ch(\Delta)$ and $K\subseteq S$ preserved under $\sigma$ and $\sigma_0$ we have $\delta(C,C^{\theta})\in W_K$ if and only if the type $S\backslash K$ simplex of $C$ is fixed by~$\theta$, and $\delta(C,C^{\theta})\in W_Kw_0$ if and only if the type $S\backslash K$ simplex of $C$ is mapped onto an opposite simplex by~$\theta$. By Theorem~\ref{thm:downwardsclosure}(3) and the bi-capped property it follows that for each $w\in \disp(\theta)$ there exists $x,y\in W$ such that $x^{-1}wx^{\sigma}=w_J$ and $y^{-1}wy^{\sigma}=w_{J'}w_0$ with $\ell(x^{-1}wx^{\sigma})=\ell(w)-2\ell(x)$ and $\ell(y^{-1}wy^{\sigma})=\ell(w)+2\ell(y)$. It follows that if $K\subseteq S$ then: 
\begin{compactenum}[$(1)$]
\item if $\disp(\theta)\cap W_K\neq\emptyset$ then $J\subseteq K$, and
\item if $\disp(\theta)\cap W_Kw_0\neq\emptyset$ then $J'\subseteq K$
\end{compactenum}
and hence $\Fix(\theta)=(\Pi,S\backslash J,\sigma)$ and $\Opp(\theta)=(\Pi,S\backslash J',\sigma\sigma_0)$. The remaining statements follow from Theorem~\ref{thm:bicapped} and Proposition~\ref{prop:ranksum}.
\end{proof}

\subsection{The finite case}\label{sec:counting}

Let $\Delta=(\Ch(\Delta),\delta)$ be a finite thick spherical building of type $(W,S)$ and for each $s\in S$ let $q_s=|\{D\in\Ch(\Delta)\mid \delta(C,D)=s\}|$ (this cardinality is independent of the particular $C\in\Ch(\Delta)$ chosen). Moreover, if $\theta$ is an automorphism of $\Delta$ then $q_s=q_{s^{\sigma}}$, where $\sigma$ is the companion automorphism of~$\theta$. It is easy to see that for all $C\in\Ch(\Delta)$ and $w\in W$,
$$
|\{D\in\Ch(\Delta)\mid \delta(C,D)=w\}|=q_{s_1}\cdots q_{s_n}
$$
where $w=q_{s_1}\cdots q_{s_n}$ is any reduced expression for $w$. We define $q_w=q_{s_1}\cdots q_{s_n}$, with $w=s_1\cdots s_n$ reduced.

For $w\in W$ and $\theta\in\mathrm{Aut}(\Delta)$ consider the set
$$
\Delta_w(\theta)=\{C\in\Ch(\Delta)\mid \delta(C,C^{\theta})=w\}.
$$
In general computing the cardinality of these sets is a very complicated problem, however it turns out that for uniclass elements the counts are simplified due to the following theorem and its corollary. We thank Arun Ram for helpful discussions related to this theorem -- in particular it arises in a very natural way from class sums in the group algebra of a finite group of Lie type (see \cite[\S5]{DRS:23}).

\begin{thm}\label{thm:counting}
Let $\Delta=(\Ch(\Delta),\delta)$ be a finite thick spherical building and let $\theta$ be an automorphism with companion automorphism $\sigma$. If $w\in W$ and $s\in S$ with $\ell(sws^{\sigma})=\ell(w)+2$ then 
$$
|\Delta_{ws^{\sigma}}(\theta)|=|\Delta_{sw}(\theta)|\quad\text{and}\quad |\Delta_{sws^{\sigma}}(\theta)|=q_s|\Delta_w(\theta)|+(q_s-1)|\Delta_{sw}(\theta)|.
$$
\end{thm}

\begin{proof}
Let $\cP_s$ denote the set of all $s$-panels of $\Delta$ (that is, residues of type $S\backslash\{s\}$). For $v\in W$ and $P\in\cP_s$ let $\Delta^P_v(\theta)=\Delta_v(\theta)\cap P$. If $P\in\cP_s$ then
$$
\delta(P,P^{\theta})=\{\delta(C,C^{\theta})\mid C\in P\}=\langle s\rangle w\langle s^{\sigma}\rangle
$$
for some $w\in W$, and we may take $w$ to be minimal length in the double coset $\langle s\rangle w\langle s^{\sigma}\rangle$. Thus if $v\in \langle s\rangle w\langle s^{\sigma}\rangle$ we have
\begin{align}
\label{eq:summing}|\Delta_v(\theta)|=\sum_{P\in\cP_s}|\Delta_v^P(\theta)|=\sum_{\{P\in\cP_s\,:\,\delta(P,P^{\theta})=\langle s\rangle w\langle s^{\sigma}\rangle\}}|\Delta_v^P(\theta)|.
\end{align}

Suppose that $\ell(sws^{\sigma})=\ell(w)+2$, and that $\delta(P,P^{\theta})=\langle s\rangle w\langle s^{\sigma}\rangle$. In particular this implies that $P\cap P^{\theta}=\emptyset$. Let $C_0\in P$ and $D_0\in P^{\theta}$ be the unique chambers with $\delta(C_0,D_0)=w$. If $C\in P\backslash\{C_0\}$ and $D\in P^{\theta}\backslash\{D_0\}$ then, from the building axioms, $\delta(C,D_0)=sw$, $\delta(C_0,D)=ws^{\sigma}$, and $\delta(C,D)=sws^{\sigma}$. There are two distinct possibilities to consider.

Firstly, if $C_0^{\theta}=D_0$ then $\Delta^P_w(\theta)=\{C_0\}$, $\Delta^P_{sws^{\sigma}}(\theta)=P\backslash\{C_0\}$, and $\Delta^P_{sw}(\theta)=\Delta^P_{ws^{\sigma}}(\theta)=\emptyset$. Thus $|\Delta^P_w(\theta)|=1$, $|\Delta^P_{sws^{\sigma}}(\theta)|=q_s$, and $|\Delta_{sw}^P(\theta)|=|\Delta_{ws^{\sigma}}^P(\theta)|=0$.

Secondly, if $C_0^{\theta}\neq D_0$ let $C_1=D_0^{\theta^{-1}}\in P\backslash\{C_0\}$. Then $\Delta^P_{ws^{\sigma}}(\theta)=\{C_0\}$, $\Delta^P_{sw}(\theta)=\{C_1\}$, and $\Delta^P_{sws^{\sigma}}(\theta)=P\backslash\{C_0,C_1\}$. Thus $|\Delta^P_{ws^{\sigma}}(\theta)|=|\Delta^P_{sw}(\theta)|=1$, $|\Delta^P_{sws^{\sigma}}(\theta)|=q_s-1$, and $|\Delta^P_w(\theta)|=0$.

Let $N_1$ (respectively $N_2$) denote the number of $s$-panels $P$ with the property that $\delta(P,P^{\theta})=\langle s\rangle w\langle s^{\sigma}\rangle$ and $C_0^{\theta}=D_0$ (respectively, $C_0^{\theta}\neq D_0$) with $C_0,D_0$ as above. Then by~(\ref{eq:summing}) we have
$$
|\Delta_w(\theta)|=N_1,\quad |\Delta_{ws^{\sigma}}(\theta)|=|\Delta_{sw}(\theta)|=N_2,\quad \text{and}\quad |\Delta_{sws^{\sigma}}(\theta)|=q_sN_1+(q_s-1)N_2,
$$
completing the proof.
\end{proof}

\begin{cor}\label{cor:counting}
If $\disp(\theta)$ consists of $\sigma$-involutions, and $w\in W$ and $s\in S$ with $\ell(sws^{\sigma})=\ell(w)+2$, then
$$
|\Delta_{ws^{\sigma}}(\theta)|=|\Delta_{sw}(\theta)|=0\quad\text{and}\quad |\Delta_{sws^{\sigma}}(\theta)|=q_s|\Delta_w(\theta)|.
$$
\end{cor}

\begin{proof}
Note that $sw$ is not a $\sigma$-involution (as $(sw)^{\sigma}=(sw)^{-1}$, combined with the condition $w^{\sigma}=w^{-1}$, implies that $sws^{\sigma}=w$, contradicting $\ell(sws^{\sigma})=\ell(w)+2$). Hence by assumption $|\Delta_{sw}(\theta)|=|\Delta_{ws^{\sigma}}(\theta)|=0$, and the result follows from Theorem~\ref{thm:counting}.
\end{proof}

For any subset $U\subseteq W$, we write $U(q)=\sum_{w\in U}q_w$ and $U(q^{1/2})=\sum_{w\in U}q_w^{1/2}$. Thus
$$
W(q)=\sum_{w\in W}q_w=\sum_{w\in W}|\{D\in\Ch(\Delta)\mid \delta(C,D)=w\}|=|\Ch(\Delta)|
$$
counts the total number of chambers in $\Delta$ (this is called the \textit{Poincar\'e polynomial} of $W$, and there are product formulae available for this polynomial for each irreducible spherical Coxeter group, see \cite{Mac:72}).

The following theorem computes the cardinalities $|\Delta_w(\theta)|$ for uniclass automorphisms in terms of the ``class sum'' $\bC(q^{1/2})$ where $\bC=\disp(\theta)$. We shall give another formula, making use of Corollary~\ref{cor:bicapped}, in Theorem~\ref{thm:counts2}.

\begin{thm}\label{thm:counts}
Let $\Delta$ be a finite spherical building with parameters $(q_s)_{s\in S}$. Let $\theta$ be a uniclass automorphism with companion automorphism~$\sigma$ and suppose that $\sigma$ has order $1$ or $2$. Then $\bC=\disp(\theta)$ is a full $\sigma$-conjugacy class, and for $w\in\bC$ we have 
$$
|\Delta_w(\theta)|=\frac{W(q)}{\bC(q^{1/2})}q_w^{1/2}.
$$
\end{thm}

\begin{proof}
By Proposition~\ref{prop:full} we have that $\bC=\disp(\theta)$ is a full $\sigma$-conjugacy class consisting of $\sigma$-involutions. We claim that $q_w^{-1/2}|\Delta_w(\theta)|=q_v^{-1/2}|\Delta_v(\theta)|$ for all $w,v\in\bC$. Since $w,v\in\bC$ there exists $x\in W$ with $v=x^{-1}wx^{\sigma}$. Let $x=s_1\cdots s_n$ be a reduced expression, and define $w_0=w$ and $w_j=s_jw_{j-1}s_j^{\sigma}$ for $1\leq j\leq n$, so that $v=w_n$. If $\ell(w_j)=\ell(w_{j-1})$ then $w_j=w_{j-1}$ by Lemma~\ref{lem:sigmacondition}, and thus it follows from Corollary~\ref{cor:counting} that
$$
|\Delta_{w_j}(\theta)|=q_{s_j}^{\epsilon_j}|\Delta_{w_{j-1}}(\theta)|
$$
where $\epsilon_j\in\{-1,0,1\}$ satisfies $\ell(w_j)=\ell(w_{j-1})+2\epsilon_j$. Moreover, we have $q_{w_j}=q_{s_j}^{2\epsilon_j}q_{w_{j-1}}$, and hence
$$
|\Delta_v(\theta)|=q_{s_1}^{\epsilon_1}\cdots q_{s_n}^{\epsilon_n}|\Delta_w(\theta)|\quad\text{and}\quad q_v=q_{s_1}^{2\epsilon_1}\cdots q_{s_n}^{2\epsilon_n}q_w
$$
and the claim follows.

Let $w\in \bC$. Since $W(q)$ is the total number of chambers, we have (using the above claim)
$$
W(q)=\sum_{v\in \bC}|\Delta_v(\theta)|=q_w^{-1/2}|\Delta_w(\theta)|\sum_{v\in \bC}q_v^{1/2}
$$
and hence the result.
\end{proof}

\subsection{The non-thick case}\label{sec:nonthick}

Proposition~\ref{prop:reduciblecase} reduces the classification of uniclass automorphisms of spherical buildings to the irreducible case. In this section we reduce the classification to the thick case. We refer to \cite{Scharlau} for the construction of the ``thick-frame'' associated to a non-thick spherical building outlined briefly below.

Let $\Delta$ be a non-thick spherical building of irreducible type $(W,S)$. Adjacent chambers $C,D\in\Ch(\Delta)$ are called \textit{thin-adjacent} if the panel containing $C\cup D$ contains precisely two chambers. Chambers $C,D\in\Ch(\Delta)$ are said to be in the same \textit{thin class} if they can be joined by a gallery involving only thin adjacencies. 

The \textit{thick-frame} of $\Delta$ is the thick spherical building $\Delta'$ with chamber set $\Ch(\Delta')$ being the set of all thin classes of $\Delta$. Thin classes $C',D'\in\Ch(\Delta')$ are adjacent if and only if there exists adjacent chambers $C,D\in\Ch(\Delta)$ with $C\in C'$ and $D\in D'$. We shall describe the Coxeter type $(W',S')$ of $\Delta'$, along with the associated $W'$-valued distance function $\delta':\Ch(\Delta')\times\Ch(\Delta')\to W'$ in the paragraphs below. 

Fix, once and for all, a choice of apartment $A_0$ of $\Delta$ and a chamber $C_0$ of $A_0$. Let $C_0'$ be the thin class containing $C_0$ (and hence $C_0'$ is contained in $A_0$ by thinness in the thin class). Let $W'$ be the reflection subgroup of $W$ generated by the reflections in the thin walls of $A_0$, and let $S'$ be the reflections in the (necessarily thin) walls bounding $C_0'$. Then $(W',S')$ is a Coxeter system. Identify the set $\Ch(A_0)$ with $W$ in the unique way so that $C_0$ corresponds to $1$ and the unique $s$-adjacent chamber to $C_0$ in $A_0$ corresponds to $s$, for each $s\in S$. Under this identification, let $U$ denote the set of elements of $W$ contained in~$C_0'$. 

Let $C'\in\Ch(\Delta')$, and let $A$ be an apartment (of $\Delta$) containing $C_0'$ and $C'$. Then $W'$ acts simply transitively on the set of thin classes in $A$ (with $S'$ being the reflections in the walls bounding $C_0'$). For $u\in U$ Let $C'[u]\in\Ch(\Delta)$ denote the unique chamber of the thin class $C'$ in the orbit of $u\in C_0'$ under the action of $W'$. It is easy to see that this does not depend on the choice of apartment~$A$. It follows that for any thin classes $C',D'\in\Ch(\Delta')$ we have $\delta(C'[1],D'[1])\in W'$, and we define
$$
\delta'(C',D')=\delta(C'[1],D'[1]).
$$

We shall frequently use the following fact. If $C',D'\in\Ch(\Delta')$ are thin classes, and if $C=C'[u]$ and $D=D'[v]$ with $u,v\in U$, then by thinness in the thin classes $C'$ and $D'$ we have
\begin{align}\label{eq:freq}
\delta(C,D)=\delta(C'[u],C'[1])\delta(C'[1],D'[1])\delta(D'[1],D'[v])=u^{-1}\delta'(C',D')v.
\end{align}

Let $\theta$ be an automorphism of $\Delta$ with companion automorphism $\sigma\in\Aut(\Pi)$. We define the induced automorphism $\theta'$ of $\Delta'$, along with its associated companion automorphism $\sigma'\in\Aut(\Pi')$ (with $\Pi'$ the Coxeter graph of $(W',S')$) below. We also define an associated bijection $\pi':U\to U$.
\begin{compactenum}[$(1)$]
\item For $C'\in \Ch(\Delta')$ let $C'^{\theta'}=\{C^{\theta}\mid C\in C'\}$. Since $\theta$ maps thin panels to thin panels, it follows that $C'^{\theta'}$ is a thin class, and that $\theta'$ is an automorphism of $\Delta'$.
\item Define $\sigma'\in\Aut(\Pi')$ by $\delta'(C',D')=s'$ if and only if $\delta'(C'^{\theta'},D'^{\theta'})=s'^{\sigma'}$, for $s'\in S'$. 
\item Define a bijection $\pi':U\to U$ by 
$$
C_0'[u]^{\theta}=C_0'^{\theta'}[\pi'(u)].
$$
That is, the $u$-chamber of the thin class $C_0'$ is mapped by $\theta$ to the $\pi'(u)$-chamber of the thin class of $C_0'^{\theta'}$. It follows that
$$
C'[u]^{\theta}=C'^{\theta'}[\pi'(u)]\quad\text{for all thin classes $C'\in\Ch(\Delta')$}.
$$
\end{compactenum}

\begin{remark}
It is possible for $\sigma'\in\Aut(\Pi')$ to be nontrivial even when $\sigma\in\Aut(\Pi)$ is trivial. For example, if $\theta$ is type preserving on a non-thick hexagon, then it is possible for $\theta'$ to be either a collineation or a duality of the thick-frame projective plane. This example illustrates the importance of the additional information contained in the bijection $\pi':U\to U$. 
\end{remark}

The following lemma gives a relation between $\sigma$, $\sigma'$, and $\pi'$. 

\begin{lemma}\label{lem:connection}
We have $\pi'(1)^{-1}w'^{\sigma'}\pi'(1)=w'^{\sigma}$ for all $w'\in W'$. 
\end{lemma}

\begin{proof}
By induction on $\ell'(w')$ (with $\ell'$ the length function on $(W',S')$) it suffices to consider the case $w'=s'\in S'$. Let $C'$ and $D'$ be thin classes with $\delta'(C',D')=s'$. Then $\delta'(C'^{\theta'},D'^{\theta'})=s'^{\sigma'}$, and so
\begin{align*}
\pi'(1)^{-1}s'^{\sigma'}\pi'(1)&=\delta(C'^{\theta'}[\pi(1)],C'^{\theta'}[1])\delta(C'^{\theta'}[1],D'^{\theta'}[1])\delta(D'^{\theta'}[1],D'^{\theta'}[\pi'(1)])\\
&=\delta(C'^{\theta'}[\pi'(1)],D'^{\theta'}[\pi'(1)])\\
&=\delta(C'[1]^{\theta},D'[1]^{\theta})\\
&=\delta(C'[1],D'[1])^{\sigma}\\
&=s'^{\sigma'},
\end{align*}
completing the proof.
\end{proof}

\begin{thm}\label{thm:reducetothick}
Let $\Delta$ be a non-thick building of type $(W,S)$ with thick frame $\Delta'$ of type $(W',S')$. Let $\theta$ be an automorphism of $\Delta$ and let $\theta'$ be the induced automorphism of $\Delta'$. Then $\theta$ is uniclass if and only if $\theta'$ is uniclass. 
\end{thm}

\begin{proof}
Let $\sigma\in\Aut(\Pi)$ be the companion automorphism of $\theta$, and let $\sigma'\in\Aut(\Pi')$ be the companion automorphism of $\theta'$ and $\pi'$ the associated bijection of $U$. 

Suppose that $\theta'$ is uniclass. Let $C\in\Ch(\Delta)$ and let $C'\in\Ch(\Delta')$ be the thin class containing $C$. Write $w=\delta(C,C^{\theta})$ and $w'=\delta'(C',C'^{\theta'})$. We have 
$$
\delta(C'[1],C'[1]^{\theta})=w'\pi'(1).
$$
Writing $\delta(C'[1],C)=u$ (an element of $U$) we have 
$$
\delta(C,C^{\theta})=\delta(C,C'[1])\delta(C'[1],C'[1]^{\theta})\delta(C'[1]^{\theta},C^{\theta})=u^{-1}w'\pi'(1)u^{\sigma}.
$$
If $D$ is any other chamber of $\Delta$ the same argument gives $\delta(D,D^{\theta})=u_1^{-1}v'\pi'(1)u_1^{\sigma}$ for some $u_1\in U$ and $v'\in\disp(\theta')$. Since $\theta'$ is uniclass we have $v'=x'^{-1}w'x'^{\sigma'}$ for some $x'\in W'$. It follows that
$$
\delta(D,D^{\theta})=u_1^{-1}x'^{-1}u\delta(C,C^{\theta})u^{-\sigma}\pi'(1)^{-1}x'^{\sigma'}\pi'(1)u_1^{\sigma}.
$$
Thus by Lemma~\ref{lem:connection} we have that $\delta(D,D^{\theta})$ is $\sigma$-conjugate to $\delta(C,C^{\theta})$, and so $\theta$ is uniclass.

Suppose that $\theta$ is uniclass. To prove that $\theta'$ is uniclass it suffices to show that if $C',D'\in\Ch(\Delta')$ with $\delta'(C',D')=s'\in S'$ then $\delta'(D',D'^{\theta'})$ is $\sigma'$-conjugate (in $W'$) to $\delta'(C',C'^{\theta'})$. Choose chambers $C\in C'$ and $D\in D'$ such that $\delta(C,D)=s\in S$. Then $D=D'[u]$ and $C=C'[u]$ for some $u\in U$, and $usu^{-1}=s'$. To see this, choose an apartment of $\Delta$ containing $C$ and $D$, and let $W'$ act on this apartment with $S'$ being the reflections in the walls of $C'$. Since $C$ and $D$ lie in distinct thin classes the wall separating them is thick, and hence is a wall of $C'$. It follows that $D=s'C$, and hence $C=C'[u]$ and $D=D'[u]$ for some $u\in U$, and the formula $usu^{-1}=s'$ follows from~(\ref{eq:freq}). Moreover, by considering the images under $\theta$ we also have $\pi'(u)s^{\sigma}\pi'(u)^{-1}=s'^{\sigma'}$. 

Let $w=\delta(C,C^{\theta})$ and $w'=\delta(C',C'^{\theta'})$. We have
\begin{align*}
\delta'(C',C'^{\theta'})&=\delta(C'[1],C'^{\theta'}[1])\\
&=\delta(C'[1],C'[u])\delta(C'[u],C'[u]^{\theta})\delta(C'[u]^{\theta},C'^{\theta'}[1])\\
&=u\delta(C,C^{\theta})\delta(C'^{\theta'}[1],C'^{\theta'}[\pi'(u)])^{-1}\\
&=uw\pi'(u)^{-1},
\end{align*}
and similarly $\delta'(D',D'^{\theta'})=u\delta(D,D^{\theta})\pi'(u)^{-1}$. Now, since $\theta$ is uniclass we have $\delta(D,D^{\theta})\in\{w,sws^{\sigma}\}$ (because $sw$ and $ws^{\sigma}$ are not in the same $\sigma$-class as $w$ by parity), and it follows from the above calculations that
$$
\delta'(D',D'^{\theta'})\in\{w',(usu^{-1})w'(\pi'(u)s^{\sigma}\pi'(u)^{-1})\}=\{w',s'w's'^{\sigma'}\},
$$
and so $\theta'$ is uniclass. 
\end{proof}

\section{Uniclass automorphisms are capped}\label{sec:3}

In this section we will show that uniclass automorphisms are necessarily capped (Theorem~\ref{thm:uncapped}). This observation will be used extensively in the proof of Theorem~\ref{thm:main1}.  We begin with some background and preliminary results required for the proof of Theorem~\ref{thm:uncapped}.

Let $\Phi$ be a reduced irreducible crystallographic root system with simple roots $\alpha_1,\ldots,\alpha_n$ and Weyl group~$W$. The \textit{height} of a root $\alpha=k_1\alpha_1+\cdots+k_n\alpha_n$ is $\mathrm{ht}(\alpha)=k_1+\cdots +k_n$. There is a unique root $\varphi\in\Phi$ of maximal height (the \textit{highest root} of $\Phi$). The \textit{polar type} of $\Phi$ is the subset $\wp\subseteq\{1,2,\ldots,n\}$ given by 
$$
\wp=\{1\leq i\leq n\mid \langle\alpha_i,\varphi\rangle\neq 0\}.
$$
If $\Phi=\sA_n$ then $\wp=\{1,n\}$, while if $\Phi\neq \sA_n$ then $\wp=\{p\}$ is a singleton set, and in this case we often refer to the element $p$ as the \textit{polar node}. 

We have $w_0s_{\varphi}=w_{S\backslash\wp}$. To see this, note that $w_0\varphi=-\varphi$ (by properties of the highest root) and so $w_0s_{\varphi}(\alpha)=w_0\alpha+\langle\alpha,\varphi^{\vee}\rangle\varphi$ for all $\alpha\in\Phi^+$. Since $\langle\alpha,\varphi^{\vee}\rangle\in\{0,1\}$ for all $\alpha\in\Phi^+$ (see \cite[IV, \S1.8, Proposition~25]{Bou:00}) it follows that $\Phi(w_0s_{\varphi})=\Phi_{S\backslash\wp}^+$, and hence the result. 

\begin{lemma}\label{lem:highestroot}
Let $\Phi$ be an irreducible crystallographic root system with highest root~$\varphi$ with Coxeter system $(W,S)$ and let $\sigma$ be the opposition relation. Let $\beta\in\Phi$ be a long root (with all roots considered long in the simply laced case). The element $w_{S\backslash\wp}$ is of minimal length in the $\sigma$-conjugacy class $\Cl^{\sigma}(s_{\beta}w_0)$.
\end{lemma}

\begin{proof}
Since all roots of the same length are conjugate, Proposition~\ref{prop:basic1} gives
$$
\Cl^{\sigma}(s_{\beta}w_0)=w_0\Cl(w_0s_{\beta}w_0)=\{w_0s_{\alpha}\mid \alpha\in\Phi_L^+\}
$$
where $\Phi_L$ denotes the set of long roots of $\Phi$. Thus $w_{S\backslash\wp}=w_0s_{\varphi}\in\Cl^{\sigma}(s_{\beta}w_0)$, and it follows from Lemma~\ref{lem:min1} that this element has minimal length in the class.
\end{proof}

In the following lemma $W$ is of type $\sB_n$ or $\sD_n$, and we label the simple roots following Bourbaki conventions. For $1\leq i\leq n$ let $W_{\geq i}=W_{\{i,i+1,\ldots,n\}}$ and let $w_{\geq i}$ be the longest element of $W_{\geq i}$.

\begin{lemma}\label{lem:classfacts}
We have the following.
\begin{compactenum}[$(1)$]
\item Let $W$ be of type $\sD_n$ with $n\geq 5$. Let $\rho$ be the order $2$ diagram automorphism if $n$ is even, and let $\rho=1$ if $n$ is odd. Then $s_1w_{\geq 4}$ is a minimal length element of $\Cl^{\rho}(s_1w_0)$.
\item Let $W$ be of type $\sB_n$ and let $3\leq i<n$. Then $s_1w_{\geq n-i+3}$ is a minimal length element in $\Cl(s_1w_{\geq i+1}w_0)$. 
\item Let $W$ be of type $\sD_n$ and let $4\leq i< n-1$ with $i$ even. Then $s_1w_{\geq n-i+3}$ is a minimal length element in $\Cl(s_1w_{\geq i+1}w_0)$. 
\item Let $W$ be of type $\sD_n$ and let $3\leq i< n-1$ with $i$ odd. Let $\sigma$ be the order~$2$ diagram automorphism interchanging types~$n-1$ and $n$. If $i>3$ (respectively $i=3$) then $s_1w_{\geq n-i+3}$ (respectively $s_1$) is a minimal length element in $\Cl^{\sigma}(s_1w_{\geq i+1}w_0)$.
\end{compactenum}
\end{lemma}

\begin{proof} These statements may be deduced from the tables in~\cite{GKP:00}, however we provide direct calculations below. 

(1) Let $v=w_{\{3,4,\ldots,n-1\}}w_{\{1,2,\ldots,n-2,n-1\}}w_{\{1,2,\ldots,n-2,n\}}$. By considering the action on roots we have
$
vs_1w_0v^{-\rho}=s_1w_{\geq 4}.
$
Writing $J=\{1,4,5,\ldots,n\}$ we have $sw_J=w_Js^{\rho}$ for all $s\in J$, and so by Lemma~\ref{lem:min1} the element $w_J=s_1w_{\geq 4}$ is of minimal length in $\Cl^{\sigma}(s_1w_0)$.

(2) and (3). Let $w_1=s_1w_{\geq i+1}w_0$. By considering the action on roots we see that 
$$
w_{\{3,4,\ldots,n-1\}}w_{\geq 2}w_1w_{\geq 2}w_{\{3,4,\ldots,n-1\}}=s_1w_{\geq n-i+3}
$$
(using the fact that $i$ is even in the $\sD_n$ case) and the result follows from Lemma~\ref{lem:min1}.

(4) Let $w_1=s_1w_{\geq i+1}w_0$. If $i>3$ is odd then
$$
w_{\{3,4,\ldots,n-1\}}w_{\geq 2}w_1w_{\geq 2}^{\sigma}w_{\{3,4,\ldots,n-1\}}^{\sigma}=s_1w_{\geq n-i+3},
$$
while if $i=3$ then $w_{\{3,4,\ldots,n-1\}}w_{\geq 2}w_1w_{\geq 2}^{\sigma}w_{\{3,4,\ldots,n-1\}}^{\sigma}=s_1$, and the result follows from Lemma~\ref{lem:min1}.
\end{proof}

In the following lemma and theorem we will use the observation that if $\theta$ is uncapped, and if $J$ denotes the set of all encircled nodes, and $K$ denotes the set of all shaded nodes, then for for each $k\in K$ the element $w_{S\backslash (J\backslash\{k\})}w_0$ is in $\disp(\theta)$ (see Remark~\ref{rem:elementsofdisp}). 

\begin{lemma}\label{lem:lowrank}
Uncapped automorphisms with the following decorated opposition diagrams are not uniclass:
$$
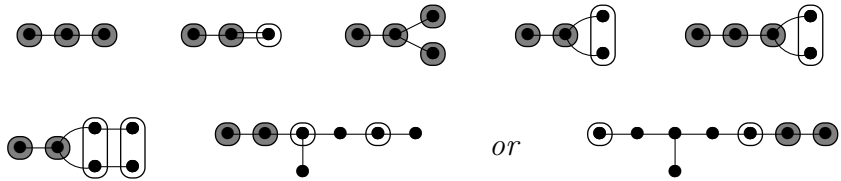
\begin{tikzpicture}[scale=0.5,baseline=-0.5ex]
\node at (1,0.3) {};
\node at (1,-0.5) {};
\node [inner sep=0.8pt,outer sep=0.8pt] at (1,0) (1) {$\bullet$};
\node [inner sep=0.8pt,outer sep=0.8pt] at (2,0) (2) {$\bullet$};
\node [inner sep=0.8pt,outer sep=0.8pt] at (3,0) (3) {$\bullet$};
\draw [line width=0.5pt,line cap=round,rounded corners,fill=ggrey] (1.north west)  rectangle (1.south east);
\draw [line width=0.5pt,line cap=round,rounded corners,fill=ggrey] (2.north west)  rectangle (2.south east);
\draw [line width=0.5pt,line cap=round,rounded corners,fill=ggrey] (3.north west)  rectangle (3.south east);
\draw (1,0)--(2,0);
\draw (2,0)--(3,0);
\node [inner sep=0.8pt,outer sep=0.8pt] at (1,0) (1) {$\bullet$};
\node [inner sep=0.8pt,outer sep=0.8pt] at (2,0) (2) {$\bullet$};
\node [inner sep=0.8pt,outer sep=0.8pt] at (3,0) (3) {$\bullet$};
\end{tikzpicture}\qquad \begin{tikzpicture}[scale=0.5,baseline=-0.5ex]
\node at (1,0.3) {};
\node at (1,-0.5) {};
\node [inner sep=0.8pt,outer sep=0.8pt] at (1,0) (1) {$\bullet$};
\node [inner sep=0.8pt,outer sep=0.8pt] at (2,0) (2) {$\bullet$};
\node [inner sep=0.8pt,outer sep=0.8pt] at (3,0) (3) {$\bullet$};
\draw [line width=0.5pt,line cap=round,rounded corners,fill=ggrey] (1.north west)  rectangle (1.south east);
\draw [line width=0.5pt,line cap=round,rounded corners,fill=ggrey] (2.north west)  rectangle (2.south east);
\draw [line width=0.5pt,line cap=round,rounded corners] (3.north west)  rectangle (3.south east);
\draw (1,0)--(2,0);
\draw (2,0.07)--(3,0.07);
\draw (2,-0.07)--(3,-0.07);
\node [inner sep=0.8pt,outer sep=0.8pt] at (1,0) (1) {$\bullet$};
\node [inner sep=0.8pt,outer sep=0.8pt] at (2,0) (2) {$\bullet$};
\node [inner sep=0.8pt,outer sep=0.8pt] at (3,0) (3) {$\bullet$};
\end{tikzpicture}\qquad \begin{tikzpicture}[scale=0.5,baseline=-0.5ex]
\node [inner sep=0.8pt,outer sep=0.8pt] at (3,0) (3) {$\bullet$};
\node [inner sep=0.8pt,outer sep=0.8pt] at (4,0) (4) {$\bullet$};
\node [inner sep=0.8pt,outer sep=0.8pt] at (5,0.5) (5a) {$\bullet$};
\node [inner sep=0.8pt,outer sep=0.8pt] at (5,-0.5) (5b) {$\bullet$};
\draw [line width=0.5pt,line cap=round,rounded corners,fill=ggrey] (4.north west)  rectangle (4.south east);
\draw [line width=0.5pt,line cap=round,rounded corners,fill=ggrey] (3.north west)  rectangle (3.south east);
\draw [line width=0.5pt,line cap=round,rounded corners,fill=ggrey] (5a.north west)  rectangle (5a.south east);
\draw [line width=0.5pt,line cap=round,rounded corners,fill=ggrey] (5b.north west)  rectangle (5b.south east);
\draw (3,0)--(4,0);
\draw (4,0) -- (5,0.5);
\draw (4,0) --(5,-0.5);
\node at (5,0.5) {$\bullet$};
\node at (5,-0.5) {$\bullet$};
\node [inner sep=0.8pt,outer sep=0.8pt] at (3,0) (3) {$\bullet$};
\node [inner sep=0.8pt,outer sep=0.8pt] at (4,0) (4) {$\bullet$};
\node [inner sep=0.8pt,outer sep=0.8pt] at (5,0.5) (5a) {$\bullet$};
\node [inner sep=0.8pt,outer sep=0.8pt] at (5,-0.5) (5b) {$\bullet$};
\end{tikzpicture}\qquad
\begin{tikzpicture}[scale=0.5,baseline=-0.5ex]
\node [inner sep=0.8pt,outer sep=0.8pt] at (3,0) (3) {$\bullet$};
\node [inner sep=0.8pt,outer sep=0.8pt] at (4,0) (4) {$\bullet$};
\node [inner sep=0.8pt,outer sep=0.8pt] at (5,0.5) (5a) {$\bullet$};
\node [inner sep=0.8pt,outer sep=0.8pt] at (5,-0.5) (5b) {$\bullet$};
\draw [line width=0.5pt,line cap=round,rounded corners,fill=ggrey] (4.north west)  rectangle (4.south east);
\draw [line width=0.5pt,line cap=round,rounded corners,fill=ggrey] (3.north west)  rectangle (3.south east);
\draw [line width=0.5pt,line cap=round,rounded corners] (5a.north west)  rectangle (5b.south east);
\draw (3,0)--(4,0);
\draw (4,0) to [bend left] (5,0.5);
\draw (4,0) to [bend right=45] (5,-0.5);
\node at (5,0.5) {$\bullet$};
\node at (5,-0.5) {$\bullet$};
\node [inner sep=0.8pt,outer sep=0.8pt] at (3,0) (3) {$\bullet$};
\node [inner sep=0.8pt,outer sep=0.8pt] at (4,0) (4) {$\bullet$};
\node [inner sep=0.8pt,outer sep=0.8pt] at (5,0.5) (5a) {$\bullet$};
\node [inner sep=0.8pt,outer sep=0.8pt] at (5,-0.5) (5b) {$\bullet$};
\end{tikzpicture}\qquad
\begin{tikzpicture}[scale=0.5,baseline=-0.5ex]
\node [inner sep=0.8pt,outer sep=0.8pt] at (3,0) (3) {$\bullet$};
\node [inner sep=0.8pt,outer sep=0.8pt] at (4,0) (4) {$\bullet$};
\node [inner sep=0.8pt,outer sep=0.8pt] at (5,0.5) (5a) {$\bullet$};
\node [inner sep=0.8pt,outer sep=0.8pt] at (5,-0.5) (5b) {$\bullet$};
\draw [line width=0.5pt,line cap=round,rounded corners,fill=ggrey] (2.north west)  rectangle (2.south east);
\draw [line width=0.5pt,line cap=round,rounded corners,fill=ggrey] (4.north west)  rectangle (4.south east);
\draw [line width=0.5pt,line cap=round,rounded corners,fill=ggrey] (3.north west)  rectangle (3.south east);
\draw [line width=0.5pt,line cap=round,rounded corners] (5a.north west)  rectangle (5b.south east);
\draw (2,0)--(4,0);
\draw (4,0) to [bend left] (5,0.5);
\draw (4,0) to [bend right=45] (5,-0.5);
\node at (5,0.5) {$\bullet$};
\node at (5,-0.5) {$\bullet$};
\node [inner sep=0.8pt,outer sep=0.8pt] at (2,0) (2) {$\bullet$};
\node [inner sep=0.8pt,outer sep=0.8pt] at (3,0) (3) {$\bullet$};
\node [inner sep=0.8pt,outer sep=0.8pt] at (4,0) (4) {$\bullet$};
\node [inner sep=0.8pt,outer sep=0.8pt] at (5,0.5) (5a) {$\bullet$};
\node [inner sep=0.8pt,outer sep=0.8pt] at (5,-0.5) (5b) {$\bullet$};
\end{tikzpicture}
$$

$$
\begin{tikzpicture}[scale=0.5,baseline=-0.5ex]
\node at (0,0.3) {};
\node [inner sep=0.8pt,outer sep=0.8pt] at (-2,0) (2) {$\bullet$};
\node [inner sep=0.8pt,outer sep=0.8pt] at (-1,0) (4) {$\bullet$};
\node [inner sep=0.8pt,outer sep=0.8pt] at (0,-0.5) (5) {$\bullet$};
\node [inner sep=0.8pt,outer sep=0.8pt] at (0,0.5) (3) {$\bullet$};
\node [inner sep=0.8pt,outer sep=0.8pt] at (1,-0.5) (6) {$\bullet$};
\node [inner sep=0.8pt,outer sep=0.8pt] at (1,0.5) (1) {$\bullet$};
\draw [line width=0.5pt,line cap=round,rounded corners,fill=ggrey] (2.north west)  rectangle (2.south east);
\draw [line width=0.5pt,line cap=round,rounded corners,fill=ggrey] (4.north west)  rectangle (4.south east);
\draw [line width=0.5pt,line cap=round,rounded corners] (3.north west)  rectangle (5.south east);
\draw [line width=0.5pt,line cap=round,rounded corners] (1.north west)  rectangle (6.south east);
\draw (-2,0)--(-1,0);
\draw (-1,0) to [bend left=45] (0,0.5);
\draw (-1,0) to [bend right=45] (0,-0.5);
\draw (0,0.5)--(1,0.5);
\draw (0,-0.5)--(1,-0.5);
\node at (0,-0.5) {$\bullet$};
\node at (0,0.5) {$\bullet$};
\node at (1,-0.5) {$\bullet$};
\node at (1,0.5) {$\bullet$};
\node [inner sep=0.8pt,outer sep=0.8pt] at (-2,0) (2) {$\bullet$};
\node [inner sep=0.8pt,outer sep=0.8pt] at (-1,0) (4) {$\bullet$};
\node [inner sep=0.8pt,outer sep=0.8pt] at (0,-0.5) (5) {$\bullet$};
\node [inner sep=0.8pt,outer sep=0.8pt] at (0,0.5) (3) {$\bullet$};
\node [inner sep=0.8pt,outer sep=0.8pt] at (1,-0.5) (6) {$\bullet$};
\node [inner sep=0.8pt,outer sep=0.8pt] at (1,0.5) (1) {$\bullet$};
\end{tikzpicture}
\qquad\begin{tikzpicture}[scale=0.5,baseline=-1.5ex]
\node [inner sep=0.8pt,outer sep=0.8pt] at (-2,0) (1) {$\bullet$};
\node [inner sep=0.8pt,outer sep=0.8pt] at (-1,0) (3) {$\bullet$};
\node [inner sep=0.8pt,outer sep=0.8pt] at (0,0) (4) {$\bullet$};
\node [inner sep=0.8pt,outer sep=0.8pt] at (1,0) (5) {$\bullet$};
\node [inner sep=0.8pt,outer sep=0.8pt] at (2,0) (6) {$\bullet$};
\node [inner sep=0.8pt,outer sep=0.8pt] at (3,0) (7) {$\bullet$};
\node [inner sep=0.8pt,outer sep=0.8pt] at (0,-1) (2) {$\bullet$};
\draw [line width=0.5pt,line cap=round,rounded corners,fill=ggrey] (1.north west)  rectangle (1.south east);
\draw [line width=0.5pt,line cap=round,rounded corners,fill=ggrey] (3.north west)  rectangle (3.south east);
\draw [line width=0.5pt,line cap=round,rounded corners] (4.north west)  rectangle (4.south east);
\draw [line width=0.5pt,line cap=round,rounded corners] (6.north west)  rectangle (6.south east);
\node at (0,0) {$\bullet$};
\node at (2,0) {$\bullet$};
\node at (0,-1.3) {};
\node at (0,0.3) {};
\node [inner sep=0.8pt,outer sep=0.8pt] at (-2,0) (1) {$\bullet$};
\node [inner sep=0.8pt,outer sep=0.8pt] at (-1,0) (3) {$\bullet$};
\node [inner sep=0.8pt,outer sep=0.8pt] at (0,0) (4) {$\bullet$};
\node [inner sep=0.8pt,outer sep=0.8pt] at (1,0) (5) {$\bullet$};
\node [inner sep=0.8pt,outer sep=0.8pt] at (2,0) (6) {$\bullet$};
\node [inner sep=0.8pt,outer sep=0.8pt] at (3,0) (7) {$\bullet$};
\node [inner sep=0.8pt,outer sep=0.8pt] at (0,-1) (2) {$\bullet$};
\draw (-2,0)--(3,0);
\draw (0,0)--(0,-1);
\end{tikzpicture} \qquad\text{or}\qquad 
\begin{tikzpicture}[scale=0.5,baseline=-1.5ex]
\node at (0,0.3) {};
\node at (0,-1.3) {};
\node [inner sep=0.8pt,outer sep=0.8pt] at (-2,0) (1) {$\bullet$};
\node [inner sep=0.8pt,outer sep=0.8pt] at (-1,0) (3) {$\bullet$};
\node [inner sep=0.8pt,outer sep=0.8pt] at (0,0) (4) {$\bullet$};
\node [inner sep=0.8pt,outer sep=0.8pt] at (1,0) (5) {$\bullet$};
\node [inner sep=0.8pt,outer sep=0.8pt] at (2,0) (6) {$\bullet$};
\node [inner sep=0.8pt,outer sep=0.8pt] at (3,0) (7) {$\bullet$};
\node [inner sep=0.8pt,outer sep=0.8pt] at (4,0) (8) {$\bullet$};
\node [inner sep=0.8pt,outer sep=0.8pt] at (0,-1) (2) {$\bullet$};
\draw [line width=0.5pt,line cap=round,rounded corners] (1.north west)  rectangle (1.south east);
\draw [line width=0.5pt,line cap=round,rounded corners] (6.north west)  rectangle (6.south east);
\draw [line width=0.5pt,line cap=round,rounded corners,fill=ggrey] (7.north west)  rectangle (7.south east);
\draw [line width=0.5pt,line cap=round,rounded corners,fill=ggrey] (8.north west)  rectangle (8.south east);
\draw (-2,0)--(4,0);
\draw (0,0)--(0,-1);
\node at (-2,0) {$\bullet$};
\node at (2,0) {$\bullet$};
\node [inner sep=0.8pt,outer sep=0.8pt] at (-2,0) (1) {$\bullet$};
\node [inner sep=0.8pt,outer sep=0.8pt] at (-1,0) (3) {$\bullet$};
\node [inner sep=0.8pt,outer sep=0.8pt] at (0,0) (4) {$\bullet$};
\node [inner sep=0.8pt,outer sep=0.8pt] at (1,0) (5) {$\bullet$};
\node [inner sep=0.8pt,outer sep=0.8pt] at (2,0) (6) {$\bullet$};
\node [inner sep=0.8pt,outer sep=0.8pt] at (3,0) (7) {$\bullet$};
\node [inner sep=0.8pt,outer sep=0.8pt] at (4,0) (8) {$\bullet$};
\node [inner sep=0.8pt,outer sep=0.8pt] at (0,-1) (2) {$\bullet$};
\end{tikzpicture}
$$
\end{lemma}

\begin{proof}
These bounded rank cases can be dealt with using counting arguments. First note that the associated buildings are necessarily small. Thus for all cases except the $\sB_3$ case the building $\Delta$ has uniform thickness $q=2$. In the $\sB_3$ case the building either has uniform thickness~$q=2$, or has thickness parameters $q_1=q_2=2$ and $q_3=4$. Proposition~\ref{prop:attaindisplacement} gives an element of $\disp(\theta)$, and thus if $\theta$ is uniclass we know the class~$\bC=\disp(\theta)$. Since the rank is bounded, the class $\bC$ can be explicitly computed (for example, using the Coxeter group algorithms in $\mathsf{MAGMA}$~\cite{MAGMA}). Then Theorem~\ref{thm:counts} gives a formula for $|\Delta_w(\theta)|$ for each $w\in \bC$, and it turns out that for each of the listed diagrams, with one exception, this formula fails to give an integer, a contradiction. The one exception is the $\sB_3$ diagram with uniform thickness~$q=2$, and for this case we provide a different proof below.

We now give the details. Note that in each case the numerical value of the Poincar\'e polynomial can be found by well known factorisations of~$W(q)$. 

Consider the $\sA_3$ diagram. Here $\sigma$ is the order $2$ automorphism and we have $\bC=\disp(\theta)=\Cl^{\sigma}(s_1w_0)$. Explicitly
$
\bC=\{s_2,s_1s_2s_3,s_3s_2s_1,s_1s_2s_1s_3s_2,s_1s_2s_3s_2s_1,s_2s_1s_3s_2s_1\}.
$
Thus $\bC(2^{1/2})=17\sqrt{2}$ and $W(2)=315$. Hence the formula from Theorem~\ref{thm:counts} gives $|\Delta_{s_2}(\theta)|=315/17$, a contradiction.

Consider the $\sB_3$ diagram. The element $s_1s_3$ is of minimal length in $\bC=\disp(\theta)=\Cl(s_1w_0)$. If $\Delta$ has thickness $q_1=q_2=2$ and $q_3=4$ then $\bC(q^{1/2})=94\sqrt{2}$ and $W(q)=16065$. Then $W(q)q_{s_1s_3}^{1/2}/\bC(q^{1/2})$ fails to be an integer. (If $\Delta$ has uniform thickness $q=2$ we compute $\bC(2^{1/2})=54$ and $W(2)=2835$, and then $W(2)q_{s_1s_3}^{1/2}/\bC(2^{1/2})=105$ turns out to be integral -- we discuss this case below). 

Consider the $\sD_4$ diagram with $\sigma=1$. Then $s_1s_3s_4$ is minimal length in $\bC=\disp(\theta)=\Cl(s_1w_0)$. We have $\bC(2^{1/2})=206\sqrt{2}$ and $W(2)=42525$, but then $W(2)2^{3/2}/\bC(2^{1/2})$ is not an integer. For the $\sD_4$ diagram with $\sigma$ of order $2$ we have that $s_1$ is of minimal length in $\bC=\disp(\theta)=\Cl^{\sigma}(s_1w_0)$. Then $\bC(2^{1/2})=210\sqrt{2}$ and $W(2)=42525$, and $|\Delta_{s_1}(\theta)|=405/2$ fails to be an integer.  For the $\sD_5$ diagram we have $s_1s_4s_5$ minimal length in $\bC$, $\bC(2^{1/2})=5456\sqrt{2}$, and $W(2)=22410675$, again giving a contradiction.

Consider the $\sE_6$ diagram. The element $s_1s_2s_6$ is a minimal length element of $\bC=\disp(\theta)=\Cl(s_2w_0)$. We have $\bC(2^{1/2})=2083706\sqrt{2}$ and $W(2)=3126356394525$, giving a contradiction. 

For the $\sE_7$ diagram the element $s_1s_2s_7$ is a minimal length element of $\bC=\disp(\theta)=\Cl(s_1s_2s_5s_7w_0)$. We have $\bC(2^{1/2})=8877543572\sqrt{2}$ and $W(2)=867088089921935556675$, giving a contradiction. Similarly, for the $\sE_8$ diagram the element $s_1s_2s_8$ is of minimal length in $\bC=\disp(\theta)=\Cl(w_{\{2,3,4,5\}}s_8w_0)$. We have $\bC(2^{1/2})=141388830406973542\sqrt{2}$ and $W(2)=254136050560806452394291280512170128125$, a contradiction.

It remains to consider the $\sB_3$ case with uniform thickness $q=2$. By direct calculation the symplectic group $\mathsf{Sp}_6(2)$ has $30$ conjugacy classes, giving rise to the $30$ distinct automorphisms of the building $\Delta$. Discarding the identity, it can be seen by direct calculation (using $\mathsf{MAGMA}$) that $22$ of these automorphisms are non-domestic (by explicitly exhibiting a chamber mapped to opposite). The remaining $7$ automorphisms are domestic, and it turns out that there is a unique exceptional domestic automorphism that is not strongly exceptional domestic. However this automorphism can be conjugated into the Borel, and hence fixes a chamber, showing that it is not uniclass. 
\end{proof}

\begin{thm}\label{thm:uncapped}
Let $\theta$ be a unclass automorphism of a thick irreducible spherical building~$\Delta$. Then $\theta$ is capped.
\end{thm}

\begin{proof}
Suppose that $\theta$ is an uncapped uniclass automorphism. We consider the following cases.
\smallskip

\noindent\textit{Case 1:} Suppose that $\Delta$ is a generalised $d$-gon. Since $\theta$ is uncapped it is necessarily domestic, and so if $d$ is even then $\theta$ is type preserving, and if $d$ is odd then $\theta$ is a duality (see \cite[Theorems~2.6 and 2.7]{PTV}). Thus $\theta$ is exceptional domestic, and so the elements $s_1w_0$ and $s_2w_0$ both lie in $\disp(\theta)$. If $d$ is even these elements are not conjugate in~$W$, and so $\theta$ is not uniclass. 

Thus $d=2n-1$ is odd. We shall regard $\Delta$ as a bipartite graph with diameter $d$ and girth $2d$ (the incidence graph of the point-line geometry). Let $v_0$ be a vertex mapped to an opposite vertex $v_d=v_0^{\theta}$, and let $(v_0,v_1,\ldots,v_{d-1},v_d)$ be a geodesic joining $v_0$ to $v_d$. Then $v_1^{\theta}=v_{d-1}$ (for otherwise the chamber $\{v_0,v_1\}$ is mapped to an opposite chamber, contradicting domesticity). It follows from Lemma~\ref{lem:determineimages} that $v_j^{\theta}=v_{d-j}$ for $j=0,1,\ldots,n-1$, and in particular the chamber $C=\{v_{n-1},v_n\}$ is fixed. Now let $(v_0,u_1,\ldots,u_{d-1},v_d\}$ be another geodesic between $v_0$ and $v_d$ with $u_1\neq v_1$. By the above argument the chamber $D=\{u_{n-1},u_n\}$ is fixed. Then the cycle $A=(v_0,v_1,\ldots,v_{d-1},v_d,u_{d-1},\ldots,u_1,v_0)$ is an apartment, and $C,D$ are opposite chambers in this apartment. Since these chambers are fixed, the apartment is preserved by $\theta$, and thus $v_d^{\theta}=v_0$. In particular $\theta$ is an involution on $A$, and since any two chambers of $\Delta$ lie in such an apartment we conclude that $\theta$ is an involution on $\Delta$, and hence is a polarity of the generalised polygon.

Choose a vertex $y_0\neq v_0,v_2$ adjacent to $v_1$. Thus $y_d:=y_0^{\theta}$ is opposite $y_0$. Let $y_1\neq v_1$ be adjacent to $y_0$ and let $(y_1,y_2,\ldots,y_d)$ be a geodesic. There is a geodesic $(y_n,y_{n-1},\ldots,y_0,v_1,v_0$, $u_1,\ldots,u_{n-3})$ of length $2n-1=d$ and so $y_n$ and $u_{n-3}$ are opposite vertices. Thus there is a unique geodesic $(y_n,a,\ldots,b,u_{n-2})$ of length $d-1$, and $b\neq u_{n-3}$. Since $u_{n-3}^{\theta}=u_{n+2}$ we obtain a closed walk 
$$
(y_n,a,\ldots,b,u_{n-3},u_{n-2},u_{n-1},u_n,u_{n+1},u_{n+2},b^{\theta},\ldots,a^{\theta},y_{n-1},y_n)
$$
of length $2d+2$, preserved by $\theta$. If $b=u_{n-1}$ then also $b^{\theta}=u_n$, and the closed walk gives rise to a cycle of length $2d-2$, a contradiction. Thus the above closed walk is a cycle. Rename the vertices of the cycle as $(z_0,z_1,\ldots,z_{2d+2})$ with $z_0=z_{2d+2}=y_n$ and $z_{2d+1}=y_{n-1}$. Then $u_{n-1}=z_d$ and $u_n=z_{d+1}$ and $z_d^{\theta}=z_{d+1}$ and $z_{d+1}^{\theta}=z_d$, and so the chamber $\{z_d,z_{d+1}\}$ is fixed by $\theta$. Since $\theta$ preserves the cycle it follows that $z_{n}^{\theta}=z_{d-(n-1)}^{\theta}=z_{d+1+(n-1)}=z_{3n-1}$ and $z_{n-1}^{\theta}=z_{d-n}^{\theta}=z_{d+1+n}=z_{3n}$. The distance between $z_n$ and $z_{3n-1}$ is $2n-1=d$ (using a path via $z_{d+1}$) and the distance between $z_{n-1}$ and $z_{3n}$ is also $d$ (using a path via $z_0$). Thus the chamber $\{z_{n-1},z_n\}$ is mapped to an opposite chamber, a contradiction.

$$
\begin{tikzpicture}[xscale=1.4,yscale=1.2]
\node [inner sep=0.8pt,outer sep=0.8pt] at (-4,0) (v0) {$v_0$};
\node [inner sep=0.8pt,outer sep=0.8pt] at (4,0) (vd) {$v_d$};
\node [inner sep=0.8pt,outer sep=0.8pt] at (-3,1) (v1) {$v_1$};
\node [inner sep=0.8pt,outer sep=0.8pt] at (-2,2) (v2) {$v_2$};
\node [inner sep=0.8pt,outer sep=0.8pt] at (-0.5,2) (vn-1) {$v_{n-1}$};
\node [inner sep=0.8pt,outer sep=0.8pt] at (0.5,2) (vn) {$v_n$};
\node [inner sep=0.8pt,outer sep=0.8pt] at (2,2) (vd-2) {$v_{d-2}$};
\node [inner sep=0.8pt,outer sep=0.8pt] at (3,1) (vd-1) {$v_{d-1}$};
\node [inner sep=0.8pt,outer sep=0.8pt] at (-3.25,-1) (u1) {$u_1$};
\node [inner sep=0.8pt,outer sep=0.8pt] at (-2.5,-2) (un-3) {$u_{n-3}$};
\node [inner sep=0.8pt,outer sep=0.8pt] at (-1.5,-3) (un-2) {$u_{n-2}$};
\node [inner sep=0.8pt,outer sep=0.8pt] at (-0.5,-3) (un-1) {$u_{n-1}$};
\node [inner sep=0.8pt,outer sep=0.8pt] at (0.5,-3) (un) {$u_n$};
\node [inner sep=0.8pt,outer sep=0.8pt] at (1.5,-3) (un+1) {$u_{n+1}$};
\node [inner sep=0.8pt,outer sep=0.8pt] at (2.5,-2) (un+2) {$u_{n+2}$};
\node [inner sep=0.8pt,outer sep=0.8pt] at (3.25,-1) (ud-1) {$u_{d-1}$};
\node [inner sep=0.8pt,outer sep=0.8pt] at (-3,0) (y0) {$y_0$};
\node [inner sep=0.8pt,outer sep=0.8pt] at (-2,0) (y1) {$y_1$};
\node [inner sep=0.8pt,outer sep=0.8pt] at (3,0) (yd) {$y_d$};
\node [inner sep=0.8pt,outer sep=0.8pt] at (2,0) (yd-1) {$y_{d-1}$};
\node [inner sep=0.8pt,outer sep=0.8pt] at (-0.5,0) (yn-1) {$y_{n-1}$};
\node [inner sep=0.8pt,outer sep=0.8pt] at (0.5,0) (yn) {$y_n$};
\node [inner sep=0.8pt,outer sep=0.8pt] at (-0.5,-1) (a) {$a$};
\node [inner sep=0.8pt,outer sep=0.8pt] at (0.5,-1) (atheta) {$a^{\theta}$};
\node [inner sep=0.8pt,outer sep=0.8pt] at (-1.5,-2) (b) {$b$};
\node [inner sep=0.8pt,outer sep=0.8pt] at (1.5,-2) (btheta) {$b^{\theta}$};
\draw (v0)--(v1)--(v2);
\draw [dashed] (v2)--(vn-1);
\draw (vn-1)--(vn);
\draw [dashed] (vn)--(vd-2);
\draw (vd-2)--(vd-1)--(vd);
\draw (v0)--(u1);
\draw [dashed] (u1)--(un-3);
\draw (un-3)--(un-2)--(un-1)--(un)--(un+1)--(un+2);
\draw [dashed] (un+2)--(ud-1);
\draw (ud-1)--(vd);
\draw (v1)--(y0)--(y1);
\draw [dashed] (y1)--(yn-1);
\draw (yn-1)--(yn);
\draw [dashed] (yn)--(yd-1);
\draw (yd-1)--(yd)--(vd-1);
\draw (yn-1)--(atheta);
\draw (yn)--(a);
\draw [dashed] (a)--(b);
\draw [dashed] (atheta)--(btheta);
\draw (b)--(un-2);
\draw (btheta)--(un+1);
\end{tikzpicture}
$$

%
%
Thus we may henceforth assume that $\Delta$ has rank at least $3$ and that $\theta$ is uncapped. In particular $\Delta$ is a small building. 
\smallskip

\noindent\textit{Case 2:} Suppose that $\theta$ is strongly exceptional domestic. That is, $\theta$ maps panels of each type to opposite panels, and so $sw_0\in \disp(\theta)$ for all $s\in S$. This forces the companion automorphism $\sigma$ to be opposition. Moreover, if $\Delta$ is not simply laced then $\disp(\theta)$ contains elements $sw_0$ and $s'w_0$ with $s,s'\in S$ not conjugate, and so $sw_0$ and $s'w_0$ are not $\sigma$-conjugate, a contradiction. Thus we may assume that $\Delta$ is simply laced, and $\disp(\theta)=\Cl^{\sigma}(sw_0)$ for any $s\in S$. Associate a crystallographic root system $\Phi$ to $(W,S)$. By Lemma~\ref{lem:highestroot} the element $w_{S\backslash\wp}$ is of minimal length in the class $\disp(\theta)$, and by Proposition~\ref{prop:anisotropic} there is a residue $R$ of type $S\backslash\wp$ such that $\theta|_R$ is anisotropic. 
If $R$ has an irreducible component of rank at least~$2$ then we obtain a contradiction with \cite[Theorem~5.1]{DPV:13} (since the building $\Delta$ is small, and in particular is finite). This leaves the $\sA_3$ and $\sD_4$ cases, and these are eliminated by Lemma~\ref{lem:lowrank}.
\smallskip

\noindent\textit{Case 3:} Suppose that $\theta$ has decorated opposition diagram
$$
\begin{tikzpicture}[scale=0.5,baseline=-0.5ex]
\node at (0,0.3) {};
\node at (0,-0.5) {};
\node [inner sep=0.8pt,outer sep=0.8pt] at (-5,0) (-5) {$\bullet$};
\node [inner sep=0.8pt,outer sep=0.8pt] at (-4,0) (-4) {$\bullet$};
\node [inner sep=0.8pt,outer sep=0.8pt] at (-3,0) (-3) {$\bullet$};
\node [inner sep=0.8pt,outer sep=0.8pt] at (-2,0) (-2) {$\bullet$};
\node [inner sep=0.8pt,outer sep=0.8pt] at (0,0) (-1) {$\bullet$};
\node [inner sep=0.8pt,outer sep=0.8pt] at (1,0) (1) {$\bullet$};
\node [inner sep=0.8pt,outer sep=0.8pt] at (2,0) (2) {$\bullet$};
\node [inner sep=0.8pt,outer sep=0.8pt] at (3,0) (3) {$\bullet$};
\draw [line width=0.5pt,line cap=round,rounded corners,fill=ggrey] (-5.north west)  rectangle (-5.south east);
\draw [line width=0.5pt,line cap=round,rounded corners,fill=ggrey] (-4.north west)  rectangle (-4.south east);
\draw [line width=0.5pt,line cap=round,rounded corners,fill=ggrey] (-3.north west)  rectangle (-3.south east);
\draw [line width=0.5pt,line cap=round,rounded corners,fill=ggrey] (-2.north west)  rectangle (-2.south east);
\draw [line width=0.5pt,line cap=round,rounded corners,fill=ggrey] (-1.north west)  rectangle (-1.south east);
\draw [line width=0.5pt,line cap=round,rounded corners,fill=ggrey] (1.north west)  rectangle (1.south east);
\draw [line width=0.5pt,line cap=round,rounded corners,fill=ggrey] (2.north west)  rectangle (2.south east);
\draw [line width=0.5pt,line cap=round,rounded corners] (3.north west)  rectangle (3.south east);
\draw (-5,0)--(-1.5,0);
\draw (-0.5,0)--(2,0);
\draw (2,0.07)--(3,0.07);
\draw (2,-0.07)--(3,-0.07);
\draw [style=dashed] (-1.5,0)--(-0.5,0);
\node [inner sep=0.8pt,outer sep=0.8pt] at (-5,0) (-5) {$\bullet$};
\node [inner sep=0.8pt,outer sep=0.8pt] at (-4,0) (-4) {$\bullet$};
\node [inner sep=0.8pt,outer sep=0.8pt] at (-3,0) (-3) {$\bullet$};
\node [inner sep=0.8pt,outer sep=0.8pt] at (-2,0) (-2) {$\bullet$};
\node [inner sep=0.8pt,outer sep=0.8pt] at (0,0) (-1) {$\bullet$};
\node [inner sep=0.8pt,outer sep=0.8pt] at (1,0) (1) {$\bullet$};
\node [inner sep=0.8pt,outer sep=0.8pt] at (2,0) (2) {$\bullet$};
\node [inner sep=0.8pt,outer sep=0.8pt] at (3,0) (3) {$\bullet$};
\end{tikzpicture}\quad\text{or}\quad 
\begin{tikzpicture}[scale=0.5,baseline=-0.5ex]
\node at (0,0.3) {};
\node at (0,-0.5) {};
\node [inner sep=0.8pt,outer sep=0.8pt] at (-1.5,0) (1) {$\bullet$};
\node [inner sep=0.8pt,outer sep=0.8pt] at (-0.5,0) (2) {$\bullet$};
\node [inner sep=0.8pt,outer sep=0.8pt] at (0.5,0) (3) {$\bullet$};
\node [inner sep=0.8pt,outer sep=0.8pt] at (1.5,0) (4) {$\bullet$};
\draw [line width=0.5pt,line cap=round,rounded corners,fill=ggrey] (1.north west)  rectangle (1.south east);
\draw [line width=0.5pt,line cap=round,rounded corners,fill=ggrey] (2.north west)  rectangle (2.south east);
\draw [line width=0.5pt,line cap=round,rounded corners] (3.north west)  rectangle (3.south east);
\draw [line width=0.5pt,line cap=round,rounded corners] (4.north west)  rectangle (4.south east);
\draw (-1.5,0)--(-0.5,0);
\draw (0.5,0)--(1.5,0);
\draw (-0.5,0.07)--(0.5,0.07);
\draw (-0.5,-0.07)--(0.5,-0.07);
\node [inner sep=0.8pt,outer sep=0.8pt] at (-1.5,0) (1) {$\bullet$};
\node [inner sep=0.8pt,outer sep=0.8pt] at (-0.5,0) (2) {$\bullet$};
\node [inner sep=0.8pt,outer sep=0.8pt] at (0.5,0) (3) {$\bullet$};
\node [inner sep=0.8pt,outer sep=0.8pt] at (1.5,0) (4) {$\bullet$};
\end{tikzpicture}
$$
We can associate a reduced crystallographic root system $\Phi$ to the decorated opposition diagram of $\theta$ in such a way that the nodes corresponding to the long simple roots are shaded (thus for Coxeter type $\sB_n$ we choose a $\sB_n$ root system, and so in this case $\wp=\{2\}$). Then $\disp(\theta)=\Cl(s_iw_0)$ for any $i$ with $\alpha_i$ long. As in Case~2 there is a residue $R$ of type $S\backslash\wp$ stabilised by $\theta$ such that $\theta|_R$ is anisotropic, giving a contradiction is $R$ has an irreducible component of rank at least~$2$. This leaves the $\sB_3$ case, which is eliminated by Lemma~\ref{lem:lowrank}.
\smallskip

\noindent\textit{Case 4:} Suppose that $\Delta$ is of type $\sD_n$ and that $\theta$ has decorated opposition diagram 
$$
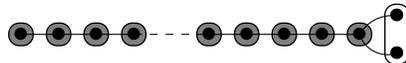
\begin{tikzpicture}[scale=0.5,baseline=-0.5ex]
\node at (0,0.8) {};
\node at (0,-0.8) {};
\node [inner sep=0.8pt,outer sep=0.8pt] at (-5,0) (-5) {$\bullet$};
\node [inner sep=0.8pt,outer sep=0.8pt] at (-4,0) (-4) {$\bullet$};
\node [inner sep=0.8pt,outer sep=0.8pt] at (-3,0) (-3) {$\bullet$};
\node [inner sep=0.8pt,outer sep=0.8pt] at (-2,0) (-2) {$\bullet$};
\node [inner sep=0.8pt,outer sep=0.8pt] at (0,0) (-1) {$\bullet$};
\node [inner sep=0.8pt,outer sep=0.8pt] at (1,0) (1) {$\bullet$};
\node [inner sep=0.8pt,outer sep=0.8pt] at (2,0) (2) {$\bullet$};
\node [inner sep=0.8pt,outer sep=0.8pt] at (3,0) (3) {$\bullet$};
\node [inner sep=0.8pt,outer sep=0.8pt] at (4,0) (4) {$\bullet$};
\node [inner sep=0.8pt,outer sep=0.8pt] at (5,0.5) (5a) {$\bullet$};
\node [inner sep=0.8pt,outer sep=0.8pt] at (5,-0.5) (5b) {$\bullet$};
\draw [line width=0.5pt,line cap=round,rounded corners,fill=ggrey] (-4.north west)  rectangle (-4.south east);
\draw [line width=0.5pt,line cap=round,rounded corners,fill=ggrey] (-2.north west)  rectangle (-2.south east);
\draw [line width=0.5pt,line cap=round,rounded corners,fill=ggrey] (2.north west)  rectangle (2.south east);
\draw [line width=0.5pt,line cap=round,rounded corners,fill=ggrey] (4.north west)  rectangle (4.south east);
\draw [line width=0.5pt,line cap=round,rounded corners,fill=ggrey] (-5.north west)  rectangle (-5.south east);
\draw [line width=0.5pt,line cap=round,rounded corners,fill=ggrey] (-3.north west)  rectangle (-3.south east);
\draw [line width=0.5pt,line cap=round,rounded corners,fill=ggrey] (-1.north west)  rectangle (-1.south east);
\draw [line width=0.5pt,line cap=round,rounded corners,fill=ggrey] (1.north west)  rectangle (1.south east);
\draw [line width=0.5pt,line cap=round,rounded corners,fill=ggrey] (3.north west)  rectangle (3.south east);
\draw [line width=0.5pt,line cap=round,rounded corners] (5a.north west)  rectangle (5b.south east);
\draw (-5,0)--(-2,0);
\draw (0,0)--(4,0);
\draw (4,0) to [bend left] (5,0.5);
\draw (4,0) to [bend right=45] (5,-0.5);
\draw [style=dashed] (-2,0)--(0,0);
\node at (5,0.5) {$\bullet$};
\node at (5,-0.5) {$\bullet$};
\node [inner sep=0.8pt,outer sep=0.8pt] at (-5,0) (-5) {$\bullet$};
\node [inner sep=0.8pt,outer sep=0.8pt] at (-4,0) (-4) {$\bullet$};
\node [inner sep=0.8pt,outer sep=0.8pt] at (-3,0) (-3) {$\bullet$};
\node [inner sep=0.8pt,outer sep=0.8pt] at (-2,0) (-2) {$\bullet$};
\node [inner sep=0.8pt,outer sep=0.8pt] at (0,0) (-1) {$\bullet$};
\node [inner sep=0.8pt,outer sep=0.8pt] at (1,0) (1) {$\bullet$};
\node [inner sep=0.8pt,outer sep=0.8pt] at (2,0) (2) {$\bullet$};
\node [inner sep=0.8pt,outer sep=0.8pt] at (3,0) (3) {$\bullet$};
\node [inner sep=0.8pt,outer sep=0.8pt] at (4,0) (4) {$\bullet$};
\node [inner sep=0.8pt,outer sep=0.8pt] at (5,0.5) (5a) {$\bullet$};
\node [inner sep=0.8pt,outer sep=0.8pt] at (5,-0.5) (5b) {$\bullet$};
\end{tikzpicture}
$$
Thus $\disp(\theta)=\Cl^{\sigma}(s_1w_0)$ (with $\sigma$ the identity if $n$ is odd, and of order~$2$ if $n$ is even). If $n\geq 5$ then by Lemma~\ref{lem:classfacts}(1) and Proposition~\ref{prop:anisotropic} there is a type $\sA_1\times\sD_{n-3}$ residue $R$ such that $\theta|_R$ is anisotropic, a contradiction if $n>5$. The cases $n=4,5$ are eliminated by Lemma~\ref{lem:lowrank}.
\smallskip

\noindent\textit{Case 5:} Suppose that $\theta$ is an uncapped uniclass automorphism of a small building of type $\sB_n$ with decorated opposition diagram
$$
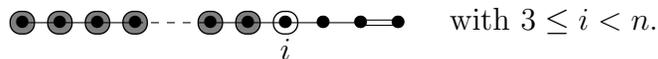
\begin{tikzpicture}[scale=0.5,baseline=-0.5ex]
\node at (0,0.5) {};
\node [inner sep=0.8pt,outer sep=0.8pt] at (-5,0) (-5) {$\bullet$};
\node [inner sep=0.8pt,outer sep=0.8pt] at (-4,0) (-4) {$\bullet$};
\node [inner sep=0.8pt,outer sep=0.8pt] at (-3,0) (-3) {$\bullet$};
\node [inner sep=0.8pt,outer sep=0.8pt] at (-2,0) (-2) {$\bullet$};
\node [inner sep=0.8pt,outer sep=0.8pt] at (0,0) (-1) {$\bullet$};
\node [inner sep=0.8pt,outer sep=0.8pt] at (1,0) (1) {$\bullet$};
\node [inner sep=0.8pt,outer sep=0.8pt] at (2,0) (2) {$\bullet$};
\node [inner sep=0.8pt,outer sep=0.8pt] at (3,0) (3) {$\bullet$};
\node [inner sep=0.8pt,outer sep=0.8pt] at (4,0) (4) {$\bullet$};
\node [inner sep=0.8pt,outer sep=0.8pt] at (5,0) (5) {$\bullet$};
\node at (2,-0.7) {$i$};
\draw [line width=0.5pt,line cap=round,rounded corners,fill=ggrey] (-5.north west)  rectangle (-5.south east);
\draw [line width=0.5pt,line cap=round,rounded corners,fill=ggrey] (-4.north west)  rectangle (-4.south east);
\draw [line width=0.5pt,line cap=round,rounded corners,fill=ggrey] (-3.north west)  rectangle (-3.south east);
\draw [line width=0.5pt,line cap=round,rounded corners,fill=ggrey] (-2.north west)  rectangle (-2.south east);
\draw [line width=0.5pt,line cap=round,rounded corners,fill=ggrey] (-1.north west)  rectangle (-1.south east);
\draw [line width=0.5pt,line cap=round,rounded corners,fill=ggrey] (1.north west)  rectangle (1.south east);
\draw [line width=0.5pt,line cap=round,rounded corners] (2.north west)  rectangle (2.south east);
\draw (-5,0)--(-2,0);
\draw (0,0)--(4,0);
\draw (4,0.07)--(5,0.07);
\draw (4,-0.07)--(5,-0.07);
\draw [style=dashed] (-2,0)--(0,0);
\node [inner sep=0.8pt,outer sep=0.8pt] at (-5,0) (-5) {$\bullet$};
\node [inner sep=0.8pt,outer sep=0.8pt] at (-4,0) (-4) {$\bullet$};
\node [inner sep=0.8pt,outer sep=0.8pt] at (-3,0) (-3) {$\bullet$};
\node [inner sep=0.8pt,outer sep=0.8pt] at (-2,0) (-2) {$\bullet$};
\node [inner sep=0.8pt,outer sep=0.8pt] at (0,0) (-1) {$\bullet$};
\node [inner sep=0.8pt,outer sep=0.8pt] at (1,0) (1) {$\bullet$};
\node [inner sep=0.8pt,outer sep=0.8pt] at (2,0) (2) {$\bullet$};
\node [inner sep=0.8pt,outer sep=0.8pt] at (3,0) (3) {$\bullet$};
\node [inner sep=0.8pt,outer sep=0.8pt] at (4,0) (4) {$\bullet$};
\node [inner sep=0.8pt,outer sep=0.8pt] at (5,0) (5) {$\bullet$};
\end{tikzpicture}\quad\text{ with $3\leq i< n$.} 
$$
By Proposition~\ref{prop:attaindisplacement} and the decorated opposition diagram we have $\disp(\theta)=\Cl(s_1w_{\geq i+1}w_0)$. Then by Lemma~\ref{lem:classfacts}(2) the element $s_1w_{\geq n-i+3}$ is minimal length in $\disp(\theta)$, and by Proposition~\ref{prop:anisotropic} there is a type $\sA_1\times  \sB_{i-2}$ residue $R$ stabilised by $\theta$ such that $\theta|_R$ is anisotropic, a contradiction if $i>3$. 

The case $i=3$ requires a different approach. In this case Lemma~\ref{lem:classfacts}(2) gives that $\disp(\theta)=\Cl(s_1w_{\geq 4}w_0)=\Cl(s_1s_n)$. Since $s_1s_n$ is conjugate to $s_2s_n$ there is a chamber $C$ with $\delta(C,C^{\theta})=s_2s_n$. In particular, the type $1$ vertex of $C$ is fixed, and so the residue $\Delta'=\Res_{S\backslash\{1\}}(C)$ (a building of type $\sB_{n-1}$) is stabilised by $\theta$. By Proposition~\ref{prop:residual} the automorphism $\theta|_R$ is uniclass, and hence $\disp(\theta|_R)=\Cl_{W_{\geq 2}}(s_2s_n)=\Cl_{W_{\geq 2}}(s_2w_{\geq 5}w_{\geq 2})$, 
where the final equality follows from Lemma~\ref{lem:classfacts}(2) applied to the group $W_{\geq 2}$ with $i=3$ (specifically, the fact that $s_2s_n$ is conjugate to $s_2w_{\geq 5}w_{\geq 2}$). Thus $\theta|_{R}$ is an uncapped automorphism of a $\sB_{n-1}$ building with the first three nodes encircled and the first two nodes shaded. That is, $\theta|_{R}$ is of the same ``type'' as $\theta$, but of rank one less (note that $i=3$ for both $\theta$ and $\theta|_{R}$). Continuing inductively we eventually obtain a restriction of $\theta$ to a residue of type $\sB_3$ with decorated opposition diagram \begin{tikzpicture}[scale=0.5,baseline=-0.5ex]
\node at (1,0.3) {};
\node at (1,-0.5) {};
\node [inner sep=0.8pt,outer sep=0.8pt] at (1,0) (1) {$\bullet$};
\node [inner sep=0.8pt,outer sep=0.8pt] at (2,0) (2) {$\bullet$};
\node [inner sep=0.8pt,outer sep=0.8pt] at (3,0) (3) {$\bullet$};
\draw [line width=0.5pt,line cap=round,rounded corners,fill=ggrey] (1.north west)  rectangle (1.south east);
\draw [line width=0.5pt,line cap=round,rounded corners,fill=ggrey] (2.north west)  rectangle (2.south east);
\draw [line width=0.5pt,line cap=round,rounded corners] (3.north west)  rectangle (3.south east);
\draw (1,0)--(2,0);
\draw (2,0.07)--(3,0.07);
\draw (2,-0.07)--(3,-0.07);
\node [inner sep=0.8pt,outer sep=0.8pt] at (1,0) (1) {$\bullet$};
\node [inner sep=0.8pt,outer sep=0.8pt] at (2,0) (2) {$\bullet$};
\node [inner sep=0.8pt,outer sep=0.8pt] at (3,0) (3) {$\bullet$};
\end{tikzpicture}, which is impossible by Lemma~\ref{lem:lowrank}.

\smallskip

\noindent\textit{Case 6:} Suppose that $\theta$ has decorated opposition diagram one of the following:
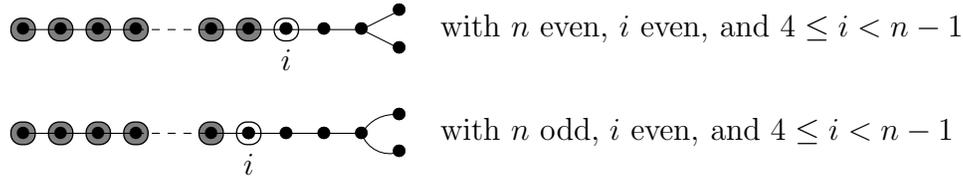
\begin{align*}
&\begin{tikzpicture}[scale=0.5,baseline=-0.5ex]
\node at (0,0.8) {};
\node [inner sep=0.8pt,outer sep=0.8pt] at (-5,0) (-5) {$\bullet$};
\node [inner sep=0.8pt,outer sep=0.8pt] at (-4,0) (-4) {$\bullet$};
\node [inner sep=0.8pt,outer sep=0.8pt] at (-3,0) (-3) {$\bullet$};
\node [inner sep=0.8pt,outer sep=0.8pt] at (-2,0) (-2) {$\bullet$};
\node [inner sep=0.8pt,outer sep=0.8pt] at (0,0) (-1) {$\bullet$};
\node [inner sep=0.8pt,outer sep=0.8pt] at (1,0) (1) {$\bullet$};
\node [inner sep=0.8pt,outer sep=0.8pt] at (2,0) (2) {$\bullet$};
\node [inner sep=0.8pt,outer sep=0.8pt] at (3,0) (3) {$\bullet$};
\node [inner sep=0.8pt,outer sep=0.8pt] at (4,0) (4) {$\bullet$};
\node [inner sep=0.8pt,outer sep=0.8pt] at (5,0.5) (5a) {$\bullet$};
\node [inner sep=0.8pt,outer sep=0.8pt] at (5,-0.5) (5b) {$\bullet$};
\draw [line width=0.5pt,line cap=round,rounded corners,fill=ggrey] (-5.north west)  rectangle (-5.south east);
\draw [line width=0.5pt,line cap=round,rounded corners,fill=ggrey] (-4.north west)  rectangle (-4.south east);
\draw [line width=0.5pt,line cap=round,rounded corners,fill=ggrey] (-3.north west)  rectangle (-3.south east);
\draw [line width=0.5pt,line cap=round,rounded corners,fill=ggrey] (-2.north west)  rectangle (-2.south east);
\draw [line width=0.5pt,line cap=round,rounded corners,fill=ggrey] (-1.north west)  rectangle (-1.south east);
\draw [line width=0.5pt,line cap=round,rounded corners,fill=ggrey] (1.north west)  rectangle (1.south east);
\draw [line width=0.5pt,line cap=round,rounded corners] (2.north west)  rectangle (2.south east);
\node [below] at (2,-0.25) {$i$};
\draw (-5,0)--(-2,0);
\draw (0,0)--(4,0);
\draw (4,0) to (5,0.5);
\draw (4,0) to   (5,-0.5);
\draw [style=dashed] (-2,0)--(0,0);
\node at (2,0) {$\bullet$};
\node [inner sep=0.8pt,outer sep=0.8pt] at (-5,0) (-5) {$\bullet$};
\node [inner sep=0.8pt,outer sep=0.8pt] at (-4,0) (-4) {$\bullet$};
\node [inner sep=0.8pt,outer sep=0.8pt] at (-3,0) (-3) {$\bullet$};
\node [inner sep=0.8pt,outer sep=0.8pt] at (-2,0) (-2) {$\bullet$};
\node [inner sep=0.8pt,outer sep=0.8pt] at (0,0) (-1) {$\bullet$};
\node [inner sep=0.8pt,outer sep=0.8pt] at (1,0) (1) {$\bullet$};
\node [inner sep=0.8pt,outer sep=0.8pt] at (2,0) (2) {$\bullet$};
\node [inner sep=0.8pt,outer sep=0.8pt] at (3,0) (3) {$\bullet$};
\node [inner sep=0.8pt,outer sep=0.8pt] at (4,0) (4) {$\bullet$};
\node [inner sep=0.8pt,outer sep=0.8pt] at (5,0.5) (5a) {$\bullet$};
\node [inner sep=0.8pt,outer sep=0.8pt] at (5,-0.5) (5b) {$\bullet$};
\end{tikzpicture}
\quad\text{with $n$ even, $i$ even, and $4\leq i< n-1$}\\
&\begin{tikzpicture}[scale=0.5,baseline=-0.5ex]
\node at (0,0.8) {};
\node [inner sep=0.8pt,outer sep=0.8pt] at (-5,0) (-5) {$\bullet$};
\node [inner sep=0.8pt,outer sep=0.8pt] at (-4,0) (-4) {$\bullet$};
\node [inner sep=0.8pt,outer sep=0.8pt] at (-3,0) (-3) {$\bullet$};
\node [inner sep=0.8pt,outer sep=0.8pt] at (-2,0) (-2) {$\bullet$};
\node [inner sep=0.8pt,outer sep=0.8pt] at (0,0) (-1) {$\bullet$};
\node [inner sep=0.8pt,outer sep=0.8pt] at (1,0) (1) {$\bullet$};
\node [inner sep=0.8pt,outer sep=0.8pt] at (2,0) (2) {$\bullet$};
\node [inner sep=0.8pt,outer sep=0.8pt] at (3,0) (3) {$\bullet$};
\node [inner sep=0.8pt,outer sep=0.8pt] at (4,0) (4) {$\bullet$};
\node [inner sep=0.8pt,outer sep=0.8pt] at (5,0.5) (5a) {$\bullet$};
\node [inner sep=0.8pt,outer sep=0.8pt] at (5,-0.5) (5b) {$\bullet$};
\draw [line width=0.5pt,line cap=round,rounded corners,fill=ggrey] (-5.north west)  rectangle (-5.south east);
\draw [line width=0.5pt,line cap=round,rounded corners,fill=ggrey] (-4.north west)  rectangle (-4.south east);
\draw [line width=0.5pt,line cap=round,rounded corners,fill=ggrey] (-3.north west)  rectangle (-3.south east);
\draw [line width=0.5pt,line cap=round,rounded corners,fill=ggrey] (-2.north west)  rectangle (-2.south east);
\draw [line width=0.5pt,line cap=round,rounded corners,fill=ggrey] (-1.north west)  rectangle (-1.south east);
\draw [line width=0.5pt,line cap=round,rounded corners] (1.north west)  rectangle (1.south east);
\node [below] at (1,-0.25) {$i$};
\draw (-5,0)--(-2,0);
\draw (0,0)--(4,0);
\draw (4,0) to [bend left] (5,0.5);
\draw (4,0) to [bend right=45] (5,-0.5);
\draw [style=dashed] (-2,0)--(0,0);
\node at (1,0) {$\bullet$};
\node [inner sep=0.8pt,outer sep=0.8pt] at (-5,0) (-5) {$\bullet$};
\node [inner sep=0.8pt,outer sep=0.8pt] at (-4,0) (-4) {$\bullet$};
\node [inner sep=0.8pt,outer sep=0.8pt] at (-3,0) (-3) {$\bullet$};
\node [inner sep=0.8pt,outer sep=0.8pt] at (-2,0) (-2) {$\bullet$};
\node [inner sep=0.8pt,outer sep=0.8pt] at (0,0) (-1) {$\bullet$};
\node [inner sep=0.8pt,outer sep=0.8pt] at (1,0) (1) {$\bullet$};
\node [inner sep=0.8pt,outer sep=0.8pt] at (2,0) (2) {$\bullet$};
\node [inner sep=0.8pt,outer sep=0.8pt] at (3,0) (3) {$\bullet$};
\node [inner sep=0.8pt,outer sep=0.8pt] at (4,0) (4) {$\bullet$};
\node [inner sep=0.8pt,outer sep=0.8pt] at (5,0.5) (5a) {$\bullet$};
\node [inner sep=0.8pt,outer sep=0.8pt] at (5,-0.5) (5b) {$\bullet$};
\end{tikzpicture}\quad\text{with $n$ odd, $i$ even, and $4\leq i< n-1$}
\end{align*}
Here the diagram automorphism associated to theta is $\sigma=1$. Applying Lemma~\ref{lem:classfacts}(3) and the argument of Case 5 we obtain a residue $R$ of type $\sA_1\times \sD_{i-2}$ with $\theta|_R$ anisotropic. If $i>4$ (so $i\geq 6$) we obtain a contradiction. If $i=4$ then an argument almost identical to the $i=3$ case of Case 5 gives a residue of type $\sD_5$ on which $\theta$ acts with decorated opposition diagram being the $\sD_5$ diagram listed in Lemma~\ref{lem:lowrank}, a contradiction.
\smallskip

\noindent\textit{Case 7:} Suppose that $\theta$ has decorated opposition diagram one of the following:
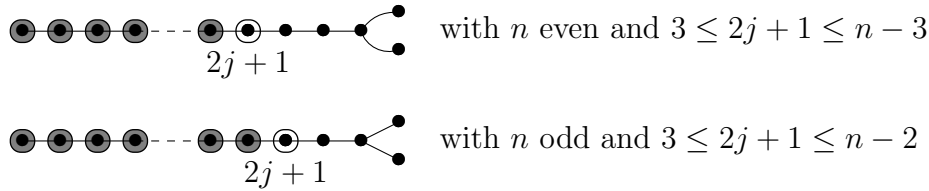
\begin{align*}
&\begin{tikzpicture}[scale=0.5,baseline=-0.5ex]
\node at (0,0.8) {};
\node [inner sep=0.8pt,outer sep=0.8pt] at (-5,0) (-5) {$\bullet$};
\node [inner sep=0.8pt,outer sep=0.8pt] at (-4,0) (-4) {$\bullet$};
\node [inner sep=0.8pt,outer sep=0.8pt] at (-3,0) (-3) {$\bullet$};
\node [inner sep=0.8pt,outer sep=0.8pt] at (-2,0) (-2) {$\bullet$};
\node [inner sep=0.8pt,outer sep=0.8pt] at (0,0) (-1) {$\bullet$};
\node [inner sep=0.8pt,outer sep=0.8pt] at (1,0) (1) {$\bullet$};
\node [inner sep=0.8pt,outer sep=0.8pt] at (2,0) (2) {$\bullet$};
\node [inner sep=0.8pt,outer sep=0.8pt] at (3,0) (3) {$\bullet$};
\node [inner sep=0.8pt,outer sep=0.8pt] at (4,0) (4) {$\bullet$};
\node [inner sep=0.8pt,outer sep=0.8pt] at (5,0.5) (5a) {$\bullet$};
\node [inner sep=0.8pt,outer sep=0.8pt] at (5,-0.5) (5b) {$\bullet$};
\draw [line width=0.5pt,line cap=round,rounded corners,fill=ggrey] (-5.north west)  rectangle (-5.south east);
\draw [line width=0.5pt,line cap=round,rounded corners,fill=ggrey] (-4.north west)  rectangle (-4.south east);
\draw [line width=0.5pt,line cap=round,rounded corners,fill=ggrey] (-3.north west)  rectangle (-3.south east);
\draw [line width=0.5pt,line cap=round,rounded corners,fill=ggrey] (-2.north west)  rectangle (-2.south east);
\draw [line width=0.5pt,line cap=round,rounded corners,fill=ggrey] (-1.north west)  rectangle (-1.south east);
\draw [line width=0.5pt,line cap=round,rounded corners] (1.north west)  rectangle (1.south east);
\node [below] at (1,-0.25) {$2j+1$};
\draw (-5,0)--(-2,0);
\draw (0,0)--(4,0);
\draw (4,0) to [bend left] (5,0.5);
\draw (4,0) to [bend right=45] (5,-0.5);
\draw [style=dashed] (-2,0)--(0,0);
\node at (1,0) {$\bullet$};
\node [inner sep=0.8pt,outer sep=0.8pt] at (-5,0) (-5) {$\bullet$};
\node [inner sep=0.8pt,outer sep=0.8pt] at (-4,0) (-4) {$\bullet$};
\node [inner sep=0.8pt,outer sep=0.8pt] at (-3,0) (-3) {$\bullet$};
\node [inner sep=0.8pt,outer sep=0.8pt] at (-2,0) (-2) {$\bullet$};
\node [inner sep=0.8pt,outer sep=0.8pt] at (0,0) (-1) {$\bullet$};
\node [inner sep=0.8pt,outer sep=0.8pt] at (1,0) (1) {$\bullet$};
\node [inner sep=0.8pt,outer sep=0.8pt] at (2,0) (2) {$\bullet$};
\node [inner sep=0.8pt,outer sep=0.8pt] at (3,0) (3) {$\bullet$};
\node [inner sep=0.8pt,outer sep=0.8pt] at (4,0) (4) {$\bullet$};
\node [inner sep=0.8pt,outer sep=0.8pt] at (5,0.5) (5a) {$\bullet$};
\node [inner sep=0.8pt,outer sep=0.8pt] at (5,-0.5) (5b) {$\bullet$};
\end{tikzpicture}\quad\text{with $n$ even and $3\leq 2j+1\leq n-3$ }\\
&\begin{tikzpicture}[scale=0.5,baseline=-0.5ex]
\node at (0,0.8) {};
\node [inner sep=0.8pt,outer sep=0.8pt] at (-5,0) (-5) {$\bullet$};
\node [inner sep=0.8pt,outer sep=0.8pt] at (-4,0) (-4) {$\bullet$};
\node [inner sep=0.8pt,outer sep=0.8pt] at (-3,0) (-3) {$\bullet$};
\node [inner sep=0.8pt,outer sep=0.8pt] at (-2,0) (-2) {$\bullet$};
\node [inner sep=0.8pt,outer sep=0.8pt] at (0,0) (-1) {$\bullet$};
\node [inner sep=0.8pt,outer sep=0.8pt] at (1,0) (1) {$\bullet$};
\node [inner sep=0.8pt,outer sep=0.8pt] at (2,0) (2) {$\bullet$};
\node [inner sep=0.8pt,outer sep=0.8pt] at (3,0) (3) {$\bullet$};
\node [inner sep=0.8pt,outer sep=0.8pt] at (4,0) (4) {$\bullet$};
\node [inner sep=0.8pt,outer sep=0.8pt] at (5,0.5) (5a) {$\bullet$};
\node [inner sep=0.8pt,outer sep=0.8pt] at (5,-0.5) (5b) {$\bullet$};
\draw [line width=0.5pt,line cap=round,rounded corners,fill=ggrey] (-5.north west)  rectangle (-5.south east);
\draw [line width=0.5pt,line cap=round,rounded corners,fill=ggrey] (-4.north west)  rectangle (-4.south east);
\draw [line width=0.5pt,line cap=round,rounded corners,fill=ggrey] (-3.north west)  rectangle (-3.south east);
\draw [line width=0.5pt,line cap=round,rounded corners,fill=ggrey] (-2.north west)  rectangle (-2.south east);
\draw [line width=0.5pt,line cap=round,rounded corners,fill=ggrey] (-1.north west)  rectangle (-1.south east);
\draw [line width=0.5pt,line cap=round,rounded corners,fill=ggrey] (1.north west)  rectangle (1.south east);
\draw [line width=0.5pt,line cap=round,rounded corners] (2.north west)  rectangle (2.south east);
\node [below] at (2,-0.25) {$2j+1$};
\draw (-5,0)--(-2,0);
\draw (0,0)--(4,0);
\draw (4,0) to  (5,0.5);
\draw (4,0) to (5,-0.5);
\draw [style=dashed] (-2,0)--(0,0);
\node at (2,0) {$\bullet$};
\node [inner sep=0.8pt,outer sep=0.8pt] at (-5,0) (-5) {$\bullet$};
\node [inner sep=0.8pt,outer sep=0.8pt] at (-4,0) (-4) {$\bullet$};
\node [inner sep=0.8pt,outer sep=0.8pt] at (-3,0) (-3) {$\bullet$};
\node [inner sep=0.8pt,outer sep=0.8pt] at (-2,0) (-2) {$\bullet$};
\node [inner sep=0.8pt,outer sep=0.8pt] at (0,0) (-1) {$\bullet$};
\node [inner sep=0.8pt,outer sep=0.8pt] at (1,0) (1) {$\bullet$};
\node [inner sep=0.8pt,outer sep=0.8pt] at (2,0) (2) {$\bullet$};
\node [inner sep=0.8pt,outer sep=0.8pt] at (3,0) (3) {$\bullet$};
\node [inner sep=0.8pt,outer sep=0.8pt] at (4,0) (4) {$\bullet$};
\node [inner sep=0.8pt,outer sep=0.8pt] at (5,0.5) (5a) {$\bullet$};
\node [inner sep=0.8pt,outer sep=0.8pt] at (5,-0.5) (5b) {$\bullet$};
\end{tikzpicture}\quad\text{with $n$ odd and $3\leq 2j+1\leq n-2$}
\end{align*}
In this case $\sigma$ has order $2$, and arguments similar to Cases 5 and 6 gives a contradiction (as in Case 6, the case $i>3$ and $i=3$ is treated separately -- in the latter case we reduce to the $\sD_4$ diagram with $\sigma$ of order $2$ listed in Lemma~\ref{lem:lowrank}).
\smallskip

We have now exhausted all possible uncapped opposition diagrams, and the proof is complete.
\end{proof}

\section{Classification of uniclass automorphisms}\label{sec:4}

In this section we prove the main theorem (Theorem~\ref{thm:main1}). By Theorem~\ref{thm:uncapped} no uncapped automorphism is uniclass, and so we may henceforth consider capped automorphisms. The analysis is case-by-case. In each instance we will work with concrete geometric models for the particular type of building, as described in the following subsection.

\subsection{Lie incidence geometries}

Let $\Delta$ be a building of spherical type~$(W,S)$. We shall adopt Bourbaki labelling \cite{Bou:00} for Dynkin diagrams. Associated to $\Delta$ there are various point-line geometries giving ``shadows'' of the building. For our purposes, this means that we consider one type of vertices of a building of type $\mathsf{X}_n$, say type $i$, as the point set of a point-line geometry, and then the line set is determined by the  panels of cotype $i$. We refer to such geometry as one \emph{of type $\mathsf{X}_{n,i}$} and call it a \emph{Lie incidence geometry}. When the diagram is simply laced, then the building in uniquely determined by the diagram and   a (skew) field $\K$, in which case we denote the point-line geometry of type $\mathsf{X}_{n,i}$ as $\mathsf{X}_{n,i}(\K)$. Each vertex of the building has an interpretation in the Lie incidence geometry, usually as a singular subspace, or a symplecton, or another convex subspace. We introduce these notions now. They are based on the fact that Lie incidence geometries are either projectie spaces, polar spaces or parapolar spaces. We provide a brief introduction, but refer the reader to the literature for more background (e.g.\ \cite{Shult}). 

All point-line geometries that we will encounter are \emph{partial linear spaces}, that is, two distinct points are contained at most one common line---and points that are contained in a common line are called \emph{collinear}; a point on a line is sometimes also called \emph{incident with that line}. We will also always assume that each line has at least three points. In a general point-line geometry $\Gamma=(X,\cL)$, where $X$ is the point set, and $\cL$ is the set of lines (which we consider here as a subset of the power set of $X$), one defines a \emph{subspace} as a set of points with the property that it contains all points of each line having at least two points with it in common. It is called \emph{singular} if each pair of points of it is collinear. It is called a \emph{hyperplane} if every line intersects it in at least one point---and then the line is either contained in it, or intersects it in exactly one point. The \emph{incidence graph} is the graph with vertices the points and lines, adjacent when incident. A subspace is called \emph{convex} if all points and lines of every shortest path between two members of the subspace are contained in the subspace. 

For a skew field $\K$, the \emph{projective space $\mathsf{A}_{n,1}(\K)$} is the point-line geometry with point set the $1$-spaces of an $(n+1)$-dimensional vector space over $\K$ (the \emph{underlying vector space}), and a typical line is the set of $1$-spaces contained in a $2$-space.  The family of singular subspaces is in one-to-one correspondence with the vertices of the building. An automorphism of a building of type $\mathsf{A}_n$ is either a \emph{collineation} of the corresponding projective space, that is, a permutation of the point set preserving the line set, or a \emph{duality}, that is, a bijection from the point set to the set of hyperplanes such that three collinear points are mapped onto three hyperplanes with pairwise the same intersection. Collineations and dualities induce a permutation of all subspaces. A duality acting on the set of subspaces as an order $2$ permutation is called a \emph{polarity}. 

A polar space, for our purposes, is just a Lie incidence geometry of type $\mathsf{B}_{n,1}$ or $\mathsf{D}_{n,1}$. There is an axiomatic approach in which the main axiom is the so-called \emph{one-or-all axiom} due to Buekenhout \& Shult \cite{Bue-Shu:74}: 
\begin{itemize}
\item[(BS)] For every point $p$ and every line $L$, either each point on $L$ or exactly one point on $L$ is collinear to $p$. 
\end{itemize}
We also require that no point is collinear to all other points, and, to ensure finite rank, that each nested sequence of singular subspaces is finite. Then there exists a natural number $r$ such that each maximal singular subspace is a projective space of dimension $r-1$. We call $r$ the \emph{rank} of $\Gamma$. We allow rank 1, in which case we just have a geometry without lines (and we assume at least three points).

We will require some special types of polar spaces when describing the finite case and their examples: a \emph{parabolic} polar space is the one related to a parabolic quadric, that is, a nondegenerate quadric of maximal Witt index in even dimensional projective space; a \emph{hyperbolic} polar space is related to a a building of type $\mathsf{D}_n$ and arises from a nondegenerate quadric  of maximal Witt index in odd dimensional projective space; an \emph{elliptic} polar space is the one related to a nondegenerate quadric of submaximal Witt index in odd dimensional projective space; a \emph{symplectic} polar space is related to a symplectic polarity. 

A \emph{parapolar space} is a point-line geometry with connected incidence graph such that (1) each pair of noncollinear points either are collinear with no or exactly one common point, or is contained in a convex subspace isomorphic to a polar space---called a \emph{symplecton}, and (2) each line is contained in a symplecton. We also require that there are at least two symplecta, and hence, the geometry is not a polar space. A pair of noncollinear points collinear to a unique common point is called \emph{special}; a pair of noncollinear points contained in a common symplecton is called \emph{symplectic}. All Lie incidence geometries which are not projective or polar spaces are parapolar spaces. 

A special type of parapolar spaces occurs when we consider the vertices of so-called \emph{polar type} of an irreducible spherical building as points. The polar type corresponds with the simple root not perpendicular with the highest root (unique in case the Dynkin diagram is not of type $\mathsf{A}_n$, see the start of Section~\ref{sec:3}). Such a Lie incidence geometry is often called a \emph{long root subgroup geometry}. We will only need those of type $\mathsf{D_{4,2}}$, $\mathsf{E_{7,1}}$, $\mathsf{E_{8,8}}$ and $\mathsf{F_{4,1}}$. In such parapolar spaces, point pairs are either identical, collinear, symplectic, special or opposite (the latter in the building theoretic sense---such points have distance $6$ in the incidence graph of the point-line geometry).  In such geometries, we have the notion of an \emph{equator} of two opposite points $p,q$, which is the set of points symplectic to both. This is turned into a geometry by letting the lines be defined by the symplecta through $p$ containing a given maximal singular subspace (maximal in both the symplecta and the whole geometry), and it is called the \emph{equator geometry}, denoted by $E(p,q)$. The equator geometry is isomorphic to the long root subgroup geometry of the residue of a point (in the building theoretic sense). It is also a \emph{fully embedded} subgeometry, that is, the point set forms a subspace.  An \emph{isometric} embedding is one in which each pair of points is collinear, symplectic and special in the embedded geometry if, and only if, it is collinear, symplectic and special, respectively, in the ambient geometry.

The terminology of symplecta stems from the theory of \emph{metasymplectic spaces}, that is, the parapolar spaces of types $\mathsf{F_{4,1}}$ and $\mathsf{F_{4,4}}$, which were introduced and investigated avant-la-lettre by Freudenthal. 

Let $\Gamma=(X,\cL)$ be a Lie incidence geometry and let $U$ be a nonmaximal singular subspace.  Then $U$ corresponds to a certain flag of the corresponding spherical building and we have a building theoretic notion of \emph{residue at $U$}. This is usually a reducible building. However,  in the geometry $\Gamma$ we distinguish the components of that residue by defining the \emph{lower residue at $U$}, denoted $\LRes_\Gamma(U)$, as the projective space defined by $U$ itself, whereas the \emph{upper residue at $U$}, denoted $\URes_\Gamma(U)$, as the point-line geometry with point set the set of singular subspaces of dimension $\dim U+1$ containing $U$, where a typical line is formed by those singular subspaces containing $U$ that are contained in a given singular subspace of dimension $\dim U+2$ containing $U$. It is again a Lie incidence geometry (possibly corresponding to a reducible spherical building). 

We now continue the proof of Theorem~\ref{thm:main1}. We consider each type of irreducible spherical building.

\subsection{Buildings of type $\mathsf{I}_2(d)$ (generalised polygons)}
Let $\Gamma$ be a generalised $d$-gon. Restricting to the thick case (c.f. Section~\ref{sec:nonthick}), we assume that each vertex of the incidence graph has valency at least $3$. In generalised polygons it is customary to call chambers (point-line) \emph{flags}. 

Recall that an \emph{ovoid} or \emph{spread} (also called \emph{distance-$d/2$ ovoid} or \emph{distance-$d/2$ spread}, respectively, in \cite{hvm}) of a generalised $d$-gon, with $d$ even, is a set of mutually opposite points or lines, respectively, such that every point and line is at distance at most $d/2$  from at least one member of the ovoid or spread, respectively.  An \emph{ovoid-spread pairing} is a set of flags such that the points of the flags form an ovoid, and the lines of the flags form a spread. A \textit{duality} is a non-type preserving automorphism of $\Gamma$, and a \emph{polarity} is a duality of order~$2$.

\begin{thm}\label{rank2}
Let $\theta$ be a nontrivial automorphism of a generalised $d$-gon $\Gamma$. Then~$\theta$ is \uniclass\ if and only if it is either anisotropic, or $d$ is even and either
\begin{compactenum}[$(1)$] \item $\theta$ is a collineation that elementwise fixes an ovoid or spread, or \item $\theta$ is a polarity (and its fixed point structure is an ovoid-spread pairing).  \end{compactenum}
\end{thm}

\begin{proof}
If $\theta$ is anisotropic then it is clearly uniclass. Suppose that $d$ is even and (1) or (2) hold. These fixed element structures have the property that the for each flag, the convex closure of the flag and its image contains either a fixed flag (if $\theta$ is a polarity), or a flag that contains an member of the fixed ovoid or spread (if $\theta$ is a collineation). It follows from \cref{ext} that $\theta$ is uniclass.

Suppose now that $\theta$ is \uniclass. We divide the proof according to the parity of $d$ and the order of companion diagram automorphism~$\sigma$. We assume that points are the vertices of type 1 (and denote them with lower case $p,q,\ldots$) and lines those of type $2$ (denoted with upper case letters $L,M,\ldots$).

\emph{Case 1: $d$ odd and $\sigma=1$.} Then either $\theta$ is the identity, or it maps a chamber to an opposite (since there are no domestic collineations by Theorem 2.6 of \cite{PTV}). Assuming the latter, we have $\disp(\theta)=\Cl(w_0)$. Let $\{p,L\}$ be a chamber mapped to an opposite. Then, by \cref{lem:determineimages}, the first half of the unique shortest path between $p$ and $p^\theta$ is symmetrically mapped onto the second half. Hence the element in the middle is fixed. The same holds for each point on $L$ distinct from the projection onto $L$ of $L^\theta$. Consequently we obtain two fixed elements of the same type (clearly not equal), hence a fixed nontrivial path. Thus $1\in\disp(\theta)=\Cl(w_0)$, a contradiction (by parity of lengths of $1$ and $w_0$).

\emph{Case 2: $d$ be odd and $\sigma\neq 1$.} Then either $\theta$ is anisotropic, or domestic. If it is domestic, then by Theorem 2.6 of \cite{PTV}, it is exceptional domestic and hence not \uniclass\ by \cref{thm:uncapped}.

\emph{Case 3: $d$ is even and $\sigma\neq 1$.} Since there are no domestic dualities in this case by \cite[Theorem 2.7]{PTV}, $\theta$ maps at least one flag to an opposite, hence the distance between a flag and its image must be $\bC=\Cl^{\sigma}(w_0)$, which contains all rotations (and hence also the identity). We claim that $\theta$ is a polarity. To see this, let $\{p,L\}$ be an arbitrary flag. Without loss of generality we may assume that a shortest path between $\{p,L\}$ and $\{p^\theta,L^\theta\}$ is $(p,L,x_1,M_1,\ldots,x_k,M_k,L^\theta,p^\theta)$, for some $k$. It follows from Lemma~\ref{lem:determineimages} that $x_1^\theta=M_k$, and continuing like this we obtain a fixed flag in the middle. If $p$ is not opposite $p^{\theta}$ we can extend the above path on either end (appending a line $M\neq L$ through $p$ to the beginning of the path, and the image $M^{\theta}$ to the end). Thus we may assume that $\{p,L\}$ is opposite $\{p^\theta,L^\theta\}$. But then there is a second shortest path, reversing the roles of $p$ and $L$, and we obtain a second fixed flag in the middle opposite to the fixed flag we found earlier. Since these flags are fixed, this implies that the two shortest paths between them are interchanged. Hence $\{p^\theta,L^\theta\}$ is mapped back to $\{p,L\}$, and so $\theta$ is a polarity.

\emph{Case 4: $d$ even and $\sigma=1$.} Assume that $\theta$ is neither the identity nor anisotropic. Then it is domestic. Since $\theta$ is capped (by Theorem~\ref{thm:uncapped}) we may assume, without loss of generality, that $\theta$ is line-domestic. Then \cite[Theorems~2.7 and~2.8]{PTV} imply that, taking into account that $\theta$ cannot fix a chamber (for otherwise the uniclass property forces $\theta$ to be the identity), the fix structure of $\theta$ is a (distance-$d/2$) ovoid. 
\end{proof}

\subsection{Buildings of type $\sA_n$ (projective spaces)}

A \emph{symplectic polarity} of a projective space is a polarity such that every point is contained in its image. Symplectic polarities are always related to a nondegenerate alternating form in the underlying vector space, and hence only exist for projective spaces of odd rank over commutative fields (see \cite{TTVM3}). A \emph{line spread} of a projective space is a partition of the point set into lines.  A line spread is a \emph{composition spread} if it induces a line spread in every subspace spanned by members of the spread.

\begin{thm}\label{Anuni}
A nontrivial automorphism $\theta$ of a projective space~$\Sigma$ is uniclass if and only if it is either an anisotropic duality, or 
\begin{compactenum}[$(1)$]
\item a symplectic polarity, or 
\item it fixes a line spread elementwise (which is automatically a composition spread).
\end{compactenum}
\end{thm}

\begin{proof}
Let $\theta$ be a uniclass automorphism of a projective space~$\Sigma$. Suppose first that the companion diagram automorphism $\sigma$ of $\theta$ has order $2$, that is, $\theta$ is a duality. Since $\{w_0\}$ is a $\sigma$-conjugacy class we either have $\disp(\theta)=\{w_0\}$ (in which case $\theta$ is anisotropic) or $w_0\notin\disp(\theta)$ (in which case $\theta$ is domestic). In the latter case, by \cite[Theorems~3.5,~3.10 and~3.11]{PVM:19a} $\theta$ is either a symplectic polarity, or $\theta$ is strongly exceptional domestic, and the latter case is eliminated by \cref{thm:uncapped}.

Now suppose that $\theta$ is a nontrivial collineation. Suppose that $\theta$ fixes a point~$x$. Since the uniclass property is residual (see Proposition~\ref{prop:residual}) we may assume, by induction, that the restriction of $\theta$ to the residue of $x$ is either the identity, or fixes a line spread. In the former case $\theta$ fixes a chamber of $\Sigma$, and hence is the identity (by the uniclass assumption). In the latter case there is a fixed plane $\alpha\ni x$. Each chamber of $\alpha$ containing $x$ is mapped onto an $s_2$-adjacent chamber. Not all lines not through $x$ of the plane $\alpha$ are fixed, and so, since no line through $x$ is fixed, one can map a chamber of $\alpha$ to an $s_1s_2$-adjacent chamber, a contradiction. Hence there are no fixed points. 

We claim that $\theta$ maps no line to an intersecting one. Let $\Sigma$ have dimension $d$. Suppose that $\theta$ maps a line $L$ to a line $L^\theta$ intersecting $L$ at $x_2$. By the previous paragraph, $x_2$ is not fixed. So we may find a point $x_1\in L\setminus\{x_1\}$ with $x_1^\theta=x_2$. Define inductively $x_{i+1}=x_i^\theta$. Then $x_1,\ldots,x_{d+1}$ generate $\Sigma$. The chain of nested subspaces $x_1,\<x_1,x_2\>,\<x_1,x_2,x_3\>,\ldots,\<x_1,\ldots,x_d\>$ is a chamber of $\Sigma$, and clearly the distance between this chamber and its image is a $(d+1)$-cycle, say $(1~2~\cdots~d+1)$. However, if we replace $x_1$ with any other point of $L\setminus\{x_2\}$, then the Weyl distance between the corresponding chamber and its image is the $d$-cycle $(1~3~4~\cdots~d+1)$, leaving $2$ fixed, a contradiction. Hence we have the property that there are no fixed points, and the line $\<x,x^\theta\>$ is stabilised for all points $x$. The set of all lines $\<x,x^\theta\>$ is clearly a composition line spread. This completes the proof of the `only if' of the theorem.

We now show that if $\theta$ is a symplectic polarity then $\theta$ is uniclass.  Since $\sigma$ has order $2$ and $\theta$ is an involution, Proposition~\ref{prop:basic} (parts (1) and (3)) imply that $\disp(\theta)$ is a union of $\sigma$-conjugacy classes. Since $\sigma$ is also the opposition diagram automorphism, for any $w\in W$ we have $\Cl^{\sigma}(w_0w)=w_0\Cl(w)$. Thus the $\sigma$-involution classes are precisely the sets $w_0\Cl(w)$ with either $w$ an involution or $w=1$. Since $n$ is odd there are $(n+1)/2$ classes of involutions in $W$, with representatives $v_i=s_1s_3s_5\cdots s_{2i-1}$ for $1\leq i\leq (n+1)/2$ (this is clear as $W$ is the symmetric group, and conjugacy classes are given by cycle type). If $i<(n+1)/2$ then $\Cl^{\sigma}(w_0v_i)$ contains an element of length strictly exceeding $\ell(w_0)-(n+1)/2$, contradicting the symplectic polarity opposition diagram ${^1}\sA_{n,(n-1)/2}^2$. Hence $\disp(\theta)=\Cl^{\sigma}(w_0v_m)$ as required.

Finally assume that $\theta$  fixes a line spread elementwise (this line spread is then automatically a composition spread by \cite[Proposition~3.3]{PVMclass}). We will show that the displacement of each chamber is contained in the class defined by a fixed point free involution on the set of type 1 vertices of the Coxeter complex of type $\mathsf{A}_n$. We argue by induction on $n$ (which is odd). If $n=1$, then this is trivial.  Now suppose $n\geq 3$. Let $C$ be an arbitrary chamber. Let $p\in C$ be the point (type 1 vertex) of $C$. We claim that $p$ is not fixed. Indeed, if $p$ were fixed, then each line $M$ through $p$ distinct from the unique spread line would be pointwise fixed, as otherwise the spread lines through the points of $M$, being contained in the plane $\<M,M^\theta\>$, would mutually intersect, a contradiction. Hence $\theta$ is the identity, a contradiction. 

Now the line $L$ generated by $p$ and $p^\theta$ is contained in $\mathsf{conv}\{C,C^\theta\}$, and so are the chambers $D$ and $D^\theta$ obtained by projecting $C$ and $C^\theta$ onto (the residue of) $L$. By the induction hypothesis, the displacement of $C$ in $\Res(L)$ is in the class given by a fixed point free involution of the planes through $L$ in the appropriate Coxeter complex. Extending this complex with $L$ and the points $p$ and $p^\theta$, the assertion now follows from \cref{ext}.
\end{proof}

\subsection{Buildings of types $\sB_n$ and $\sD_n$ (polar spaces)}

In this section we consider collineations of polar spaces. Note that this includes the case of automorphisms of type $\sD_n$ buildings interchanging types $n-1$ and $n$. The case of trialities of $\sD_4$ is considered in the next section.

We begin with some preliminaries. Let $\Gamma$ be a polar space of rank $n\geq 2$ and let $U$ be a singular subspace of dimension $i\leq n-2$. In the theory of polar spaces it is customary to denote by $\Res(U)$ the polar space obtained from $\Gamma$ and $U$ by taking as point set the set of singular subspaces of dimension $i+1$. The lines are then determined by the singular subspaces of dimension $i+2$ (if any) containing $U$. 

Let $\Gamma$ be a polar space. An \emph{ovoid} is a set of points intersecting every maximal singular subspace in exactly one point. A subspace of $\Gamma$ is called \emph{ideal} if it induces an ovoid in the residue of each of its submaximal subspaces. A subspace of \emph{corank $i$} has the property that each singular subspace of dimension $i$ intersects the subspace nontrivially, and there exists a singular subspace of dimension $i-1$ disjoint from it. 
\begin{lemma}
Each ideal subspace $S$ of corank $i$ of a polar space $\Gamma$ of rank $n$ is itself a polar space of rank $n-i$, $0\leq i\leq n-1$.
\end{lemma}
\begin{proof}
We show this by induction on $n-i$. If $n-i=1$, then $S$ is an ovoid and so the assertion follows. 

Now suppose $n-i\geq 2$. Since $S$ is a subspace, each line is thick. Also, if some point $x\in S$ were collinear to all points of $S$, then $S$ induces in the residue of a submaximal singular subspace $U$ not containing $x$ exactly one point, hence certainly not an ovoid, a contradiction.  The one-or-all axiom holds because it holds in $\Gamma$. It remains to show that the rank is $n-i$. Obviously $S$ induces an ideal subspace of corank $i$ in each point residue at a point of $S$. By induction, $S$ induces in each such residue a polar space of rank $n-i-1$. Thus $S$ is a polar space of rank $n-i$.
\end{proof}

\begin{thm}\label{Bn1uni}
Let $\theta$ be a collineation of a polar space. Then $\theta$ is uniclass if and only if $\theta$ is either anisotropic, or
\begin{compactenum}[$(1)$]
\item the fixed points of $\theta$ form an ideal subspace, or
\item it fixes no point and it fixes elementwise a line spread.
\end{compactenum}
\end{thm}

\begin{proof}
Let $\theta$ be a uniclass collineation of a polar space~$\Gamma$. Assume first that there is a fixed point. If $\theta$ is the identity then it fixes the ideal subspace $\Gamma$, and so assume that $\theta$ is not the identity. Since there is a fixed point, $\theta$ is not anisotropic, and hence it is domestic. Moreover, since it fixes a point it is not point-domestic, as otherwise it would fix a chamber by \cite[Proposition~3.16]{PVMclass}. Hence it has opposition diagram $\mathsf{B}_{n;i}^1$ or $\mathsf{D}_{n;i}^1$, $i<n$. This diagram tells us that we can find a flag $F$ of type $\{1,2,\ldots,i\}$ mapped to an opposite. Let $U$ be a maximal singular subspace of $F^\perp\cap(F^\theta)^\perp$. Then $\dim U=r-i-1$. Theorem~1 of \cite{PVMclass} tells us that there is a subspace $S$ of corank $i$ pointwise fixed and that $U$ belongs to it, hence is pointwise fixed. Hence within the upper residue of some fixed subspace of dimension $r-i-1$, we see the longest word as displacement. If the rank of the pointwise fixed subspace now were strictly larger than $r-i$, then we find a subspace of dimension $r-i$ pointwise fixed, which means in the upper residue of a subspace of dimension $r-i-1$, we see a displacement different from the longest word, a contradiction. Similarly, $\theta$ admits no other fixed points than those of $S$.   Now consider a pointwise fixed subspace $W$ of dimension $r-i-2$. For every maximal singular subspace $M$ through $W$, we can find a subspace $M'$ of $M$ complementary to $W$; this has dimension $i$ and hence contains a fixed point $u$. We conclude that the pointwise fixed subspaces of dimension $r-i-1$ through $W$ form an ovoid. Hence the fixed points of $\theta$ form an ideal subspace. 

Now assume that there are no fixed points. Then $\theta$ is not the identity. Suppose it is not anisotropic. Then it is domestic. Since it does not fix any point, it follows from \cite[Lemma 2.1]{PVMclass} that $\theta$ is point-domestic. Then by \cite[Proposition~3.1]{PVMclass} $\theta$ elementwise fixes a line spread.

We now prove the converse. Suppose first that the fixed points of $\theta$ form an ideal subspace. Let $C$ be a chamber of $\theta$. Let the ideal subspace $\cS$ have corank $i$. We first claim that no point is mapped onto a collinear one. Suppose for a contradiction that $p$ is a point mapped onto a collinear one. Projecting $p$ onto a pointwise fixed singular subspace of dimension $n-i-1$ and looking in the residue of a hyperplane of that projection, we may assume that $i=n-1$, so $\cS$ is an ovoid. Let $M$ be a maximal singular subspace containing $p$ and $p^\theta$. Then $M$ contains a unique point $s\in\cS$. Select a hyperplane $H$ of $M$ containing $p$ but neither $p^\theta$ nor $s$. Then any maximal singular subspace $M'\neq M$ containing $H$ contains a point $s'\in\cS$ not collinear to $p^\theta$, a contradiction since $s'\perp p$. The claim is proved.  Let $\bC$ be the class of displacements corresponding to the automorphism of the Coxeter complex of type $\mathsf{B}_n$ given by $i$ sign transpositions and the identity everywhere else. 

Now denote by $M$ the maximal singular subspace belonging to $C$. Then $M$ contains a pointwise fixed subspace $S$ of dimension $r-i-1$. Let $D$ be the (building-theoretic) projection of $C$ onto (the vertex) $S$. By our claim above, $S=M\cap M^\theta$, so $M$ belongs to $\mathsf{conv}\{C,C^\theta\}$ and hence $\{D,D^\theta\}\subseteq\mathsf{conv}\{C,C^\theta\}$. Since, by the claim above, $\theta$ acts anisotropically on $\URes_\Gamma(S)$, which is a polar space of rank $i$, and trivial on $\LRes_\Gamma(S)$, the assertion follows from \cref{ext}. 

Now assume that $\theta$ elementwise fixes a line spread $\cS$ of $\Gamma$. Let $C$ be an arbitrary chamber of $\Gamma$. We show that the displacement of $C$ belongs to the class $\bC$ defined by an involutive automorphism of the Coxeter complex of type $\mathsf{B}_n$ mapping each vertex of type $1$ onto a non-opposite distinct vertex. Let $p$ be the point of $C$. Set $L=pp^\theta$. Then $L\in\mathsf{conv}\{C,C^\theta\}$. If $D$ is the (building-theoretic) projection of $C$ onto $L$, then we infer $\{D,D^\theta\}\subseteq\mathsf{conv}\{C,C^\theta\}$. Since $\theta$ induces a collineation in the upper residue of $L$ elementwise fixing a line spread, and since it acts fixed point freely on $L$, the displacement of $D$ belongs to $\bC$. Now \cref{ext} implies that $\delta(C,C^\theta)\in\bC$. 
\end{proof}

\subsection{Trialities of $\sD_4$} 
A \emph{triality} is a type rotating automorphism of a building of type $\mathsf{D_4}$. We assume the labelling chosen such that the types are rotated like $1\mapsto 3\mapsto 4\mapsto 1$. We do not assume that a triality necessarily has order $3$.

We will work in the Lie incidence geometry $\sD_{4,2}(\KK)$, in which ``points'' are the type~$2$ vertices of the building $\sD_4(\KK)$. Equivalently, one may regard the points the cosets in $G/P$, where $G=\sD_4(\KK)$ and $P$ is the standard parabolic subgroup of $G$ of type $\{1,3,4\}$. This incidence geometry has diameter $3$. The set of minimal length $(W_{\{1,3,4\}},W_{\{1,3,4\}})$ double coset representatives in $W$ is 
$
\{e,2,2132,2342,2142,21342,213242132\}.
$
The representative $2$ corresponds to ``distance~$1$'' (or collinearity), in the sense that the points $gP$ and $hP$ are at distance $1$ if and only if $g^{-1}h\in Ps_2P$. Similarly, the representatives $2132,2342,2142$ all correspond to one form of distance $2$, the representative $21342$ corresponds to another form of distance $2$ (we call this distance $2'$), and $213242132$ corresponds to distance $3$ (or opposition).

\begin{lemma}\label{trial}
If $\theta$ is a non-domestic \uniclass\ triality of $\sD_4$ then $\theta$ maps no point of $\sD_{4,2}(\KK)$ to distance $0$, $1$, or $2'$ (that is, all points are mapped to either distance $2$ or distance $3$). 
\end{lemma}

\begin{proof}
Let $\sigma$ be the triality diagram automorphism with $s_1^{\sigma}=s_3$, $s_3^{\sigma}=s_4$, and $s_4^{\sigma}=s_1$. Since $\theta$ is not domestic we have $w_0\in\disp(\theta)$, and hence $\disp(\theta)=\Cl^{\sigma}(w_0)$. Direct calculation gives that 
\begin{align*}
&\Cl^{\sigma}(w_0)=\{121324, 321421, 121321, 123242 , 142132, 214213, 232421, 232423, 121421,\\
&\quad  12324213, 13421324, 13214213 , 2132421324 , 2132142132 ,   1 2 1 3 2 1 4 2 1 3 2 4 , 1 2 1 3 2 4 2 1 3 2 \}
\end{align*}
(one may directly verify that this set is closed under $\sigma$-conjugation by~$S$). In particular the elements of $\Cl^{\sigma}(w_0)$ all lie in the double cosets with representatives $2132,2342,2142$ or $213242132$, hence the result.
\end{proof}

\begin{thm}\label{nonondom}
There are no uniclass trialities.
\end{thm}

\begin{proof}
Let $\theta$ be a triality. If $\theta$ is domestic then by \cite{Mal:14} $\theta$ is a triality of type $\mathsf{I}_{\mathsf{id}}$ (that is, $\theta$ is a triality of order $3$ and its fixed point structure is a split Cayley hexagon). In particular, $\theta$ fixes a chamber of $\Delta$, and so $1\in\disp(\theta)$. The opposition diagram of $\theta$ is ${^3}\sD_{4;1}^2$ (see \cite{PVMclass}) and hence by Proposition~\ref{prop:attaindisplacement} we have $s_2w_0\in\disp(\theta)$. However $1$ and $s_2w_0$ are not $\sigma$-conjugate (by parity of length). 

Suppose now that $\theta$ is not domestic, and so $w_0\in \disp(\theta)$. If $\theta$ is uniclass then $\disp(\theta)=\Cl^{\sigma}(w_0)$. We argue in the polar space $\mathsf{D_{4,1}}$. We first claim that there is some non-absolute point (that is, there exists a point contained in its image). Indeed, suppose for a contradiction that all points are absolute. Let $p$ be an arbitrary point and choose $q\in p^\theta$, $q\neq p$. Then $p^\theta\cap q^\theta$ is a line $L^\theta$ that contains $q$, with $L=\<p,q\>$. Hence, 
by \cref{trial}, this implies $L=L^\theta$. Hence $p^{\theta^2}$ contains $L^\theta$. Since this is true for each line $L$ in $p^\theta$ through $p$ (by the arbitrariness of $q$), this implies $p^\theta=p^{\theta^2}$, a contradiction and the claim follows.  

Hence we may assume that a point $p$ is not absolute. Set $p^\theta=W_+$ and let $W_-=\<p,p^\perp\cap W_+\>$. Then $W_-^\theta$ is a point $q$ that is contained in $W_+$. Suppose for a contradiction that $q\in W_-$ (then $q\perp p$). The lines through $p$ inside $W_-$ are mapped onto lines through $q$ inside $W_+$. Set $\pi=p^\perp\cap W_+$. Then $\<p,q\>^\theta$ is not contained in $\pi$, as this would imply that $\<p,q\>$ and its image are collinear in $\mathsf{D_{4,2}}$, contradicting \cref{trial}. Now  pick a line $K^\theta$ through $q$ in $W_+$ not contained in $\pi$ and distinct from $\<p,q\>^\theta$. Then $K$ intersects $\pi$ in some point $r\in \pi\setminus\{q\}$. Since $r\perp K^\theta$, $r\notin K^\theta$, but $p^\perp\cap K^\theta=\{r\}$, we deduce that $K$ and $K^\theta$ are at distance $2'$ from each other in $\mathsf{D_{4,2}}$, again contradicting \cref{trial}. We conclude that $q\notin W_-$. Now let $M$ be an arbitrary line through $p$ in $W_-$. As the distance from $M$ to $M^\theta$ in $\mathsf{D_{4,2}}$ is not $2'$, we see as before that $M\cap M^\theta$ is nonempty (and contained in $\pi$). Consequently, if $L$ is a line in $\pi$, and we denote the planes $\<p,L\>$ and $\<q,L\>$ by $\alpha$ and $\beta$, respectively, then the lines in $\alpha$ through $p$ are mapped onto the lines of $\beta$ through~$q$.  

Let $U_+$ be the solid of type $3$ containing $\alpha$ and let $U_-$ be the solid of type $4$ containing $\beta$.  Note that, as $W_-$ and $U_-$ share the line $L$, they share a plane. Also $U_+^\theta=U_-$.  Pick an arbitrary point $r$ on $L$ and set $L_p=\<p,r\>$ and $L_q=\<q,r\>$. Then we have the following gallery:
\begin{multline*}\{p,L_p,U_+,W_-\}\stackrel{1}{\sim}\{r,L_p,U_+,W_-\}\stackrel{2}{\sim}\{r,L,U_+,W_-\}\stackrel{3}{\sim} \{r,L,W_+,W_-\}\stackrel{4}{\sim}\\ \{r,L,W_+,U_-\}\stackrel{2}{\sim}\{r,L_q,W_+,U_-\}\stackrel{1}{\sim}\{q,L_q,W_+,U_-\}=\{W_-^\theta,L_p^\theta,p^\theta,U_+^\theta\},\end{multline*}

implying that the displacement $123421$ is attained, however this does not lie in the twisted conjugacy class of $w_0$ (see the proof of Lemma~\ref{trial}), a contradiction.
\end{proof}

\subsection{Buildings of type $\sF_4$}

We first deal with dualities, and then collineations. 

\subsubsection{Dualities}

We have the following basic lemma.

\begin{lemma}\label{lem:uniqueclass}
Let $\sigma$ be the duality diagram automorphism of $\sF_4$. There exists a unique $\sigma$-conjugacy class $\bC$ of $\sigma$-involutions, and $\bC=\Cl^{\sigma}(1)=\Cl^{\sigma}(w_0)$. Moreover, for each element $w\in \bC\backslash\{w_0\}$ there exists $s\in S$ with $\ell(sws^{\sigma})=\ell(w)+2$.
\end{lemma}

\begin{proof}
By Theorem~\ref{thm:downwardsclosure} each $\sigma$-conjugacy class of $\sigma$-involutions contains an element $w_J$ with $sw_J=w_Js^{\sigma}$ for all $s\in J$. The only possibility is $J=\emptyset$, and hence $\Cl^{\sigma}(\id)$ is the only class of $\sigma$-involutions. Since $w_0$ is a $\sigma$-involution it follows that $w_0\in\Cl^{\sigma}(1)$, and the final claim follows using Theorem~\ref{thm:downwardsclosure} (since the class is self dual).
\end{proof}

We can now prove Theorem~\ref{thm:main1} for dualities of $\sF_4$ building. 

\begin{thm}\label{thm:F41}
A duality $\theta$ of a thick $\sF_4$ building is \uniclass\ if and only if it is a polarity, and hence its fixed element structure consists of type $\{1,4\}$ and type $\{2,3\}$ simplices forming a Moufang octagon. 
\end{thm}

\begin{proof}
Suppose that $\theta$ is \uniclass. We claim that $\theta$ is a polarity (that is, has order~$2$), from which the result follows by \cite[Theorem~2.5.2]{hvm}. 

By \cite[Lemma~4.1]{PVM:19a} no duality of a thick $\sF_4$ building is domestic, and hence by Proposition~\ref{prop:full} we have $\disp(\theta)=\Cl^{\sigma}(w_0)$. Let $C$ be a chamber with $\delta(C,C^{\theta})=w_0$, and let $\cA$ be the unique apartment containing $C$ and $C^{\theta}$. For $w\in W$ write $C_w$ for the unique chamber of $\cA$ with $\delta(C,C_w)=w$. 

By Lemma~\ref{lem:determineimages}, if $u\in W$ with $\ell(u^{-1}w_0u^{\sigma})=\ell(w_0)-2\ell(u)$ then $C_u^{\theta}=C_{w_0u^{-\sigma}}$. Taking $
v=s_1s_2s_1s_3s_2s_1s_3s_4s_3s_2s_1s_3$ and noting that $w_0=vv^{-\sigma}$ it follows that $C_v^{\theta}=C_{w_0v^{\sigma}}=C_v$, and similarly $C_{v^{\sigma}}^{\theta}=C_{v^{\sigma}}$. But $\delta(C_v,C_{v^{\sigma}})=v^{-1}v^{\sigma}=w_0$, and so $C_v$ and $C_{v^{\sigma}}$ are opposite chambers in $\cA$. Thus every chamber of $\cA$ lies in the convex hull of $\{C_v,C_{v^{\sigma}}\}$, and it follows that $C_w^{\theta}=C_{w_0w^{\sigma}}$ for all $w\in W$. Thus $\cA$ is stabilised by $\theta$, and $\theta|_{\cA}$ has order~$2$.

Now let $D$ be any chamber of $\Delta$. If $w=\delta(D,D^{\theta})\neq w_0$ then by Lemma~\ref{lem:uniqueclass} there is $s\in S$ with $\ell(sws^{\sigma})=\ell(w)+2$. Let $E$ be any chamber with $D\sim_s E$. Then $\delta(E,E^{\theta})=sws^{\sigma}$ and the chambers $D,D^{\theta}$ lie in $\mathsf{conv}\{E,E^{\theta}\}$. Continuing inductively we see that there exists a chamber $C$ with $\delta(C,C^{\theta})=w_0$ such that $D,D^{\theta},C,C^{\theta}$ lie in a common apartment. The argument of the previous paragraph shows that this apartment is stabilised, and that $\theta$ has order $2$ on this apartment. Hence $\theta$ is a polarity on $\Delta$, and hence $\theta$ fixes a Moufang octagon by \cite[Theorem~2.5.2]{hvm}.

Conversely, suppose that $\theta$ is a polarity. 
Thus each $w\in\disp(\theta)$ is a $\sigma$-involution, and hence $\disp(\theta)$ is a union of $\sigma$-conjugacy classes of $\sigma$-involutions by Proposition~\ref{prop:basic}. However there is a unique such class (by Lemma~\ref{lem:uniqueclass}), and so $\theta$ is \uniclass. 
\end{proof}

\subsubsection{Collineations}

We now turn to type preserving automorphisms (collineations). We begin with some preliminaries. A building of type $\mathsf{F_4}$ is not determined by a field alone, but by a pair $(\K,\AA)$, where $\AA$ is a quadratic alternative division ring over $\K$. It is customary to choose the types so that residues of type $\{1,2\}$ (which are residues of flags of type $\{3,4\}$) correspond to projective planes coordinatised by $\K$, and those of type $\{3,4\}$ correspond to projective planes coordinatised by~$\AA$. In this way, the vertices of type $1$ are centres of the long root elations (c.f. \cite[\S2.1]{PVMexc}). We denote the corresponding building by $\mathsf{F_4}(\K,\AA)$. The split case corresponds to $\AA=\K$, the trivial one-dimensional algebra over $\K$. 

Let $\Gamma=(X,\cL)$ be an embeddable polar space, and let $O\subseteq X$ be an ovoid of $\Gamma$. Then we say that $O$ is \emph{flat} if it arises as the intersection of $X$ with a subspace of some ambient projective space in which $\Gamma$ is embedded. Also, we say that $O$ is \emph{linear} if for any pair of points $x,y$ of $O$, the intersection of the line through $x$ and $y$ in any ambient projective space in which $\Gamma$ is embedded with $X$ is fully contained in $O$. Now, a set of vertices of type $1$ and $4$ of a building of type $\mathsf{F_4}$ forming a Moufang quadrangle, with the property that the fixed vertices of type $1$ or $4$ incident with a fixed vertex $v$ of type $4$ or $1$, respectively, form an ovoid in polar space corresponding to the residue at $v$, which is flat or linear, respectively, will be called an \emph{ideal quadrangular Veronesean}. 

The following is shown in \cite[Main Result]{Lam-Mal:23}.

\begin{prop}\label{newlemma}
A collineation $\theta$ of a thick $\sF_4$ building has opposition diagram $\mathsf{F_{4;2}}$ and fix diagram $\mathsf{F_{4;2}}$ if, and only if, its fix structure is an ideal quadrangular Veronesean. 
 In particular this means that no such collineation exists for $\AA=\K$ or $\AA$ non-associative. 
\end{prop}

Then the following additional properties are shown in \cite{Lam-Mal:23}, where we let $\theta$ act on the parapolar space $\mathsf{F_{4,4}}(\K,\AA)$.

\begin{lemma}\label{symplfixed}
Let $\theta$ be an automorphism of $\mathsf{F_4}(\K,\AA)$ with opposition diagram $\mathsf{F_{4;2}}$ and fix diagram $\mathsf{F_{4;2}}$. 
\begin{compactenum}[$(1)$]
\item If a point $p$ is mapped onto a symplectic one, then the unique symp containing $p$ and $p^\theta$ is stabilised.
\item A point mapped to an opposite by $\theta$ is symplectic with at least two mutually opposite fixed points.
\end{compactenum}
\end{lemma}
\begin{proof}
(1) is \cite[Corollary~5.2.2$(iv)$]{Lam-Mal:23}, and (2) follows from the proof of \cite[Proposition 4.2.1]{Lam-Mal:23}, in particular the third of that proof.
\end{proof}

We will also use the notion of an \emph{extended equator geometry}.
\begin{defn}\label{equator}
Let $p,q$ be two opposite points of $\mathsf{F_{4,4}}(\K,\AA)$. The \emph{equator geometry} $E(p,q)$ is the point-line geometry with point set the points symplectic to $p$ and $q$, and line set the sets of points corresponding to symplecta through a fixed plane through $p$. Then the \emph{extended equator} $\wE(p,q)$ is the union of all equators $E(x,y)$, for $x$ and $y$ opposite points of $E(p,q)$. It gets the structure of a geometry when endowed with the intersections with the symps determined by any two symplectic points in it. It contains the equator geometry $E(x,y)$ or any pair $(x,y)$ of opposite points in it. It is isomorphic to the rank 4 polar space whose point residue is isomorphic to a symp of $\mathsf{F_{4,1}}(\K,\AA)$.   
\end{defn}

It follows from \cite{Lam-Mal:23} that, as soon as an extended equator geometry contains two fixed points, it contains a lot of them. More precisely, the following lemma will be useful as a reduction result. Let $\theta$ be as in \cref{symplfixed}.
\begin{lemma}\label{reduction}
Let $p,q$ be two opposite fixed points of $\theta$. Then the fixed points in $\wE(p,q)$ form an ideal subspace of corank $2$.
\end{lemma}
\begin{proof}
In the proofs of Theorem~5.3.2$(iii)$ and Theorem~5.4.3$(ii)$ in \cite{Lam-Mal:23}, it is shown that the fixed point structure of $\theta$ in $\wE(p,.q)$ is a rank 2 polar subspace, and that $\theta$ induces a plane- and solid-domestic collineation in it. Then Theorem~6.1 of \cite{TTVM} yields the assertion. 
\end{proof}
The next lemma follows immediately from Lemma~2.11.4 of \cite{Lam-Mal:23}.
\begin{lemma}\label{apaF4}
Apartments of $\wE(p,q)$ correspond to apartments of $\Delta$ having at least two opposite points in $\wE(p,q)$. 
\end{lemma}

We can now prove Theorem~\ref{thm:main1} for collineations of $\sF_4$ buildings. 

\begin{thm}\label{thm:F42}
A collineation $\theta$ of a thick $\sF_4$ building is \uniclass\ if, and only if, either it is the identity, or it is anisotropic, or its fix structure is an ideal quadrangular Veronesean.
\end{thm}

\begin{proof}
Suppose first that $\theta$ is \uniclass, and assume that it is neither the identity nor anisotropic. Then it is domestic and it does not fix any chamber. It follows from  \cite{Lam-Mal:23} that the fix structure consists of vertices of type $1$ and $4$ forming a Moufang quadrangle, with the property that the fixed vertices of type $1$ or $4$ incident with a fixed vertex $v$ of type $4$ of $1$, respectively, form an ovoid in polar space corresponding to the residue of $v$, which is flat or linear, respectively.

Now the converse. Let $\bC$ be the class of displacements determined by the opposition diagram $\mathsf{F_{4;2}}$ (hence determined by an involution of the Coxeter  system of type $\mathsf{F_4}$ fixing exactly four type 1 and four type 4 vertices (in a quadrangle)). We argue in $\mathsf{F_{4,4}}(\K,\AA)$. By \cref{newlemma}, $\theta$ is domestic with opposition and fix diagram $\mathsf{F_{4;2}}$. Let $C$ be any chamber, and suppose first that the vertex $x$ of type $4$ of $C$ is mapped onto an opposite.  Then $x$ is a point of $\mathsf{F_{4,4}}(\K,\AA)$. \cref{symplfixed} yields two mutually opposite fixed points $p,q$ symplectic to $x$, and hence also to $x^\theta$. It follows from \cref{reduction} and \cref{apaF4} that the displacement of $C$ is contained in the class of Weyl distances induced by the map $\kappa$ on a Coxeter complex of type $\mathsf{B_{4}}$ corresponding to a uniclass coliineation fixing an ideal subspace of corank 2.  Clearly $\kappa$ induces an involution of the Coxeter complex of type $\mathsf{F_4}$ fixing exactly four points and four symps. Hence the displacement of $C$ is in  $\bC$.

Next, suppose that the vertex $x$ of type $4$ of $C$ is mapped onto a symplectic vertex. By \cref{symplfixed} the symplection $\xi$ determined by $x$ and $x^\theta$ is fixed. By assumption, the fix structure of $\theta$ in $\xi$ is an ovoid, hence $\theta$ induces a \uniclass\ collineation in $\xi$, and clearly that class uniquely determines $\bC$. Projecting $C$ and $C^\theta$ onto $\xi$, the assertion follows from \cref{ext}. 

So we may assume $x$ is fixed. Then we project $C$ and $C^\theta$ onto $x$ and the same argument as in the previous paragraph---now with the symp of $\mathsf{F_{4,1}}(\K,\AA)$ corresponding to $x$--- shows the assertion.
\end{proof}

\subsection{Buildings of type $\mathsf{E_6}$}

We begin with some preliminaries. 

A \textit{symplectic polarity} of a building of type $\sE_6$ is a polarity whose fixed point structure is a building of type $\sF_4$ containing residues isomorphic to symplectic polar spaces (such an $\sF_4$ building is a \textit{standard split metasymplectic space}).

Let $\Gamma=\mathsf{E_{6,1}}(\K)$, and let $\Gamma^*:=\mathsf{E_{6,6}}(\K)$ be its \emph{dual}. We can view an apartment $\cA$ of $\Gamma$ as the unique generalised quadrangle $\GQ(2,4)$ with $3$ points per line and $5$ lines per point, as follows from \cite[p.~202]{Cox:63} (see also \cite[Section~10.10]{Bro-Mal:22}). Two points are non-collinear in $\GQ(2,4)$ if and only if the corresponding vertices of $\cA$ form an edge in $\cA$ (that is, are collinear in $\Gamma$). 

\begin{lemma}\label{gq24}
Let $\kappa$ be a collination of $\GQ(2,4)$. Then the set of domestic points (that is, points that are not mapped to opposite points) forms a subspace of $\GQ(2,4)$. In particular, if all points collinear to some point $x$ are domestic, and one additional point  not collinear to $x$ is domestic, then all points are domestic and $\kappa$ is point-domestic.  In the latter case $\kappa$ is either an axial elation (in which case it has exactly $3$ fixed points), or fixed a spread linewise (in which case is has no fixed points at all).
\end{lemma}

\begin{proof}
This follows easily from the analysis of domesticity in generalised quadrangles in~\cite{TTVM2}. 
\end{proof}

\begin{cor}\label{cor:gq24}
The automorphism of a Coxeter complex of type $\mathsf{E_6}$ induced by the longest word corresponds to an axial collineation of $\GQ(2,4)$. 
\end{cor}

\begin{proof}
Since a chamber must be mapped onto an opposite chamber, it is easy to check that each point of the symps of these chambers, except for the intersection point, is mapped onto a noncollinear point. Hence by \cref{gq24}, the induced collineation in $\GQ(2,4)$ is point-domestic. Since it fixes at least one point (the intersection point of the symps of the opposite chambers), the assertion again follows from \cref{gq24}.  
\end{proof}

\begin{lemma}\label{residualE6}
If an automorphism $\kappa$ of the Coxeter complex of type $\mathsf{E_6}$ fixes a type $6$ vertex $\xi$ and acts on the type $1$ vertices incident with $\xi$ as an involuton with exactly two fixed points, then $\kappa$ is induced by the longest word. 
\end{lemma}

\begin{proof}
Translated to $\GQ(2,4)$, we have a collineation fixing all points of a certain line $L$, and interchanging the points on the lines through one of the points $p$ of $L$. Composed with the axial collineation with axis $L$ we obtain a central collineation with centre $p$, which must be the identity, as $\GQ(2,4)$ does not admit non-trivial central collineations by 8.1.2 of \cite{Pay-Tha:09}. The lemma follows.
\end{proof}

\begin{defn}
Let $V$ be a set of points of $\Gamma$, no two of which are collinear, and not contained in one symp. Then $V$ is called an \emph{ideal Veronesean} if the intersection of $V$ with the symp $\xi$ determined by any pair of points of $V$, is an ovoid of $\xi$. Such a symp $\xi$ is called a \emph{host space} of $V$.
\end{defn}

The following properties of ideal Veroneseans are proved in Lemmas~4.8, 4.10 and~4.11, respectively, in \cite{NPVV}.

\begin{prop}\label{kanE6}
Let $V$ be an ideal Veronesean in $\Gamma$. Then 
\begin{compactenum}[$(1)$] 
\item every pair of distinct host spaces intersects in a unique point;
\item the set of host spaces is an ideal Veronesean in the dual $\Gamma^*$ of $\Gamma$.
\item Every point of $\Gamma$ belongs to at least one host space. 
\end{compactenum} 
\end{prop}

We now come to the main theorem for $\sE_6$ buildings. 

\begin{thm}\label{E6uni}
An non-trivial automorphism $\theta$ of $\mathsf{E_{6}}(\K)$, for some field $\K$, is uniclass if and only if it is an anisotropic duality, a symplectic polarity, or a collineation fixing an ideal Veronesean pointwise.
\end{thm}

\begin{proof}
Suppose $\theta$ is \uniclass. If $\theta$ is a duality,  then either it is anisotropic, or it is domestic and hence a symplectic polarity (since we may assume by \cref{thm:uncapped} that $\theta$ is capped). Now suppose that $\theta$ is a collineation. If $\theta$ is domestic, then it fixes a chamber, and so $\theta$ is the identity. If it is not the identity, then the displacements are in the class of the longest word in the Coxeter group, which, by \cref{cor:gq24}, is realised by an axial collineation of $\GQ(2,4)$. Since such a collineation is point-domestic in the quadrangle, it does not map vertices of type 1 of $\mathsf{E_6}(\K)$ to vertices at distance 1. Now it follows from \cite[Theorem~4.1]{NPVV} that the fix point structure is an ideal Veronesean. 

Now we show the converse. If $\theta$ is a polarity, then all displacements are $\sigma$-involutions. We know from \cref{prop:basic}(3) that if an element of a $\sigma$-class is a displacement, then the full class is contained in the displacements. There are $5$ classes of $\sigma$-involutions. Four of these classes have elements of length exceeding $36-12=24$, which contradicts the opposition diagram. Hence only one class remains, and this is the $\sigma$-class of the identity. 

Finally, suppose $\theta$ is a collineation the fix structure of which is an ideal Veronesean. Let $C$ be an arbitrary chamber and let $x$ be its vertex of type 1. If $x$ is fixed, then $\theta$ induces in the residue of $x$, viewed as polar space of type $\mathsf{D_{5,1}}$, a collineation fixing an ovoid.  \cref{Bn1uni}(1) and (the dual of) \cref{residualE6} imply that $\delta(C,C^\theta)$ belongs to the conjugacy class of the longest word. 

So we may assume that $x$ is not fixed. We argue in $\Gamma=\mathsf{E_{6,1}}(\K)$. Then $x$ and $x^\theta$ are contained in a unique  symplecton $\xi$, which is fixed by \cref{kanE6}$(3)$. So $\xi$ is contained in $\mathsf{conv}\{C,C^\theta\}$ and hence, if $D$ is the projection of $C$ onto $\xi$, then $\{D,D^\theta\}\subseteq\mathsf{conv}\{C,C^\theta\}$. As in the previous paragraph, $\theta$  induces in $\xi$ a collineation fixing an ovoid. Now the assertion follows from \cref{Bn1uni}(1) and \cref{residualE6}, combined with \cref{ext}. 
\end{proof}

\subsection{Buildings of type $\mathsf{E_7}$}

We begin with some preliminaries on two types of fixed structures. 

\subsubsection{Fixing a metasymplectic space}
By \cite[Theorem~7.23]{NPVV} we have:
\begin{prop}\label{nofixedpoints}
Let $\theta$ be an automorphism of the building $\mathsf{E_7}(\K)$. Then the following are equivalent.
\begin{compactenum}[$(1)$]
\item $\theta$ does not fix any chamber and has opposition diagram $\mathsf{E_{7;3}}$.
\item The fixed point structure of $\theta$ induced in $\mathsf{E_{7,1}}(\K)$ is a fully isometrically embedded metasymplectic space $\mathsf{F_{4,1}}(\K,\LL)$, for some quadratic extension $\LL$ of $\K$. 
\item The collineation induced in $\mathsf{E_{7,7}}(\K)$ has no fixed points and does not map any point to a symplectic one, but it maps at least one point to a collinear one.  
\end{compactenum}
\end{prop}

The following results are shown in \cite{NPVV}, where the action of $\theta$ on $\mathsf{E_{7,7}}(\K)$ is examined. \cref{nofixedpoints}$(3)$ tells us that the displacement of the points is highly restricted. The displacement of the symps is also highly restricted.

\begin{lemma}\label{notspecial}
Let $\theta$ be an automorphism of $\mathsf{E_7}(\K)$ with opposition diagram $\mathsf{E_{7;3}}$ and fix diagram $\mathsf{E_{7;4}}$.
\begin{compactenum}[$(1)$]
\item Let $\xi$ be an arbitrary symp of $\mathsf{E_{7,7}}(\K)$. Then either $\xi$ is fixed, or $\xi\cap\xi^\theta$ is a line, or $\xi$ is opposite $\xi^\theta$. 
\item For each symp $\xi$, the intersection $\xi\cap\xi^\theta$ is preserved by $\theta$.
\item If a point $x$ is mapped onto a collinear point $x^\theta$, then the line $xx^\theta$ is fixed. 
\item If a line $L$ is fixed, then  $\theta$ induces in the upper residue of $L$ a collineation pointwise fixing an ideal subspace of corank $2$ (with fix diagram $\mathsf{D_{5;3}^1}$) in the corresponding polar space of type $\mathsf{D_5}$.
\end{compactenum} 
\end{lemma}
\begin{proof}
\begin{compactenum}[$(1)$]
\item By Lemma~7.7 of \cite{NPVV}, no symp is mapped onto a disjoint but non-opposite one, and by Lemma~7.17 of \cite{NPVV}, no symp is mapped onto an ``adjacent'' one. Hence the assertion. 
\item This is trivial if $\xi$ is fixed or is mapped onto an opposite. If $\xi$ and $x^\theta$ share a line, then this follows from Lemma~7.9 of \cite{NPVV}. 
\item This is precisely Lemma~7.8 in \cite{NPVV}.
\item In the penultimate paragraph of the proof of Proposition~7.18 of \cite{NPVV} it is shown that $\theta$ induces a collineation in the said polar space mapping no point to a collinear one, and fixing only points, lines ad planes. By Proposition~3.8$(v)$ of~\cite{NPVV}, the assertion follows. \qedhere
\end{compactenum} 
\end{proof}

\subsubsection{Fixing a dual polar space}
The following results are contained in \cite{NPVV}, 
where again the action of $\theta$ on $\mathsf{E_{7,7}}(\K)$ is examined.

\begin{lemma}\label{evendist}
Let $\theta$ be an automorphism of $\mathsf{E_7}(\K)$ with opposition diagram $\mathsf{E_{7;4}}$ and fix diagram $\mathsf{E_{7;3}}$. 
\begin{compactenum}[$(1)$]
\item A point is never mapped to a collinear one, nor to an opposite one. Moreover, the symp determined by a point $x$ mapped onto a symplectic one, and its image $x^\theta$, is stabilised.
\item The collineation induced by $\theta$ in the residue of any fixed point, pointwise fixes an ideal Veronesean. 
\item The collineation induced by $\theta$ in a fixed symp pointwise fixes an ideal subspace of rank $2$ and corank $4$.  \end{compactenum}
\end{lemma}
\begin{proof} 
\begin{compactenum}[$(1)$]
\item This follows from Corollary~6.1 and Lemma~6.7 of \cite{NPVV}.
\item By the previous result, the collineation induced by $\theta$ in the residue of a fixed point does not map a point of $\mathsf{E_{6,1}}(\K)$ to a collinear point. Then the assertion follows from Theorem~4.1 of \cite{NPVV}; see also Proposition~6.17 of \emph{loc.\ cit.} 
\item By (the proof of) Proposition~6.9 of \cite{NPVV} the collineation induced by $\theta$ in a fixed symp does not map points to collinear ones and pointwise fixes a subquadrangle, that is, a subspace of rank 2. By Proposition~3.8$(v)$, this subquadrangle is also an ideal subspace of corank $4$. \qedhere \end{compactenum}
\end{proof}

A substructure with the properties $(2)$ and $(3)$ of \cref{evendist}, that is, a substructure consisting of a set $\cP$ of points, a set $\cL$ of lines and a set $\cS$ of symps of $\mathsf{E_{7,7}}(\K)$ such that \begin{compactenum}[$(1)$] \item the set of lines and symps in $\cL$ and $\cS$, respectively, incident with a point $p\in\cP$ defines an ideal Veronesean in the residue of $p$; \item the set of points and lines in $\cP$ and $\cL$, respectively, incident with a symp $\xi\in\cS$ defines an ideal subspace of rank $2$ and corrank $4$ in $\xi$,\end{compactenum} will be called an \emph{ideal dual polar Veronesean}.
\subsubsection{Main theorem}
In this section we prove Theroem~\ref{thm:main1} for $\sE_7$ buildings. 

\begin{thm}\label{theoE7}
Let $\theta$ be an automorphism of $\mathsf{E_{7}}(\K)$, for some field $\K$. Then the following are equivalent
\begin{compactenum}[$(1)$]
\item The fixed point structure of $\theta$ induced in $\mathsf{E_{7,1}}(\K)$ is a fully embedded metasymplectic space $\mathsf{F_{4,1}}(\K,\LL)$, with $\LL$ a quadratic extension of $\K$, isometrically embedded as long root subgroup geometry, or the fixed point structure of $\theta$ induced in $\mathsf{E_{7,7}}(\K)$ is an ideal dual polar Veronesean. 
\item $\disp(\theta)$ is contained in a single nontrivial conjugacy class.  
\end{compactenum} 
\end{thm}

Note that the two examples correspond to the opposition diagram $\mathsf{E_{7;3}}$ and $\mathsf{E_{7;4}}$, respectively, and to the fix diagrams $\mathsf{E_{7;4}}$ and $\mathsf{E_{7;3}}$, respectively. In the latter case, the projective plane determined by the ideal Veronesean is a quaternion plane, or a plane over an inseparable extension of degree 4 of $\K$, in characteristic $2$. We will call such a dual polar space briefly \emph{quaternion}.

We start with opposition diagram $\mathsf{E_{7;3}}$ and fix diagram $\mathsf{E_{7;4}}$.

\begin{prop}\label{propE7}
Let $\theta$ be an automorphism of $\mathsf{E_{7}}(\K)$, for some field $\K$, the fixed point structure of which induced in $\mathsf{E_{7,1}}(\K)$ is a fully embedded metasymplectic space $\mathsf{F_{4,1}}(\K,\LL)$, with $\LL$ a quadratic extension of $\K$, isometrically embedded as long root subgroup geometry. Then $\theta$ is \uniclass.
\end{prop}

\begin{proof}
Let $C$ be a chamber of $\mathsf{E_7}(\K)$, and we consider $C$ in $\Delta=\mathsf{E_{7,7}}(\K)$. We show that $\delta(C,C^\theta)$ belongs to the conjugacy class $\bC$ of the Weyl group of type $\mathsf{E_7}$ defined by the opposition diagram $\mathsf{E_{7;3}}$ or, equivalently, the fix diagram $\mathsf{E_{7;4}}$. It induces the fix diagram $\mathsf{D_{5;3}^1}$ in a fixed vertex of type $6$.  

Let $\xi$ be the symp of $C$. By \cref{notspecial}, $\xi$ is mapped either to itself, to a symp intersecting $\xi$ in a line, or to an opposite symp.  At the same time, points are only mapped onto collinear and opposite ones, by virtue of  \cref{nofixedpoints} and \cref{nofixedpoints}. We will use this without further reference. 

So there are three possibilities.
\begin{enumerate}
\item \emph{The symp $\xi$ is fixed.} Then the result follows from \cref{Bn1uni}(2) applied to $\mathsf{D_6}(\K)$. 
\item \emph{The symps $\xi$ and $\xi^\theta$ share a line $M$.} By \cref{notspecial}$(3)$, the line  $M$ is fixed. Since $\theta$ acts fixed point freely on $M$, \cref{notspecial}$(4)$ implies that $\theta$ belongs to $\bC$. 
\item \emph{The symps $\xi$ and $\xi^\theta$ are opposite.}  Our first aim is to show that the set $\wE$ of points of $\xi$ mapped to a collinear point is a nondegenerate ideal subspace (of corank $2$), actually isomorphic to $\mathsf{B_{4,1}}(\KK,\LL)$, for a separable quadratic extension $\LL$ of $\KK$.  

We first claim that there exist at least two opposite fixed lines intersecting $\xi$. Indeed, if each point of $\xi$ is mapped onto a collinear one, this is trivial. So assume there is a point $x\in\xi$ mapped to an opposite one. Considering the mapping $\theta_x$, noting that it is a domestic duality, and that $\xi$ is mapped onto an opposite line through $x$, we infer from \cite{Mal:12} that there exist at least two nonplanar lines $L_1,L_2$ through $x$ in $\xi$ which are absolute with respect to $\theta_x$, that is, whose image $\xi_1,\xi_2$ under $\theta_x$ contains $L_1,L_2$, respectively.  Let $i\in\{1,2\}$. Then the foregoing implies that $L_i^\theta$ intersects $\xi_i$ in a point $y_i$. The inverse image $y_i'$ of $y_i$ is contained in $\xi_i$, hence is not opposite $y_i$; if follows that $y_i'$ is collinear to $y_i$. The points $y_1'$ and $y_2'$ belong to $\xi$ and are not collinear. Now note that the line $M_1:=y_1y_1'$ is opposite the line $M_2:=y_2y_2'$. Indeed, $y_2\in\xi^\theta$, which is opposite $\xi$, so $y_2^\perp\cap\xi=\{y_2'\}$. It follows that $y_2$ is opposite $y_1'$. Likewise, $y_1$ is opposite $y_2'$. This implies that $M_1$ and $M_2$ are opposite. The claim is proved.

Now let $K$ be a fixed line contained in a symp $\zeta_i$ together with $y_iy_i'$. Since $y_1'$ and $y_2'$ are symplectic, they have to be collinear with the same point of $K$, which then necessarily belongs to $\xi$. Also, if two collinear points $z_1,z_2$ belong to $\wE$, then considering a symp through the lines $z_1z_1^\theta$ and $z_2z_2^\theta$, we see that all points of $z_1z_2$ belong to $\wE$; hence $\wE$ is a subspaces. In the metasymplectic space $\mathsf{F_{4,4}}(\K,\LL)$, this amounts to a \emph{hyperbolic line}. These remarks now imply that the \emph{extended equator} of $\mathsf{F_{4,4}}(\K,\LL)$ defined by the points corresponding to the fixed lines $M_1$ and $M_2$ consists solely of lines intersecting $\xi$ in a point. Since this geometry is a polar space isomorphic to $\mathsf{B_{4,1}}(\K,\LL)$, and $\wE$ is a subspace, we see that the former is embedded in $\xi$ and our first aim follows from the fact that the codimension of the subspace of the ambient projective space of $\xi$ needed to intersect in $\wE$ is the same as the corank of $\wE$ as a polar space in $\xi$, namely, $2$.

Now it is convenient to consider the dual situation, that is, the translation to $\mathsf{E_{7,1}}(\K)$. The symps $\xi$ and $\xi^\theta$ correspond to two opposite points $x,x^\theta$. The fact that in the previous paragraphs we had two opposite fixed lines intersecting $\xi$ and $\xi^\theta$ implies that the imaginary line defined by $x$ and $x^\theta$ is stablized, hence the equator $E(x,x^\theta)$ is stabilized. The previous paragraph translated to $E(x,x^\theta)\cong\mathsf{D_{6,2}}(\K)$ means that the fixed structure of $\theta$ in $E(x,x^\theta)$ is $\mathsf{B_{4,2}}(\KK,\LL)$.  The chambers $C$ and $C^\theta$ induce unique chambers $D$ and $D^\theta$ in $E(x,x^\theta)$. Noting that $\mathsf{B_{4,1}}(\KK,\LL)$ is an ideal subspace in $\mathsf{D_{6,1}}(\K)$, we find an apartment $\cA$ containing $C,C^\theta, D,D^\theta$ and an apartment of $E(x,x^\theta)$ corresponding to the fix diagram $\mathsf{D_{6;3}^1}$. However, an apartment of type $\mathsf{B_{4,2}}$ is also of type $\mathsf{F_{4,1}}$. Hence the fix structure in $\cA$ amounts to the fix diagram $\mathsf{E_{7;4}}$, proving $\delta(C,C^\theta)\in\bC$.
\end{enumerate}
The proposition is proved. \end{proof}

We now turn to the opposition diagram $\mathsf{E_{7;4}}$. Denote by $\bC$ the displacement class defined by the fix diagram $\mathsf{E_{7;3}}$, that is, multiplication with the longest word of a standard $\mathsf{D_4}$ subsystem. 

\begin{prop}\label{propE7bis}
Let $\theta$ be an automorphism of $\mathsf{E_{7}}(\K)$, for some field $\K$, the fixed point structure of which induced in $\mathsf{E_{7,7}}(\K)$ is an ideal dual polar Veronesean, hence a fully embedded quaternion dual polar space, isometrically embedded. Then $\theta$ is \uniclass\ and its displacement belongs to $\bC$. 
\end{prop}

\begin{proof}
Let $C$ be an arbitrary chamber, and let $x$ be the vertex of type $7$ of $C$. Then, as a point of $\mathsf{E_{7,7}}(\K)$, \cref{evendist} asserts that $x$ is either fixed or mapped onto a symplectic point. First suppose that $x=x^\theta$. Then everything happens in the residue of $x$ and the assertion follows from \cref{E6uni}, \cref{Bn1uni}(1) and \cref{evendist}$(3)$. 

Hence we may assume that $x$ is symplectic to $x^\theta$. Let $\zeta$ be the symp containing $x$ and $x^\theta$. In view of \cref{evendist}, the result now follows from \cref{ext} and \cref{evendist}. 
\end{proof}

We can now finish the proof of \cref{theoE7}. Suppose the automorphism $\theta$ of $\mathsf{E_7}(\K)$ is \uniclass. If $\theta$ fixes a chamber, it is trivlal; if it maps a chamber to an opposite it is anisotropic. If it does not fix any chamber and is domestic, then, according to the main result of \cite{NPVV}, we either have $(1)$ of \cref{theoE7}, or $\theta$ pointwise fixes an equator geometry of type $\mathsf{D_{6,2}}$ in the associated long root geometry. In the latter case $\theta$ belongs to the opposition diagram $\sE_{7;4}$, and hence the fix diagram $\mathsf{E_{7;3}}$. It follows that $\theta$ does not fix any vertex of type $4$, a contradiction since the above mentioned equator contains type $4$ vertices. \cref{theoE7} is now proved.

\subsection{Buildings of type $\sE_8$}

We begin with some preliminaries. Buildings of type $\mathsf{E}_8$ constitute the unique class of buildings where always only at most one type of nontrivial uniclass collineation exists, and it has opposition diagram $\mathsf{E_{8;4}}$ and $\mathsf{E_{8;4}}$ fix diagram. 

The following lemma is shown in \cite{PVMexc4}.

\begin{prop}\label{coruniclassE8}
Let $\theta$ be an automorphism of $\mathsf{E_8}(\K)$ with opposition diagram $\mathsf{E_{8;4}}$ and fixed diagram $\mathsf{E_{8;4}}$. Let $p$ be a point such that $p^\theta$ is opposite $p$. Then the fix structure of $\theta$ contained in $E(p,p^\theta)$ is a parapolar space isomorphic to $\mathsf{F_{4,1}}(\K,\LL)$, isometrically and fully embedded, where $\LL$ is a quadratic extension of $\K$.
\end{prop}

Also the proofs of the following properties are contained in  \cite{PVMexc4}.
\begin{lemma}\label{nospecialE8}
Let $\theta$ be an automorphism of $\mathsf{E_8}(\K)$ with opposition diagram $\mathsf{E_{8;4}}$ and fixed diagram $\mathsf{E_{8;4}}$.
\begin{compactenum}[$(1)$] \item No point is mapped onto a special point. \item No point is mapped onto a collinear point.\item No symp is mapped onto an adjacent one. \item If a point $x$ is mapped onto a symplectic one, then the corresponding symp $\xi$ is fixed. 
\end{compactenum}
\end{lemma}

The main result for $\sE_8$ buildings is as follows. 

\begin{thm}\label{E8uni}
An automorphism $\theta$ of $\Delta:=\mathsf{E_{8}}(\K)$, for some field $\K$, is \uniclass\ if and only if it  is either an anisotropic collineation or pointwise fixes a fully (and automatically isometrically) embedded metasymplectic space $\mathsf{F_{4,1}}(\K,\HH)$, with $\HH$ a quaternion algebra over $\K$ or an inseparable quadratic  field extension of degree $4$, in the associated long root geometry $\mathsf{E_{8,8}}(\K)$.
\end{thm}

\begin{proof}
First let $\theta$ be uniclass, say each displacement belongs to the class $\bC$. If $\theta$ maps some chamber to an opposite, then $\theta$ is an anisotropic colineation. If $\theta$ fixes some chamber, then $\theta$ is the identity. hence we may assume that $\theta$ is domestic and does not fix any chamber. According to \cite{PVMexc4}, $\theta$ either pointwise fixes a fully (and automatically isometrically) embedded metasymplectic space $\mathsf{F_{4,1}}(\K,\HH)$, with $\HH$ a  quaternion algebra over $\K$ an inseparable quadratic  field extension of degree $4$, in the associated long root geometry $\mathsf{(E_{8,8}}(\K)$, or pointwise fixes an equator geometry of type $\mathsf{E_{7,1}}$ in $\mathsf{E_{8,8}}(\K)$. In the latter case singular subspaces of dimension $3$ are fixed. In both cases the class $\bC$ is determined by the opposition diagram $\mathsf{E_{8;4}}$, which has the same fix diagram. Hence no vertex of type $5$ is fixed. But these correspond to singular $3$-spaces in $\mathsf{E_{8,8}}(\K)$. Hence only the first case occurs. 

Now we show the converse, that is, we assume that $\theta$ pointwise fixes a fully (and automatically isometrically) embedded quaternion metasymplectic space $\mathsf{F_{4,1}}(\K,\HH)$, with $\HH$ as in te statement of the theorem, in the associated long root geometry $\mathsf{E_{8,8}}(\K)$.

Let $C$ be a chamber of $\Delta$, and denote by $\bC$ the class of displacement determined by the longest element of a standard $\mathsf{D_4}$ subsystem of $\mathsf{E_8}$. Let $p$ be the element of type $8$ in $C$ and conceive $\Delta$ as the long root subgroup geometry $\mathsf{E_{8,8}}(\K)$.  According to \cref{nospecialE8} and \cref{nospecialE8}, there are three possibilities.
\begin{compactenum}[$(1)$]
\item $p^\theta=p$. In this case, everything happens in the residue of $p$, where we can apply \cref{propE7bis} and  and the residuality of being \uniclass\ to conclude.
\item $p$ is opposite $p^\theta$. In this case, we use \cref{coruniclassE8} to find an apartment $\cA$ of $\Delta$ containing $p$ and $p^\theta$, and such that the equator $E(p,p^\theta)$ of $p$ and $p^\theta$ contains the projections of $C$ and $C^\theta$ into $E(p,p^\theta)$. Clearly the fixed points in $\cA$ are those in $E(p,p^\theta)$ and constitute a subgeometry of type $\mathsf{F_{4,1}}$. Hence the displacement of $C$ belongs to $\bC$. 
\item $p$ is symplectic to $p^\theta$. By \cref{nospecialE8} the symp $\zeta$ containing $p$ and $p^\theta$ is fixed by $\theta$.  Let $D$ be the projection of $C$ onto $\zeta$. Then $D^\theta$ is the projection of $C^\theta$ onto $\zeta$. Also, $\zeta$ belongs to the convex closure of $C$ and $C^\theta$, hence also $D$ and $D^\theta$ belong to the convex closure of $C$ and $C^\theta$. The fixed point structure of $\theta$ in $\zeta$ is an ideal subspace of type $\mathsf{B_{3,1}}$ and hence assertion follows from \cref{ext}.
\end{compactenum}
This completes the proof of the proposition.\end{proof}

\goodbreak

\subsection{Proof of Theorem~\ref{thm:main1} and Corollaries~\ref{cor:bicapped} and~\ref{cor:pairing}}

We now summarise the proof of Theorem~\ref{thm:main1}. 

\begin{proof}[Proof of Theorem~\ref{thm:main1}] (1) is given in Theorem~\ref{rank2}, and (2) is Theorem~\ref{Anuni}. (3) is proved in Theorem~\ref{Bn1uni} along with Theorem~\ref{nonondom} to exclude the possibility of uniclass trialities of $\sD_4$. The $\sF_4$ case is proved in Theorems~\ref{thm:F41} and~\ref{thm:F42}, and the $\sE_n$ cases $(n=6,7,8$) are proved in Theorems~\ref{E6uni}, \ref{theoE7}, \ref{E8uni}.
\end{proof}

We now give the proof of Corollaries~\ref{cor:bicapped} and~\ref{cor:pairing}. 

\begin{proof}[Proof of Corollary~\ref{cor:bicapped}] To compute the $\sigma$-conjugacy class of each uniclass automorphism (c.f. Table~\ref{AllWS}) note that if $J$ is the type of the maximal fixed simplices, then $w_{S\backslash J}\in\disp(\theta)$ (by Proposition~\ref{prop:attaindisplacement}), and hence $\disp(\theta)=\Cl^{\sigma}(w_{S\backslash J})$. The list of $\sigma$-classes given in Table~\ref{AllWS} follows, and comparing with Theorem~\ref{thm:bicapped} we arrive at Corollary~\ref{cor:bicapped}. 
\end{proof}

Our classification of uniclass automorphisms (Theorem~\ref{thm:main1}) implies that for each uniclass automorphism $\theta$ of an irreducible thick spherical building of type $\mathsf{X}_n$, there exists a uniclass automorphism $\theta'$ of a possibly different thick building of the same type $\mathsf{X}_n$ such that $\Fix(\theta)=\Opp(\theta')$ and $\Opp(\theta)=\Fix(\theta')$. We say that $\theta$ and $\theta'$ are \emph{paired}, and also that the corresponding fixed Weyl substructures are \emph{paired}. 

If $\theta$ is paired with itself, it is called \emph{self-paired}. Note that there is exactly one irreducible type (with rank at least 2) with the property that there exists exactly one type of nontrivial Weyl substructure, which is necessarily self-paired, and that is type $\mathsf{E_8}$. Moreover, note that several types do not admit nontrivial Weyl substructures, for instance $\mathsf{A}_{2n}$, $\mathsf{I_2}(2n+1)$. 

\begin{remark} On the level of thin buildings, the pairing of uniclass automorphisms can be described explicitly as follows. Let $\bC$ be a bi-capped class of $\sigma$-involutions. If $w\in \bC$ then the automorphism $\theta=w\sigma$ acts on the Coxeter complex with displacement set $\bC$, and the automorphism $\theta'=ww_0\sigma_0\sigma$ acts on the Coxeter complex with displacement set $\psi(\bC)=\bC w_0$. Thus $\theta$ and $\theta'$ are paired. 
\end{remark}

\begin{proof}[Proof of Corollary~\ref{cor:pairing}]
For the classical cases, Corollary~\ref{cor:pairing} follows immediately from Theorem~\ref{thm:main1}. For the exceptional cases existence of automorphisms fixing the desired Weyl substructures is proved the the associated papers~\cite{Lam-Mal:23,Mal:12,NPVV,PVMexc4}. By Proposition~\ref{prop:duality}, the ranks of paired Weyl substructures necessarily add up to the rank of the building in question. 
\end{proof}

We note that Theorem~\ref{thm:main1} ``nearly'' proves that an automorphism of $\Delta$ is uniclass if and only if it is either anisotropic or its fixed element structure is a Weyl substructure. Indeed Theorem~\ref{thm:main1} proves the forwards direction, however for the reverse direction, for example in the $\sF_4$ case,  Theorem~\ref{thm:main1} implies the following weaker statement: If $\theta$ is either type preserving and fixes an ideal quadrangular Veronesian, or $\theta$ is a polarity (hence fixing a Moufang octagon) then $\theta$ is uniclass. Thus to have the full equivalence, one needs to prove additionally that no collineation fixes precisely a Moufang octagon, and that every duality fixing a Moufang octagon is necessarily a polarity. Similar comments apply to dualities in buildings of type $\mathsf{I}_2(m)$, $\sA_n$, and $\sE_6$. In these latter cases the proof of the additional statements are relatively straightforward, however in the $\sF_4$ case they are non-trivial. Therefore we shall postpone the details to future work~\cite{PRV}, where we will also give an axiomatic uniform definition of Weyl substructures.

\section{The finite case}\label{sec:5}

Theorem~\ref{thm:counts} determines the cardinalities $|\Delta_w(\theta)|$ for a uniclass automorphism~$\theta$ in terms of the class sum $\bC(q^{1/2})$, where $\bC=\disp(\theta)$. In the following theorem we use Corollary~\ref{cor:bicapped} to determine $|\Delta_w(\theta)|$ another way, using the fixed Weyl substructure. This leads to another formula for $|\Delta_w(\theta)|$, and combining with Theorem~\ref{thm:counts} we deduce a formula for $\bC(q^{1/2})$. 

\begin{thm}\label{thm:counts2}
Let $\theta$ be a uniclass automorphism of a thick spherical building $\Delta$ of type $(W,S)$ with fixed structure a Weyl substructure $\Delta'$ of type $(W',S')$ with parameters $q'=(q_{s'}')_{s'\in S'}$. For $w\in\disp(\theta)$ we have
$$
|\Delta_w(\theta)|=W_J(q)W'(q')q_w^{1/2}q_{w_J}^{-1/2},
$$
and writing $\bC=\disp(\theta)$ we have
$$
\bC(q^{1/2})=\frac{W(q)q_{w_J}^{1/2}}{W_J(q)W'(q')}.
$$
Here $W'(q')$ is the Poincar\'e polynomial of $(W',S')$, and so $W'(q')$ is the number of chambers of the Weyl substructure~$\Delta'$.
\end{thm}

\begin{proof}
Let $w_J$ be the unique minimal length element of the bi-capped class $\bC$ (using Corollary~\ref{cor:bicapped}). The chambers of the fixed Weyl substructure are the simplices of type $S\backslash J$ in the building $\Delta$ that are fixed by $\theta$. Thus the number $F_J$ of simplices of type $S\backslash J$ of $\Delta$ fixed by~$\theta$ equals the total number of chambers of $\Delta'$, giving $F_J=W'(q')$. 

We now count $F_J$ in another way. If $C\in\Ch(\Delta)$ then the type $S\backslash J$ simplex of $C$ is fixed if and only if $\delta(C,C^{\theta})\in W_J$. Since $w_J$ is the minimal length element of $\disp(\theta)$ it follows that the type $S\backslash J$ simplex of $C$ is fixed if and only if $\delta(C,C^{\theta})=w_J$. Moreover if $D\in\Res_J(C)$ then $\delta(D,D^{\theta})\in W_J$, and since $w_J$ has minimal length in $\disp(\theta)$ we have $\delta(D,D^{\theta})=w_J$. Hence $F_J=|\Delta_{w_J}(\theta)|/W_J(q)$, and so $W'(q')=|\Delta_{w_J}(\theta)|/W_J(q)$. The result now follows from Theorem~\ref{thm:counts}.
\end{proof}

\begin{example}
Let $\theta$ be a uniclass automorphism of the building $\Delta=\sE_7(q)$ with fixed Weyl substructure $\Delta'$ is  a building of type $\sF_4$ with parameters $(q,q^2)$. Then $\disp(\theta)=\Cl(w_J)$, where $J=\{2,5,7\}$. By Theorem~\ref{thm:counts2} and well known factorisations of Poincar\'{e} polynomials  of type $\sE_7$ and $\sF_4$ (see \cite{Mac:72}) we have
\begin{align*}
|\Delta_w(\theta)|&=\frac{(q+1)^3(q^2+1)(q^4+1)(q^5+1)(q^9+1)(q^6-1)(q^{12}-1)}{(q-1)^2}q^{(\ell(w)-3)/2}
\end{align*}
for all $w\in\disp(\theta)$, and
$$
\bC(q^{1/2})=\frac{(q^5-1)(q^9-1)(q^{14}-1)}{(q-1)^3(q+1)}
$$
where $\bC=\Cl(s_3s_5s_7)$. In particular, we also see that $|\bC|=\bC(1)=315$.
\end{example}

\cref{FWS} provides a list of Weyl substructures in the finite irreducible case using the following conventions. When a polar space has type $\mathsf{B}_n$ and parameters $(q,q)$ we mention in the column ``Remarks'' whether it concerns a parabolic or symplectic polar space; in all other cases it is clear from the parameters which polar space is meant. Also, we adopt the convention $\mathsf{B_1}=\mathsf{A_1}$. Moreover, if parameters $(s,t)$ are stated for type $\mathsf{B_1}$, one has to omit $s$, that is, the Weyl substructure has rank 1 and consists of exactly $t+1$ points.  In general, parameters $(s,t)$ for a building of type $\mathsf{B}_n$ means that in the corresponding polar space there are precisely $s+1$ points on a line and every submaximal singular subspace is contained in exactly $t+1$ maximal singular subspaces. In terms of the notation of \cref{sec:counting}, we have $s=q_{s_i}$, for $i\in\{1,2,\ldots,n-1\}$ and $q_{s_n}=t$.  There is some ambiguity in rank $2$, in which case we call points the vertices corresponding to the smallest fixed type (in Bourbaki labelling) of the absolute diagram. For instance for $\Gamma(\mathsf{D_4},\mathsf{B_2})$ over $\FF_q$, there are vertices of type $1$ fixed, so these are the points of the fixed quadrangle, which has then parameters $(q,q^2)$. However, for $\Gamma'(\mathsf{D_4},\mathsf{B_2})$ over $\FF_q$, there are no vertices of type 1 fixed, but there are vertices of type $2$ fixed. The parameters of the fixed quadrangle are then $(q^2,q)$. Note that these cases correspond to each other under a triality.

\begin{table}[h!]
\renewcommand{\arraystretch}{1.5}
\begin{tabular}{l|c||c|c||c|c}
\mbox{Absolute type} &  \mbox{Parameters} & \mbox{Relative type} & \mbox{Parameters}& \mbox{Remarks}& Ref.\\ \hline\hline
\multirow{2}{*}{$\mathsf{A}_{2n+1}$, $n\geq 1$} & \multirow{2}{*}{$q$} & $\mathsf{B}_{n+1}$ & $(q,q)$ & Symplectic&(1)\\ 
& &  $\mathsf{A}_{n-1}$ & $q^2$ & &(2)\\ \hline 
$\mathsf{B}_n$, $n\geq 2$ & $(q,q)$ & $\mathsf{B}_{n-1}$ & $(q,q^2)$ & Parabolic, $q$ odd&(3)\\ \hline
$\mathsf{B}_{2n}$, $n\geq 1$& $(q,q)$ & $\mathsf{B}_n$ & $(q^2,q^2)$ & Symplectic, $q$\mbox{ odd}&(4)\\ \hline
$\mathsf{B}_{2n}$, $n\geq 1$ & $(q,q^2)$ & $\mathsf{B}_{n}$ & $(q^2,q^3)$ &&(5)\\ \hline
$\mathsf{B_2}$ & $(2^{2e+1},2^{2e+1})$ & $\mathsf{A_1}$ & $2^{4e+1}$ & Polarity&(6)\\ \hline
\multirow{2}{*}{$\mathsf{D}_n$, $n\geq 3$} & \multirow{2}{*}{$q$} & $\mathsf{B}_{n-1}$ & $(q,q)$ & Parabolic &(7)\\ 
& & $\mathsf{B}_{n-2}$ & $(q,q^2)$ & &(8)\\ \hline
$\mathsf{D}_{2n}$, $n\geq 2$ & $q$ & $\mathsf{B}_{n}$ & $(q^2,q)$ & &(9)\\ \hline
$\mathsf{E_6}$ & $q$ & $\mathsf{F_4}$ & $(q,q)$ & Polarity&(11)\\ \hline
$\mathsf{E_7}$ & $q$ & $\mathsf{F_4}$ & $(q,q^2)$ & &(12)\\ \hline
$\mathsf{F_4}$ & $2^{2e+1}$ & $\mathsf{I_2(8)}$ & $(2^{2e+1},2^{4e+2})$ & Polarity&(13)\\ \hline
$\mathsf{G_2}$ & $q$ & $\mathsf{A_1}$ & $q^3$ & $q\equiv2\mod3$&(14)\\ \hline
$\mathsf{G_2}$ & $3^{2e+1}$ & $\mathsf{A_1}$ & $3^{6e+3}$ & Polarity&(15)\\ \hline
\end{tabular}
\vspace{6pt}
\caption{Nontrivial finite Weyl substructures\label{FWS}}\label{tab:FWS}
\end{table}

We will now prove that this table is complete, if one disregards ovoids in finite Moufang octagons of order $(q,q^2)$. We conjecture, however, that no such ovoid is the fix structure of any collineation. 

\begin{thm}\label{thm:finitecase}
Disregarding the possible existence of ovoids in Moufang octagons of order $(2^{2e+1},2^{4e+2})$, the list of nontrivial (that is, $\Delta'\neq \Delta$) Weyl substructures that occur in finite thick irreducible Moufang spherical buildings is given in \emph{\cref{tab:FWS}}
\end{thm}

\begin{proof}
For type $\mathsf{A}_n$ there is nothing to prove as all possibilities can occur over finite fields. 

By Proposition~\ref{prop:anisotropic}, a uniclass collineation of a building of type $\mathsf{B}_n$ or $\mathsf{D}_n$, fixing a geometric subspace of corank at least $2$ or at least $3$, respectively, acts as a anisotropic collineation on an irreducible residue of rank at least 2 or 3, respectively. If $\Delta$ is a parabolic polar space in odd characteristic, then a nontrivial colllneation can fix a nondegenerate hyperplane inducing an elliptic polar space in $\Delta$ and its pole. In characteristic 2 such a nontrivial collineation does not exist as a parabolic polar space is isomorphic to a symplectic polar space, and in this case the collineation must be nontrvial over the space generated by the elliptic quadric, a contradiction.   In hyperbolic polar spaces, one can find nontrivial collineations pointwise fixing a suitable hyperplane or subhyperplane (a hyperplane of a hyperplane) and their pole so that uniclass collineations of corank $i=1,2$ arise. The Weyl substructures are parabolic and elliptic polar spaces of one or two, respectively, ranks lower.  See rows~(3), (7) and (8) of \cref{FWS}

Composition line spreads in symplectic polar spaces of even rank (row (4) in \cref{FWS}) in odd characteristic, in hyperbolic polar spaces of even rank (row (9) in \cref{FWS}) in arbitrary characteristic, and in elliptic polar spaces of even rank (row (5) in \cref{FWS}) and arbitrary characteristic exist in the finite case by Section~3.3 of \cite{PVMclass}; the same reference shows that there are no others. 

Concerning type $\mathsf{E_6}$, it is also well known (for an explicit geometric proof for arbitrary fields, see \cite{ADS-NSNS-HVM}) that the building $\mathsf{F_4}(q)$ arises as fix structure of a (symplectic) polarity in the building $\mathsf{E_6}(q)$. Moreover, since there do not exist quaternion or octonion division algebras over finite fields, there do not exist ideal Veroneseans in $\mathsf{E_6}(q)$.  

Proposition~5.7 of \cite{DSV} asserts that each building $\mathsf{F_4}(q,q^2)$ is the fix structure via a partial composition spread of each nontrivial member of a nontrivial group of collineations of $\mathsf{E_7}(q)$, cf.\ Proposition~8.1 in \cite{NPVV}.   Also, since there do not exist quaternion division algebras over finite fields, there do not exist ideal dual polar Veroneseans in $\mathsf{E_7}(q)$.

For the same reason as in the previous paragraph, there do not exist finite quaternion metasymplectic spaces, hence no uniclass collineations of $\mathsf{E_8}(q)$ exist at all. 

By \cite[Main Result (DOM14)$(iii)$]{Lam-Mal:23}, an ideal quadrangular Veroneseans in buildings isomorphic to $\mathsf{F_4}(q,q)$ or $\mathsf{F_4}(q,q^2)$ stems from a Moufang quarangle with Tits index $\mathsf{^2D_{5,2}^{(2)}}$, and these only exist when the base field admits a quaternion division algebra, hence not in the finite case. On the other hand, a polarity of $\mathsf{F_4}(q)$ that produces a Moufang octagon always exists as soon as $q$ is a power of $2$ with an odd exponent, cf.\ \cite{Tits:60/61}.

The assertions for Moufang hexagons follow from \cite{Tits:60/61} (polarities) and \cite{PVMexc2}.

Finally suppose that a spread is fixed in a Moufang octagon, where we fix the duality class by requiring that each line pencil is paramatrized by a Suzuki ovoid and each point row by the base field. Two members of the spread determine a unique non-thick but full suboctagon (cf.\ \cite{JVM}). By the definition of spread, each point $p$ of that suboctagon lies at distance at most 3 from a member $L$ of the spread, and hence  $L$ is contained in the convex closure of $p$ and its image, hence in the suboctagon. It follows that we find a spread in the suboctagon, inducing an ovoid in the generalized quadrangle underlying the nonthick suboctagon. That quadrangle is a subquadrangle of a symplectic quadrangle and so the collineation acts on an ambient projective $3$-space, where it pointwise fixes an ovoid, forcing it to be the identity. 
\end{proof}

\section{Connection with the Freudenthal--Tits Magic Square}\label{FTMS}
There are several ways to introduce the Freudenthal--Tits Magic Square. Perhaps its most basic form is as a table of Lie algebras constructed by Tits \cite{Tits:65}. As explained in \cite{Tits:65}, this table has several forms: split, compact or mixed. But perhaps the most popular form is the mixed one. Replacing each algebra with its corresponding Tits index (cf.~\cite{Tits:66}), written as a Tits diagram, we obtain, with the conventions of \cite{Tits:66}, the following appearance of the nine cells in the South-East corner of the Freudenthal--Tits Magic Square:

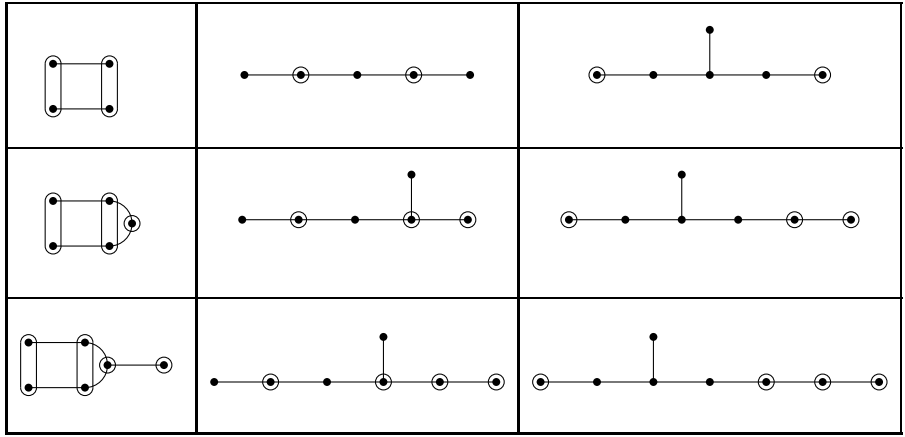
\begin{center}
\begin{tabular}{|c|c|c|}
\hline
\raisebox{-.45cm}{\begin{tikzpicture}[scale=0.3]
\node at (0,0.3) {};
\fill (-3.75,0) circle (5pt);
\fill (-3.75,-2) circle (5pt);
\draw (-4.1,0) arc [start angle=180, end angle=0, radius = 10pt];
\draw (-4.1,-2) arc [start angle=180, end angle=360, radius = 10pt];
\draw (-4.1,0)--(-4.1,-2);
\draw (-3.4,0)--(-3.4,-2);
\fill (-1.25,0) circle (5pt);
\fill (-1.25,-2) circle (5pt);
\draw (-1.6,0) arc [start angle=180, end angle=0, radius = 10pt];
\draw (-1.6,-2) arc [start angle=180, end angle=360, radius = 10pt];
\draw (-1.6,0)--(-1.6,-2);
\draw (-0.9,0)--(-0.9,-2);
\draw (-3.75,0)--(-1.25,0);
\draw (-3.75,-2)--(-1.25,-2);
\end{tikzpicture} }
&
\begin{tikzpicture}[scale=0.3]
\node at (0,0.3) {};
\fill (-3.75,0) circle (5pt);
\fill (-1.25,0) circle (5pt);
\draw (-1.25,0) circle (10pt);
\fill (1.25,0) circle (5pt);
\fill (3.75,0) circle (5pt);
\draw (3.75,0) circle (10pt);
\fill (6.25,0) circle (5pt);
\draw (-3.6,0)--(-1.25,0);
\draw (-1.25,0)--(1.1,0);
\draw (3.75,-0)--(6.25,0);
\draw (1.4,0)--(3.6,0);
\end{tikzpicture} 
&
{\begin{tikzpicture}[scale=0.3]
\node at (0,2.7) {};
\fill (-3.75,0) circle (5pt);
\fill (-1.25,0) circle (5pt);
\draw (-3.75,0) circle (10pt);
\fill (1.25,0) circle (5pt);
\fill (3.75,0) circle (5pt);
\draw (6.25,0) circle (10pt);
\fill (6.25,0) circle (5pt);
\draw (-3.6,0)--(-1.25,0);
\draw (-1.25,0)--(1.1,0);
\draw (3.75,-0)--(6.25,0);
\draw (1.4,0)--(3.6,0);
\fill (1.25,2) circle (5pt);
\draw (1.25,0)--(1.25,2);
\end{tikzpicture} }
\\ [20pt] \hline
\raisebox{-.35cm}{\begin{tikzpicture}[scale=0.3]
\node at (0,0.3) {};
\fill (-3.75,0) circle (5pt);
\fill (-3.75,-2) circle (5pt);
\draw (-4.1,0) arc [start angle=180, end angle=0, radius = 10pt];
\draw (-4.1,-2) arc [start angle=180, end angle=360, radius = 10pt];
\draw (-4.1,0)--(-4.1,-2);
\draw (-3.4,0)--(-3.4,-2);
\fill (-1.25,0) circle (5pt);
\fill (-1.25,-2) circle (5pt);
\draw (-1.6,0) arc [start angle=180, end angle=0, radius = 10pt];
\draw (-1.6,-2) arc [start angle=180, end angle=360, radius = 10pt];
\draw (-1.6,0)--(-1.6,-2);
\draw (-0.9,0)--(-0.9,-2);
\draw (-3.75,0)--(-1.25,0);
\draw (-3.75,-2)--(-1.25,-2);
\draw (-1.25,0) arc [start angle=90, end angle=-90, radius=1cm];
\fill (-0.25,-1) circle (5pt);
\draw (-0.25,-1) circle (10pt);
\end{tikzpicture} }
&
\begin{tikzpicture}[scale=0.3]
\node at (0,2.7) {};
\fill (-3.75,0) circle (5pt);
\fill (-1.25,0) circle (5pt);
\draw (-1.25,0) circle (10pt);
\fill (1.25,0) circle (5pt);
\fill (3.75,0) circle (5pt);
\draw (3.75,0) circle (10pt);
\fill (6.25,0) circle (5pt);
\draw (-3.6,0)--(-1.25,0);
\draw (-1.25,0)--(1.1,0);
\draw (3.75,-0)--(6.25,0);
\draw (1.4,0)--(3.6,0);
\fill (3.75,2) circle (5pt);
\draw (6.25,0) circle (10pt);
\draw (3.75,0)--(3.75,2);
\end{tikzpicture} 
&
\begin{tikzpicture}[scale=0.3]
\node at (0,0.3) {};
\fill (-3.75,0) circle (5pt);
\fill (-1.25,0) circle (5pt);
\draw (-3.75,0) circle (10pt);
\fill (1.25,0) circle (5pt);
\fill (3.75,0) circle (5pt);
\draw (6.25,0) circle (10pt);
\fill (6.25,0) circle (5pt);
\draw (-3.6,0)--(-1.25,0);
\draw (-1.25,0)--(1.1,0);
\draw (3.75,-0)--(6.25,0);
\draw (1.4,0)--(3.6,0);
\fill (1.25,2) circle (5pt);
\draw (1.25,0)--(1.25,2);
\draw (8.75,0) circle (10pt);
\fill (8.75,0) circle (5pt);
\draw (6.25,0)--(8.75,0);
\end{tikzpicture} 
\\ [22pt] \hline
\raisebox{-.4cm}{\begin{tikzpicture}[scale=0.3]
\node at (0,1.5) {};
\fill (-3.75,0) circle (5pt);
\fill (-3.75,-2) circle (5pt);
\draw (-4.1,0) arc [start angle=180, end angle=0, radius = 10pt];
\draw (-4.1,-2) arc [start angle=180, end angle=360, radius = 10pt];
\draw (-4.1,0)--(-4.1,-2);
\draw (-3.4,0)--(-3.4,-2);
\fill (-1.25,0) circle (5pt);
\fill (-1.25,-2) circle (5pt);
\draw (-1.6,0) arc [start angle=180, end angle=0, radius = 10pt];
\draw (-1.6,-2) arc [start angle=180, end angle=360, radius = 10pt];
\draw (-1.6,0)--(-1.6,-2);
\draw (-0.9,0)--(-0.9,-2);
\draw (-3.75,0)--(-1.25,0);
\draw (-3.75,-2)--(-1.25,-2);
\draw (-1.25,0) arc [start angle=90, end angle=-90, radius=1cm];
\fill (-0.25,-1) circle (5pt);
\draw (-0.25,-1) circle (10pt);
\fill (2.25,-1) circle (5pt);
\draw (2.25,-1) circle (10pt);
\draw (-0.25,-1)--(2.25,-1);
\node at (0,-3) {};
\end{tikzpicture} }
&
\begin{tikzpicture}[scale=0.3]
\node at (0,2.8) {};
\fill (-3.75,0) circle (5pt);
\fill (-1.25,0) circle (5pt);
\draw (-1.25,0) circle (10pt);
\fill (1.25,0) circle (5pt);
\fill (3.75,0) circle (5pt);
\draw (3.75,0) circle (10pt);
\fill (6.25,0) circle (5pt);
\draw (-3.6,0)--(-1.25,0);
\draw (-1.25,0)--(1.1,0);
\draw (3.75,-0)--(6.25,0);
\draw (1.4,0)--(3.6,0);
\fill (3.75,2) circle (5pt);
\draw (6.25,0) circle (10pt);
\draw (3.75,0)--(3.75,2);
\draw (8.75,0) circle (10pt);
\fill (8.75,0) circle (5pt);
\draw (6.25,0)--(8.75,0);
\end{tikzpicture} 
&
\begin{tikzpicture}[scale=0.3]
\node at (0,0.3) {};
\fill (-3.75,0) circle (5pt);
\fill (-1.25,0) circle (5pt);
\draw (-3.75,0) circle (10pt);
\fill (1.25,0) circle (5pt);
\fill (3.75,0) circle (5pt);
\draw (6.25,0) circle (10pt);
\fill (6.25,0) circle (5pt);
\draw (-3.6,0)--(-1.25,0);
\draw (-1.25,0)--(1.1,0);
\draw (3.75,-0)--(6.25,0);
\draw (1.4,0)--(3.6,0);
\fill (1.25,2) circle (5pt);
\draw (1.25,0)--(1.25,2);
\draw (8.75,0) circle (10pt);
\fill (8.75,0) circle (5pt);
\draw (6.25,0)--(8.75,0);
\draw (11.25,0) circle (10pt);
\fill (11.25,0) circle (5pt);
\draw (11.25,0)--(8.75,0);
\end{tikzpicture} 
\\ [12pt] \hline
\end{tabular}
\end{center}
Now if we interpret the irreducible diagrams as fix or opposition diagrams, then we obtain the following table.

\begin{center}
\begin{tabular}{|c|c|c|}
\hline
& \phantom{.} \raisebox{-6pt}{$\mathsf{A_{5;2}}$} \phantom{.}& \phantom{.} \raisebox{-6pt}{$\mathsf{E_{6;2}}$} \phantom{.} \\ [.45cm]\hline
\raisebox{-6pt}{$\mathsf{^2A_{5;3}}$} & \raisebox{-6pt}{$\mathsf{D_{6;3}^2}$}& \raisebox{-6pt}{$\mathsf{E_{7;3}}$} \\ [.45cm] \hline
\raisebox{-6pt}{$\mathsf{^2E_{6;4}}$} & \raisebox{-6pt}{$\mathsf{E_{7;4}}$} & \raisebox{-6pt}{$\mathsf{E_{8,4}}$}\\ [.45cm] \hline \end{tabular}
\end{center}
We now observe that all the fix and opposition diagrams of \uniclass\ automorphisms of exceptional types $\sE_6,\sE_7,\sE_8$ appear in the table; moreover the fix diagram of the \uniclass\ automorphism belonging to the opposition diagram of any cell appears in the cell which lies symmetric with respect to the main diagonal. Hence cells lying symmetric with respect to the main diagonal are paired in our sense. On the one hand, this adds some magic to the square; on the other hand our results ``explain'' the magic why the sum of the relative ranks of cells lying symmetric with respect to the main diagonal add up to the absolute rank.  

The automorphisms themselves belong to the so-called \emph{delayed Magic Square}, see Section 9.2 in \cite{Mal:23}, except for the first row, interpreted as fix diagrams, where also Galois descent is allowed. Note that the table above also displays the two possible fix and opposition diagrams for nontrivial and non-anisotropic \uniclass\ automorphisms of buildings of type $\mathsf{A_5}$, unlike for type $\mathsf{D_6}$.   

With similar interpretation, the compact form of the Freudenthal--Tits Magic Square corresponds to anisotropic automorphisms (the fix diagrams are all empty---everything is mapped to an opposite), and the split form corresponds to the identity (the fix diagrams are all full---everything is fixed).  Hence for the exceptional types $\mathsf{E_6,E_7,E_8}$, each uniclass automorphism is encoded in one or the other  form of the Freudenthal--Tits Magic Square (and this also holds for type $\mathsf{A_5}$).

\end{document}